\documentclass[11pt]{article}
\usepackage{amssymb,amsmath,amsfonts,amsthm,array,bm,bbm,color,graphicx,hyperref,soul}

\hypersetup{colorlinks=true, linkcolor=red, anchorcolor=blue, citecolor=blue}
\providecommand{\U}[1]{\protect\rule{.1in}{.1in}}

\numberwithin{equation}{section}

\newtheorem{theorem}{Theorem}[section]
\newtheorem{lemma}[theorem]{Lemma}

\newtheorem{proposition}[theorem]{Proposition}
\newtheorem{remark}[theorem]{Remark}

\newtheorem{definition}[theorem]{Definition}

\def\<{\langle}
\def\>{\rangle}
\def \d{{\rm d}}

\def\M{\mathcal{M}}
\def\F{\mathcal{F}}
\def\E{\mathbb{E}}
\def\N{\mathbb{N}}
\def\P{\mathbb{P}}
\def\R{\mathbb{R}}
\def\T{\mathbb{T}}

\def\p{\partial}
\def\eps{\varepsilon}
\def\vphi{\varphi}
\def\law{\operatorname{Law}}

\def\div{{\rm div}}
\def\t{\widetilde}
\def\C{\mathcal{C}}

\definecolor{dark}{rgb}{0.0, 0.42, 0.0}

\title{Propagation of chaos for multi-species moderately interacting particle systems up to Newtonian singularity}
\author{José Antonio Carrillo, Shuchen Guo and Alexandra Holzinger\footnote{Email: carrillo@maths.ox.ac.uk, guo@maths.ox.ac.uk, alexandra.holzinger@maths.ox.ac.uk,  Address: Mathematical Institute, University of Oxford, Oxford OX2 6GG, UK.}}
\date{\today}

\begin{document}

\maketitle

\begin{abstract}

We derive a class of multi-species aggregation-diffusion systems from stochastic interacting particle systems via relative entropy method with quantitative bounds. We show an algebraic $L^1$-convergence result using moderately interacting particle systems approximating attractive/repulsive singular potentials up to Newtonian/Coulomb singularities without additional cut-off on the particle level. The first step is to make use of the relative entropy between the joint distribution of the particle system and an approximated limiting aggregation-diffusion system. A crucial argument in the proof is to show convergence in probability by a stopping time argument. The second step is to obtain a quantitative convergence rate to the limiting aggregation-diffusion system from the approximated PDE system. This is shown by evaluating a combination of relative entropy and $L^2$-distance.  
\end{abstract}

Aggregation-Diffusion Equations, Mean-Field Limit, Relative Entropy

%\tableofcontents

\section{Introduction}
The derivation of multi-species aggregation-diffusion systems from particle approximations is a fundamental question for modeling the interactions of large ensembles of "individuals" such as cells, electrons, ions or agents, with a large range of applications in mathematical biology, semiconductors, plasma physics and opinion dynamics. We here focus on underpinning the rigorous derivation of these equations and the challenging question of estimating the order of approximation in terms of the number of particles. These multi-species aggregation-diffusion systems on $\R^d$ read as
\begin{equation}\label{intro1}
\p_t \bar f_\alpha=\sum_{\beta=1}^{n}\div(\bar f_\alpha\nabla V_{\alpha\beta}\ast \bar f_\beta)+\sigma_{\alpha}\Delta \bar  f_\alpha,
\end{equation}
for fixed number of species $n\in \N$, and for indexes of species $\alpha,\beta=1,2,\ldots,n$.
Here, the attractive or repulsive, possibly as singular as Newtonian/Coulomb, potentials between the $\alpha$-th and $\beta$-th species are denoted by $V_{\alpha\beta}$. System \eqref{intro1} describes the evolution of the density of each subpopulation $\bar f_\alpha$ in a coupled system due to the interaction forces $\nabla V_{\alpha\beta}\ast \bar f_\beta$ between different subpopulations. An archetypic single-species example of aggregation-diffusion equations is the parabolic-elliptic Keller-Segel model for chemotaxis \cite{keller1970initiation}
$$
\p_t \bar f=\div (\bar f \nabla V\ast \bar f)+\sigma \Delta \bar f,
$$
where $V$ is the attractive Coulomb potential $V(x)=-1/|x|^{d-2}$ for dimension $d\geq 3$, $V(x)=\log |x|$ for $d=2$, and its generalisations to systems \cite{W02,SW05,KW19,KW22,carrillo2024well}. For more facets about aggregation-diffusion equations, we refer to recent reviews \cite{bailo2024aggregation,GC,CCY19} and the references therein.

In particular, system \eqref{intro1} can be regarded as a multi-species generalisation of the well-known Keller-Segel model for chemotaxis, however, there exists a wide range of applications for multi-species aggregation-diffusion equations with singular interaction kernels as considered in this article. Indeed, many natural and social phenomena arise from intracomponent and intercomponent interactions within multi-species systems. One classical example is the two-component plasma in statistical physics, where positively charged particles and negatively charged particles interact through attractive and repulsive Coulomb forces, see \cite{serfaty2024lectures} and the references therein. Similarly, the bipolar model (electrons and holes) in semiconductor theory exhibits dynamics analogous to those of a two-component plasma, \cite{MRS,jungel2009transport}. In addition, socio-economical dynamics such as spatial conflicts can be modeled by multi-species aggregation-diffusion systems, where individuals with the same opinion attract each other, while those with different opinions exhibit repulsion, for instance \cite{clifford1973model,GPY17}. These wide-ranging applications underline the importance of studying the microscopic particle derivation of \eqref{intro1}. 

From a physical point of view, each subpopulation can be considered to represent the average behaviour of a large number of interacting particles or individuals. In the classical mean-field theory, which goes back to the seminal works of \cite{kac1956foundations,mckean1967propagation}, particles are modeled by a system of coupled SDEs where the weight of interaction scales like $1/N$ with $N$ denoting the number of particles in the system. 
As $N$ grows, under  suitable assumptions on the initial data and the interaction kernel, stochastic particle systems show some average behaviour described by the nonlinear Fokker-Planck equations \eqref{intro1}. This convergence is often called the mean-field limit, more detailed discussion can be found in \cite{JW17}. Different from the mean-field scaling, the moderately interacting regime has been proposed by Oelschläger in \cite{O85} where he shows that it can be used to approximate porous-medium-like nonlinearity instead.  In this regime, each particle interacts with other $N^\vartheta$ particles with $0<\vartheta<1$, governed by a compactly supported potential $\chi^N$ which converges to Dirac-Delta function when $N$ goes to $\infty$, and the strength of interaction becomes $1/N^\vartheta$. More generally, convolving a singular potential with $\chi^N$ can be viewed as a generalisation of moderate interactions, which has attracted a lot of attention since Oelschläger, for example \cite{meleard1987propagation,jourdain1998convergence,S00,JM98,FP08} and more recently \cite{flandoli2020uniform,CDHJ21,GL23,olivera2023quantitative,olivera2024quantitative,chen2024fluctuations}. A further motivation to study this generalisation of moderate interactions is the challenging well-posedness of the stochastic interacting particle system, due to possible collisions, for singular attractive Newtonian/Coulomb potentials \cite{FJ17,FT24}. 

More precisely, the generalised moderate interaction regime considers that each species contains $N$ particles, and the microscopic model is given by the following SDE system on $\R^d$: 
\begin{equation}\label{intro2}
\,\d X_{\alpha,i}^{\eps}(t) = -\frac{1}{N}\sum_{\beta=1}^{n} \sum_{j=1}^N \nabla V_{\alpha\beta}^\eps(X_{\alpha,i}^{\eps}(t) - X_{\beta,j}^{\eps}(t)) \,\d t + \sqrt{2\sigma_{\alpha}} \,\d B_{\alpha,i}(t),
\end{equation}
for indexes of species $\alpha,\beta=1,2,\ldots,n$ and indexes of particles $i,j=1,2,\ldots,N$. Here $V_{\alpha\beta}^\eps = \chi^\eps* V_{\alpha\beta}^\eps$ are regularised potentials approximating the singular interaction kernels $V_{\alpha\beta}$ when $\eps$ goes to $0$. The regularisation parameter $\eps$ can be taken to have algebraic connection with number of particle $N$ as $\eps=N^{-\ell}$, which plays a crucial role in subsequent study of fluctuations. 

In this work, we close the gap of the derivation of the multi-species aggregation-diffusion systems \eqref{intro1}, up to Coulomb singularity of the interaction kernels $V_{\alpha\beta}$, from the stochastic interacting particle system \eqref{intro2} by proving a strong global-in-time propagation of chaos result in $L^1$-norm with an algebraic rate in $N$ for $d\geq 3$. In particular, our result also covers the particle derivation for the parabolic-elliptic Keller-Segel model in $d\geq 3$ under suitable assumptions on the initial data.

Previous works considered particle system \eqref{intro2} in the classical moderate scaling regime to obtain local cross-diffusion systems \cite{CDJ19}, results that were consequently extended to more general pressure in \cite{CG24}. Moreover, cross-diffusion systems of Shigesada–Kawasaki–Teramoto type were derived in \cite{CDHJ21} from moderately interacting particle systems.  We notice that, unlike our particle system \eqref{intro2}, these above results concerning moderate interactions for multi-species systems consider a scaling regime, where the regularisation parameter $\eps$ logarithmically depends on the particle number $N$, eventually leading to \textit{logarithmic convergence rates}. Here, we overcome this obstacle for Coulomb/Newtonian interaction kernels obtaining algebraic rates. Finally, we mention that in the framework of weak-interaction scaling, nonlocal Lotka-Volterra models have been derived in \cite{baladron2012mean,FM15} with stronger assumptions on the interaction kernels. 

A crucial argument in previous approaches in order to show the strong $L^1(\R^d)$ convergence in the generalised moderate regime for single species makes use of a Taylor expansion of the singular potential up to second order derivatives, see \cite{chen2024fluctuations}, which prevents to include Coulomb singularity even in case of $n=1$. In the present article, we can indeed overcome this difficulty by combining the stopping time approach of \cite{chen2024fluctuations}, which we extend to multi-species settings, with techniques developed in \cite{HLP20}, carefully balancing the local Lipschitz constant of the approximated singular kernels in terms of $N$, to finally include Coulomb singularities without additional cut-off on the particle level on $\R^d$. 

In order to clarify the main results of this work, we start by considering the solution of the Cauchy problem for the aggregation-diffusion system on $\R^d$ with $d\geq 3$ in \eqref{intro1} written as
\begin{equation}\label{aggregation-diffusion}
\begin{cases}
\displaystyle\p_t \bar f_\alpha=\sum_{\beta=1}^{n}\div\big(\bar f_\alpha\nabla V_{\alpha\beta}\ast \bar f_\beta\big)+\sigma_{\alpha}\Delta \bar  f_\alpha,\\ 
\bar f_\alpha(0)=\bar f_\alpha^0,
\end{cases}\alpha=1,2,\ldots,n,
\end{equation}
where the interaction potential is given by $V_{\alpha\beta}=a_{\alpha\beta}V$ with the following assumptions: 
\begin{enumerate}
    \item [(H1)] the constants $a_{\alpha\beta}\in \R$, especially $a_{\alpha\beta}>0$ and $a_{\alpha\beta}<0$ corresponding to repulsive and attractive regimes between $\alpha$-th and $\beta$-th species respectively; 
    \item [(H2)] the potential $V(x)=1/|x|^s$ and $0<s\leq d-2$, which covers sub-Newtonian/Coulomb and Newtonian/Coulomb interactions;
    \item [(H3)] the linear diffusion coefficients $\sigma_\alpha>0$;
    \item[(H4)] the initial condition $\bar f_\alpha^0\geq 0$ with  $\bar f_\alpha^0\in L^1\cap L^\infty(\R^d)$.
\end{enumerate} 
We will quantify the approximation of the solution to \eqref{aggregation-diffusion} by interacting particle systems of the form \eqref{intro2}. This solution can be guaranteed to exist uniquely under some smallness assumption \eqref{smallness assumption} of every initial data $\bar f_\alpha$. More precisely, there exists a unique global solution $\bar f=(\bar f_1,\ldots, \bar f_n)$ of \eqref{aggregation-diffusion} satisfying for $\alpha=1,2,\ldots,n$,
$$
\bar f_\alpha \in L^\infty(0,T;L^1\cap L^\infty(\R^d))\cap L^2(0,T;H^1(\R^d)), \text{ for any } T>0.
$$
This is shown in Theorem \ref{well-possedness} for completeness.

We now state our assumptions on the interacting particle system \eqref{intro2}. Let $\big(\Omega,\mathcal{F}, (\mathcal{F}_t)_{t\geq 0},\P\big)$ be a given filtered probability space; $B_R\subset \R^d$ denotes the bounded ball centered at the origin with radius $R>0$ while $B_R^c$ is its complement. For $\alpha=1,2,\ldots,n$ and $i=1,2,\ldots,N$,  the underlying generalised moderately interacting particle dynamics can be written as 
\begin{equation}\label{X}
\left\{\begin{array}{l}
\displaystyle\d X_{\alpha,i}^{\eps}(t) = - \frac{1}{N}\sum_{\beta=1}^{n} \sum_{j=1}^N \nabla V_{\alpha\beta}^\eps(X_{\alpha,i}^{\eps}(t) - X_{\beta,j}^{\eps}(t)) \,\d t + \sqrt{2\sigma_{\alpha}} \,\d B_{\alpha,i}(t),\\
    \displaystyle X_{\alpha,i}^{\eps}(0) = Z_{\alpha,i},
\end{array}\right. 
\end{equation}
where the following assumptions hold:
\begin{itemize}
\item [(H5)] the potential $V_{\alpha\beta}^\eps= a_{\alpha\beta}V\ast\chi^\eps $ with mollifier $\chi^\eps(x)=\eps^{-d}\chi(\eps^{-1}x)$, where $\chi\in C_c^\infty(\R^d)$ is a radially symmetric probability density with $\rm{supp}\chi\subset B_1$ and algebraic connections between $\eps$ and $N$ holds as $\eps=N^{-\ell}$ (the range of $\ell$ refers to \eqref{range l} in the main theorem);
\item[(H6)] $\big(Z_{\alpha,i}\big)_{i\geq 1}$ is a family of i.i.d random variables on $\R^d$ with the common distribution $\bar f_\alpha^0$;
\item[(H7)] $(B_{\alpha,i})_{i\geq 1,\alpha\geq 1}$ are i.i.d $d$-dimension $\F_t$-Brownian motions which are independent of $Z_{\alpha,i}$.
\end{itemize}

We remark that by classical theory particle system \eqref{X} is well-defined due to the smoothness of the interaction kernels. Moreover, we want to highlight that we only take convolution-type regularisations on the particle level with no need for additional truncation, which has been used in previous works especially for Coulomb and super-Coulomb singularities, \cite{li2023convergence,olivera2023quantitative}.  

By It\^{o}'s formula, the joint distribution of $nN$ particles $\big[(X_{\alpha,i}^\eps)_{i=1}^{N}\big]_{\alpha=1}^n$ satisfies the following Liouville equation (Kolmogorov forward equation) on $\R^{dnN}$ as
\begin{equation}\label{Liouville}
\left\{\begin{array}{l}\p_t f_{N,\eps}=\displaystyle\sum_{\alpha,\beta=1}^n\sum_{i=1}^N\div_{x_{\alpha,i}}\big(f_{N,\eps}\frac{1}{N}\sum_{j=1}^N\nabla V_{\alpha\beta}^\eps(x_{\alpha,i}-x_{\beta,j})\big)+\sum_{\alpha=1}^n\sum_{i=1}^N\sigma_{\alpha}\Delta_{x_{\alpha,i}}f_{N,\eps},\\
f_{N,\eps}(0)=\displaystyle\prod_{\alpha=1}^n(\bar f_\alpha^0)^{\otimes N},\end{array}\right.
\end{equation}
which is a linear parabolic equation with smooth coefficients for any given $\eps>0$ and $N$. Hence one can obtain existence and uniqueness of
weak solutions in $L^\infty(0,T,L^1\cap L^\infty(\R^{dnN}))$ for any $T>0$ with suitable assumptions on the initial data by classical theory. We notice that, unlike the single species, the solution $f_{N,\eps}$ is not fully symmetric. As a result, we shall define the marginal distribution of multi-species particle systems, where particles are considered to be identical in each species, and we can select any number of particles from each species. 
\begin{definition}\label{marginal}
For any $n$-tuples $\mathbf{K}=(K_1,\ldots,K_n)$ such that   $K_\alpha\in\N$ for $\alpha=1,\ldots,n$, with $|\mathbf{K}|:=\sum_{\alpha=1}^nK_\alpha$, the $n$-species $\mathbf{K}$-th marginal distribution of \eqref{Liouville} is given by $$f_{N,\eps}^{(\mathbf{K})}:=\int_{\R^{(d\sum_{\alpha=1}^n(N-K_\alpha))}}f_{N,\eps}%(x_{1,1},\ldots,x_{1,N},\ldots,x_{n,1},\ldots,x_{n,N} )
\,\d x_{1,K_1+1}\cdots\,\d x_{1,N}\cdots \,\d x_{n,K_n+1}\cdots\,\d x_{n,N},$$
which describes the joint distribution of particles
$$
X_{1,1}^\eps,\ldots, X_{1,K_1}^\eps; \ldots; X_{\alpha,1}^\eps,\ldots, X_{\alpha,K_\alpha}^\eps; \ldots; X_{n,1}^\eps,\ldots, X_{n,K_n}^\eps,
$$
where we have $|\mathbf{K}|$ particles in total and $K_\alpha$ particles in each subpolulation.
\end{definition}

Similar to many related works concerning moderately interacting particle systems, for example \cite{FP08,chen2024fluctuations,GL23}, we will introduce the intermediate PDE system $\t f_{\alpha,\eps}$ $(\alpha=1,\ldots,n)$ as well as its corresponding McKean-Vlasov SDE. The PDE and SDE at the intermediate level play crucial roles  in particle approximations. 
The intermediate PDE system reads as
\begin{equation}\label{intermediate}
    \begin{cases}
    \displaystyle\p_t \t f_{\alpha,\eps}  = \sum_{\beta =1}^n\div\big(\t f_{\alpha,\eps} \nabla V_{\alpha\beta}^\eps \ast \t f_{\beta,\eps}\big) +\sigma_{\alpha} \Delta \t f_{\alpha,\eps}, \\
      \t f_{\alpha,\eps}(0)=\bar f_\alpha^0, %\in L^1\cap L^\infty(\R^d),
    \end{cases}\alpha=1,2,\ldots,n,
\end{equation} 
where  the regularised potential $V^\eps_{\alpha\beta}$ is the same as in \eqref{X} which satisfies the corresponding assumptions,  and the initial data coincides with that in \eqref{aggregation-diffusion}. 
The above intermediate PDE system coincides with the Fokker-Planck equation of the following intermediate SDE: %due to It\^{o}'s formula
\begin{equation}\label{tildeX}
\begin{cases}
    \displaystyle\,\d \t X_{\alpha}^{\eps}(t) = -\sum_{\beta=1}^{n}  \nabla V_{\alpha\beta}^\eps\ast\t f_{\beta,\eps}(t,\t X_{\alpha}^{\eps}(t)) \,\d t + \sqrt{2\sigma_{\alpha}} \,\d \t B_{\alpha}(t),\\
\law(\t X_{\alpha}^{\eps}(t))=\t f_{\alpha,\eps}
    \vspace*{0.2cm}\\
    \displaystyle\t X_{\alpha}^{\eps}(0) = Z_{\alpha},
\end{cases}\quad \alpha=1,2,\ldots,n
\end{equation}
where $Z_{\alpha}$ is a random variables on $\R^d$ with the distribution $\bar f_\alpha^0$, the Brownian motion $\t B_{\alpha}$ is $d$-dimension $\F_t$-Brownian motions that is independent of $Z_{\alpha}$. For any fixed $\eps>0$, the existence and uniqueness of strong solution of the SDEs system \eqref{tildeX} follow by standard theory \cite[Theorem 3.1.1]{PR07}, which implies there exists a measure-valued solution of intermediate PDEs system \eqref{intermediate}. Remark that throughout this paper $\bar{\cdot}$ denotes quantities at limiting level (parameter $\eps$ independent); $\t{\cdot}$ denotes quantities at intermediate level (parameter $\eps$ dependent). 

We also denote $\t{f}_{N,\eps}$ as the tensorised solution of the intermediate PDE system \eqref{intermediate} as follows $$
\t{f}_{N,\eps}=\t{f}_{\eps,1}^{\otimes N}\otimes\cdots\otimes \t{f}_{\eps,n}^{\otimes N}.
$$
It is easy to see $\t{f}_{N,\eps}$ satisfies the following equation on $\R^{dnN}$
\begin{equation}\label{tensor intermediate}
\p_t \t{f}_{N,\eps}=\sum_{\alpha,\beta=1}^n\sum_{i=1}^N\div_{x_{\alpha,i}}\big(\t{f}_{N,\eps}\nabla V_{\alpha\beta}^\eps\ast \t{f}_{\beta,\eps}(x_{\alpha,i})\big)+\sum_{\alpha=1}^n\sum_{i=1}^N\sigma_{\alpha}\Delta_{x_{\alpha,i}}\t{f}_{N,\eps}.
\end{equation}

We are now able to state our main quantitative propagation of chaos result for the multi-species system, which requires certain regularity assumptions of the limiting PDE \eqref{aggregation-diffusion} and the intermediate PDE \eqref{intermediate}.

\begin{theorem}\label{POC} 
Under assumptions (H1)-(H7), let $\bar f=(\bar f_1,\ldots,\bar f_n)$ be the solution of aggregation-diffusion system \eqref{aggregation-diffusion} satisfying $\bar f_\alpha\in L^\infty(0,T;L^1\cap L^\infty(\R^d))\cap L^2(0,T;H^1(\R^d))$ for $\alpha=1,\ldots,n$. And let $\t f_{\eps}=(\t f_{1,\eps},\ldots,\t f_{n,\eps})$ be the solution of the intermediate PDE system \eqref{intermediate} satisfying the uniform-in-$\eps$ bound \begin{equation}\label{uniform in eps}
\max_{\alpha=1,\ldots,n}\sup_\eps\|\t f_{\alpha,\eps}\|_{L^\infty(0,T;L^1\cap L^\infty)\cap L^2(0,T;H^1)}<C(T).
\end{equation} We further assume that the parameter $\eps$ has algebraic connection with $N$ as $\eps=N^{-\ell}$ with the range 
\begin{equation}\label{range l}
\left\{\begin{array}{cc}
    &\displaystyle0<\ell<\frac{1} {C_0},\quad\text{when } s=d-2,\\&
 \\&\displaystyle0<\ell<\frac{1}{2s+4},\quad\text{when } s<d-2,
\end{array}\right.
\end{equation}
where  $C_0$  is a constant that $$C_0=C_0(T,d,n,\max_{\alpha,\beta}|a_{\alpha\beta}|, \|\chi\|_{W^{2,1}\cap W^{2,\infty}},\max_\alpha\sup_\eps\|\t f_{\eps,\alpha}\|_{L^\infty(0,T;L^1\cap L^\infty)}).$$

Then the relative entropy between any $\mathbf{K}$-th marginal distribution defined in Definition \ref{marginal} and tensorised solution of \eqref{aggregation-diffusion} has the following bound:
\begin{equation}\label{stronger POC}
\sup_{t\in[0,T]}\big\|f_{N,\eps}^{(\mathbf{K})}-\prod_{\alpha=1}^n\bar f_\alpha^{\otimes K_\alpha}\big\|_{L^1(\R^{d|\mathbf{K}|})}\leq \frac{C(T)}{N^{\zeta}}, \quad 
\alpha=1,2,\ldots,n,
\end{equation}
where  the constant $C$ depends on $\max_\alpha K_\alpha$ while does not depend on  $N$, and the parameter is given as $\zeta=\min\{\ell,\frac{1}{2}-\ell(s+2)-\varrho\}$ for some $\varrho>0$ arbitrarily  small.
\end{theorem}

\begin{remark}\
\begin{itemize}
    \item[i)] The existence and uniqueness of such the solution of \eqref{aggregation-diffusion} in $L^\infty(0,T;L^1\cap L^\infty)\cap L^2(0,T;H^1)$ is guaranteed by Theorem \ref{well-possedness} under small initial data \eqref{smallness assumption} for $d\geq 3$. Under the same smallness condition, uniform-in-$\eps$ bound  \eqref{uniform in eps} will be shown in Lemma \ref{uniformeps} as an intermediate step to prove the well-posedness of the limiting equation.  However, we will not use the smallness assumption directly in the proof of Theorem \ref{POC}. 
    
\item[ii)]  The Coulomb case in two dimensions, i.e., $V(x)=\log|x|$ is not covered in the above theorem. We will further comment on the two dimensional Coulomb case in Remarks \ref{gradient2d}, \ref{remark:logPDE} and \ref{lp2d}. 

\item[iii)]  We also notice that the above convergence rate can be improved. In the sub-Coulomb regime $(s<d-2)$, or Coulomb regime $(s=d-2)$ but with higher regularity of the initial data in $(H4)$ as $\bar f_\alpha^0\in W^{1,1}\cap W^{1,\infty}(\R^d),$ the  power $\zeta$ in convergence rate can be larger as
 $\zeta=\min\{\ell, 1/2-\ell(s+1)-\varrho\}$ for some $\varrho>0$ arbitrarily small. For more details, we refer to Remark \ref{higher regularity} and \ref{higher regularity 2} in Section \ref{mean-field limit}.
 \end{itemize}
\end{remark}

Our main strong $L^1$-propagation of chaos estimate \eqref{stronger POC} in Theorem \ref{POC} is obtained in a two-step procedure using the relative entropy in both parts of the proof. First, we exploit the intermediate system by evaluating the relative entropy between solutions of the Liouville equations \eqref{Liouville} and the (tensorised) intermediate PDE system \eqref{intermediate}. Consequently, as a second step in the proof, we then estimate the relative entropy between the solutions of intermediate PDE system \eqref{intermediate} and the limiting aggregation-diffusion equation \eqref{aggregation-diffusion}, in which we exploit a combination between relative entropy and $L^2$-norm estimates. We refer to these two steps as \textit{mean-field estimate} and \textit{PDE error estimate} hereafter.

The relative entropy method has been successfully used to rigorously prove mean-field limits for many different models, see for example \cite{JW17,JW18,carrillo2024relative}. In addition, in \cite{chen2023quantitative}, a combination between the relative entropy and the regularised $L^2$-estimate by Oelschläger \cite{oelschlager1987fluctuation} has been used to prove a propagation of chaos result for the viscous porous medium equation from a moderately interacting particle system. Inspired by the approach of \cite{chen2023quantitative} for single species, in this article we show convergence of the particle system in relative entropy for multi-species systems \eqref{aggregation-diffusion} by proving a convergence in probability result with arbitrary algebraic rate. 

Finally, we want to compare our result to related works in the literature. The parabolic-elliptic Keller-Segel model with sub-critical mass on the torus $\T^2$ is derived from particle systems with singular kernel in \cite{bresch2023mean,de2023attractive} via the modulated free energy method, which can be seen as a combination of the relative entropy method of \cite{JW18} and the modulated energy method of \cite{serfaty2020mean}. The weak convergence of the empirical measure for the critical mass case on the whole plane $\R^2$ is given by \cite{tardy2022convergence}. In the moderate interacting regime, propagation of chaos for the Keller-Segel model with logarithmic cut-offs are shown in \cite{chen2023rigorous,liu2019propagation}. As mentioned before, our main ingredient to show convergence in relative entropy of the stochastic particle system is a quantitative estimate in probability, which has been previously considered for bounded kernels in \cite{chen2024well}. A truncation argument in the moderate regime for both attractive and repulsive Riesz-type kernel on $\R^d$ with Coulomb and higher singularities has been used in \cite{olivera2023quantitative} to prove a quantitative convergence of regularised empirical measure to the solution of the PDE, which implies qualitative propagation of chaos. Moreover, the convergence of the regularised empirical measure towards a Keller-Segel model has been obtained in  \cite{olivera2024quantitative} by using moderately interacting particle systems without truncation on $\T^d$. Recently, this work has been gerneralised to second order systems with Besov-type interaction, see \cite{hao2024propagation}. The convergence in probability towards the single-species regularised Keller-Segel model has been considered in \cite{huang2019learning}. Under the assumption of this type of convergence, the relative entropy bound between the regularised particle system and the regularised Keller-Segel model is obtained in \cite{nikolaev2024quantitative}.  

Comparing the aforementioned results with our present work, by using a moderately interacting particle system, we are able to prove a multi-species quantitative propagation of chaos result with attractive and repulsive Coulomb-type interaction kernels, which is indeed algebraic in $N$. This can be achieved without additional truncation on the microscopic level. The algebraic nature of our result can serve as a starting point for analysing the corresponding fluctuation behaviour of the microscopic particle system, which we leave for future work. 

The article is organised as follows. Section \ref{proof} provides an overview of our main ideas and outline of the proof of Theorem \ref{POC}. In particular, we state in this section the precise mean-field estimates and the PDE error estimates, presented as Proposition \ref{RE convergence} and Proposition \ref{second step}, respectively. This section also includes the global well-posedness result stated in Theorem \ref{well-possedness}. In Section \ref{proof convprob}, we prove a convergence in probability result, which is a key component of our analysis. 
This is followed by Section \ref{mean-field limit}, where we prove our mean-field estimate in relative entropy  (Proposition \ref{RE convergence}). The final two sections are concerned with the PDE analysis of system \eqref{intro1}: Section \ref{PDE distance} derives the PDE error estimate (Proposition \ref{second step}) between the intermediate and the limiting system, while Section \ref{pde analysis} provides a detailed proof of the global well-posednes result in Theorem \ref{well-possedness}.

%%%%%%%%%%%%%%%%%%%%%%

\section{Strategy of the proofs}\label{proof}

Our strategy is to combine the mean-field limit from the particle system \eqref{X} to an intermediate PDE \eqref{intermediate}, and a PDE error estimate between the intermediate PDE \eqref{intermediate} and the limiting PDE \eqref{aggregation-diffusion}, where we refer to Proposition \ref{RE convergence} and Proposition \ref{second step} respectively. The main result (Theorem \ref{POC}) follows from these two propositions.  %The two key  steps are Proposition . 
We will state these results in the sense of the relative entropy defined as follows.

\begin{definition}[Relative entropy]\label{def RE}
For any two probability density functions $\mu$ and $\nu$ on an arbitrary dimensional Euclidean space $E$, which are absolutely continuous to the Lebesgue measure, the (unrenormalised) relative entropy reads as
$$
H(\mu|\nu)=\int_{E}\mu\log \frac{\mu}{\nu}
$$
where the integral is with respect to the Lebesgue measure on $E$.  
\end{definition}

The relative entropy is a nonnegative quantity which controls the square of the $L^1$-distance by the Csiszár-Kullback-Pinsker inequality (see for instance \cite{villani2021topics}):
\begin{equation}\label{CKP}
\|\mu-\nu\|_{L^1(E)}\leq \sqrt{2H(\mu|\nu)}.    
\end{equation}

Recall Definition \ref{marginal} of multi-marginal distribution. We have the following lemma. The proof follows from \cite[Lemma 3.9]{DM01}, and we will sketch it in Section \ref{appendix sub}. 

\begin{lemma}\label{subadditivity}
The relative entropy defined on $\R^{dnN}$ between $f_{N,\eps}$ and $\t f_{N,\eps}$ controls the relative entropy of their multi-index marginals on $\R^{d|\mathbf{K}|}$ in the following way, 
\begin{equation}\label{subadd ineq}
H(f_{N,\eps}^{(\mathbf{K})} |\prod_{\alpha=1}^n\t f_{\alpha,\eps}^{\otimes K_\alpha})\leq \frac{\max_\alpha K_\alpha}{N}
H(f_{N,\eps}|\t f_{N,\eps}).
\end{equation}

\end{lemma}

From now on, let $\bar f=(\bar f_1,\ldots,\bar f_n)$ be the solution of the aggregation-diffusion system \eqref{aggregation-diffusion} satisfying $\bar f_\alpha\in L^\infty(0,T;L^1\cap L^\infty(\R^d))\cap L^2(0,T;H^1(\R^d))$ for $\alpha=1,\ldots,n$, and let $\t f_{\eps}=(\t f_{1,\eps},\ldots,\t f_{n,\eps})$ be the solution of the intermediate PDE system \eqref{intermediate} satisfying the uniform-in-$\eps$ bound  \eqref{uniform in eps}.

The first proposition (Proposition \ref{RE convergence}) shows a mean-field type estimate between the joint distribution of the particle system \eqref{X}, which is the  solution of the Liouville equation \eqref{Liouville}, and the solution of the tensorised PDE at the intermediate level \eqref{tensor intermediate}. The main estimate \eqref{H bound} in this proposition is in the sense of the relative entropy defined in Definition \ref{def RE}.

\begin{proposition}[Mean-field estimate]\label{RE convergence}
Under the same assumptions as in Theorem \ref{POC},
the relative entropy between the solution of   \eqref{Liouville} and the solution of  \eqref{tensor intermediate} can be controlled in such way that
\begin{equation}\label{H bound}
\sup_{t\in[0,T]}\frac{H\big(f_{N,\eps}(t)|\t{f}_{N,\eps}(t)\big)}{N}\leq \frac{C(T)}{N^{1-\ell(2s+4)-\varrho}},
\end{equation}
 for some $\varrho>0$ arbitrarily  small. In particular, it holds for the marginal distribution and the tensorised solution of the intermediate PDE \eqref{intermediate} that \begin{equation}\label{particle error}
\sup_{t\in[0,T]}\Big\|f_{N,\eps}^{(\mathbf{K})}-\prod_{\alpha=1}^n\t f_{\alpha,\eps}^{\otimes K_\alpha}\Big\|_{L^1(\R^{d|\mathbf{K}|})}\leq \frac{C(T)}{N^{1/2-\ell(s+2)-\varrho/2}},
\end{equation}
where the constant $C$ also depends on $\max_\alpha K_\alpha$. 
\end{proposition}
We remark that estimate \eqref{particle error} can be seen as a direct sequence of Lemma \ref{subadditivity}, Csiszár-Kullback-Pinsker inequality \eqref{CKP} and estimate \eqref{H bound}. The proof of Proposition \ref{RE convergence}, especially the bound \eqref{H bound}, will be given in Section \ref{mean-field limit}.

Evaluating the relative entropy for interacting particle systems  can be seen in many previous works, for instance \cite{JW17,JW18,chen2023quantitative, carrillo2024relative}. 
In our case we do it for multi-species systems in order to derive the following bound
\begin{equation}\label{relative entropy bound}
\frac{\,\d}{\,\d t}\frac{H(f_{N,\eps}|\t{f}_{N,\eps})}{N}\lesssim\E\bigg[\frac{1}{N}\sum_{\alpha,\beta=1}^n\sum_{i=1}^N\Big|\nabla V_{\alpha\beta}^\eps\ast \t{f}_{\beta,\eps}(X^\eps_{\alpha,i})-\frac{1}{N}\sum_{j=1}^N\nabla V_{\alpha\beta}^\eps(X^\eps_{\alpha,i}-X^\eps_{\beta,j})\Big|^2\bigg],
\end{equation}
which shows clearly how the cross-interactions influence the structure of the relative entropy estimates compared to \cite{JW18,chen2023quantitative}. 

In order to estimate the expectation on the right-hand side of \eqref{relative entropy bound}, we will use $N$-copies version of \eqref{tildeX} in Section \ref{mean-field limit}. More precisely, we construct the following intermediate SDE with the same initial data and Brownian motion as in \eqref{X} satisfying assumption (H6) and (H7): 
\begin{equation}\label{tildeXi}
 \begin{cases}
    \displaystyle\,\d \t X_{\alpha,i}^{\eps}(t) = -\sum_{\beta=1}^{n} \nabla V_{\alpha\beta}^\eps\ast \t f_{\beta,\eps}(t, \t X_{\alpha,i}^{\eps}(t)) \,\d t + \sqrt{2\sigma_{\alpha}} \,\d B_{\alpha,i}(t)\\
    \law(\t X_{\alpha,i}^\eps(t))=\t f_{\alpha,\eps}\vspace*{0.2cm}\\
    \displaystyle X_{\alpha,i}^{\eps}(0) = Z_{\alpha,i},
\end{cases}  
\end{equation}
for $i=1,2,\ldots,N$ and $\alpha=1,2,\ldots,n$.
Notice that $\t X^\eps_{\alpha,i}$ and $\t X^\eps_{\beta,j}$ are independent if either $i\neq j$ or $\alpha\neq \beta$.

We next plug in terms into \eqref{relative entropy bound} concerning the intermediate SDE \eqref{tildeXi} satisfied by $\t X_{\alpha,i}^\eps$, which then can be bounded by three terms as follows:
$$
\begin{aligned}
& \E\bigg[\frac{1}{N}\sum_{\alpha,\beta=1}^n\sum_{i=1}^N\big|\nabla V_{\alpha\beta}^\eps\ast \t{f}_{\beta,\eps}\big(X_{\alpha,i}^\eps\big)-\frac{1}{N}\sum_{j=1}^N\nabla V_{\alpha\beta}^\eps\big(X_{\alpha,i}^\eps-X_{\beta,j}^\eps\big)\big|^2\bigg]\\
\lesssim&\, \E\bigg[\frac{1}{N}\sum_{\alpha,\beta=1}^n\sum_{i=1}^N\big|\nabla V_{\alpha\beta}^\eps\ast \t{f}_{\beta,\eps}\big(X_{\alpha,i}^\eps\big)-\nabla V_{\alpha\beta}^\eps\ast \t{f}_{\beta,\eps}\big(\t{X}_{\alpha,i}^\eps\big)\big|^2\bigg]\\
&+ \E\bigg[\frac{1}{N}\sum_{\alpha,\beta=1}^n\sum_{i=1}^N\big|\nabla V_{\alpha\beta}^\eps\ast \t{f}_{\beta,\eps}\big(\t{X}_{\alpha,i}^\eps\big)-\frac{1}{N}\sum_{j=1}^N\nabla V_{\alpha\beta}^\eps\big(\t{X}_{\alpha,i}^\eps-\t{X}_{\beta,j}^\eps(t)\big)\big|^2\bigg]\\
&+ \E\bigg[\frac{1}{N}\sum_{\alpha,\beta=1}^n\sum_{i=1}^N\big|\frac{1}{N}\sum_{j=1}^N\nabla V_{\alpha\beta}^\eps\big(\t{X}_{\alpha,i}^\eps-\t{X}_{\beta,j}^\eps(t)\big)-\frac{1}{N}\sum_{j=1}^N\nabla V_{\alpha\beta}^\eps\big(X_{\alpha,i}^\eps-X_{\beta,j}^\eps\big)\big|^2\bigg]\\
=:& J_1+J_2+J_3.
\end{aligned}
$$
Term $J_2$ only depends on the intermediate SDE \eqref{tildeXi}, and can be controlled by the law of large numbers Lemma \ref{LLN}, which is a generalisation of  the law of large numbers estimate in \cite{chen2024fluctuations} to the multi-species case. 
To deal with $J_1$ and $J_3$, we will show the following quantitative error estimate in probability between the particle system \eqref{X} satisfied by  $X_{\alpha,i}^\eps$ and the intermediate SDE \eqref{tildeXi} satisfied by $\t X_{\alpha,i}^\eps$, for some suitable $\lambda$ and $\gamma$
\begin{equation}\label{convergence in probability}
\sup_{t\in[0,T]}\P\big(\max_{\alpha=1,\ldots,n}\max_{i=1,\ldots,N}|\t{X}_{\alpha,i}^\eps(t)-X_{\alpha,i}^\eps(t)\big|\geq N^{-\lambda}\big)\leq C(T,\gamma) N^{-\gamma}.
\end{equation}
In order to prove this convergence in probability (Proposition \ref{convergence in prob}), we follow  \cite{chen2024fluctuations} and generalise it to the multi-species case. Moreover, a careful adaptation has to be made to allow attractive Coulomb  interaction kernels (i.e. $s=d-2$) without further cut-off. The technique for Coulomb potentials, which is inspired by papers \cite{huang2019learning,HLP20},  is to construct an integrable auxiliary function as
$$
K^\eps(x)=\begin{cases}\displaystyle
    \frac{1}{|x|^{ {s+2}}}\quad |x|\geq 4\eps, \vspace{0.1cm}\\
    \displaystyle\frac{1}{(4\eps)^{s+2}}\quad |x|< 4\eps,
    \end{cases}
$$in order to avoid using Taylor's expansion directly. But unlike \cite{huang2019learning,HLP20}, the convergence in probability \eqref{convergence in probability} will be proved by a quite different approach,  namely the stopping-time argument developed in \cite{chen2024fluctuations}. More details can be found in Section \ref{proof convprob}. Besides, by introducing this auxiliary function, we can simplify some proofs in \cite{chen2024fluctuations}, and deal with sub-Coulomb and Coulomb potential in a unified way. %which we will specify later in Section \ref{proof convprob}; 
Thus, by combining the estimates of $J_1$, $J_2$ and $J_3$, the evolution of the relative entropy between the particle system \eqref{X} and the intermediate PDE \eqref{intermediate} can be controlled. 

The second proposition (Proposition \ref{second step}  below) states a PDE error estimate in terms of the relative entropy and $L^2$-distance between the intermediate PDE system \eqref{intermediate} and the limiting PDE system \eqref{aggregation-diffusion}.
\begin{proposition}[PDE error estimate]\label{second step}
Under the same assumptions as in Theorem \ref{POC}, the relative entropy between the solution of \eqref{aggregation-diffusion} and \eqref{intermediate} can be estimated such that for any $\eps>0$,
\begin{equation}\label{PDE error}\sup_{t\in[0,T]}\Big(\|\t f_{\alpha,\eps}-\bar f_{\alpha}\|_{L^2}^2+H(\t f_{\alpha,\eps}|\bar f_{\alpha})\Big)\leq C\eps^2,\quad 
\alpha=1,2,\ldots,n,
\end{equation}
where the constant $C$ does not depend on $\eps$. 
In particular, it holds that
\begin{equation}\label{tensor PDE error}
\sup_{t\in[0,T]}\Big\|\prod_{\alpha=1}^n\t f_{\alpha,\eps}^{\otimes K_\alpha}-\prod_{\alpha=1}^n\bar f_{\alpha}^{\otimes K_\alpha}\Big\|_{L^1(\R^{d|\mathbf{K}|})}\leq C\eps,\quad 
\alpha=1,2,\ldots,n,
\end{equation}
where the constant $C$ depends on $|\mathbf{K}|$.
\end{proposition}
Estimate \eqref{tensor PDE error} can be seen as a lifted version of estimate \eqref{PDE error} between the tensorised solutions of $\t f_\eps$ and $\bar f$ on $\R^{d|\mathbf{K}|}$, which can be obtained by   Csiszár-Kullback-Pinsker inequality \eqref{CKP}. The main idea of showing \eqref{PDE error} is to combine the evolution of the $L^2$ distance between $\t f_\eps$ and $\bar f$ with the evolution of the relative entropy at the PDE level, which is needed to close the relative entropy estimate. By combining these two distances, we can keep the assumptions on the initial conditions lower than in previous works (see for instance \cite{chen2023quantitative,chen2024fluctuations}).

The proof of Theorem \ref{POC} follows by combining  Proposition \ref{RE convergence} and Proposition \ref{second step}. Notice that our approach allows us to get an algebraic instead of a logarithmic connection between the  regularisation parameter $\eps$ and the number of particles $N$, which eventually gives us the algebraic convergence rate towards the system of PDEs with singular kernels \eqref{aggregation-diffusion}.

%between the joint distribution \eqref{Liouville} and the tensorised solution of intermediate equations \eqref{intermediate}.

%To close this section, we present a $L^\infty$-bound estimate which will appear several times later, and we state it separately as a remark.

Finally, for the completeness of our analysis,  we establish the global-in-time well-posedness of the aggregation-diffusion system \eqref{aggregation-diffusion}, as expected, under smallness conditions on the initial data \eqref{smallness assumption}, see Theorem \ref{well-possedness} below. It shows that the required PDE conditions in Theorem \ref{POC} can be fulfilled under some sufficient assumptions on the initial data. 

\begin{theorem}[Global well-posedness of \eqref{aggregation-diffusion}]\label{well-possedness}
Let assumptions $(H1)$-$(H4)$ hold, and assume the following smallness condition on the initial data such that
\begin{equation}\label{smallness assumption}
\sum_{\beta=1}^{n} |a_{\alpha\beta}|\|\bar f_\beta^0\|_{L^{d+1}}^{\frac{2s(d+1)}{d^2}}\leq \frac{4\sigma_\alpha^2}{(d+1)^2C_{HLS}^2C_{GNS}^2\sum_{\beta=1}^n |a_{\alpha\beta}|},\quad \forall \alpha=1,2,\ldots n,  
\end{equation}
where the constants $C_{HLS}$ and $C_{GNS}$ come from the Hardy-Littlewood-Sobolev inequality and the Gagliardo–Nirenberg-Sobolev inequality. Then there exists a unique weak solution $\bar f=(\bar f_1,\ldots,\bar f_n)$ with
\begin{equation}\label{regularity}
\bar f_\alpha \in L^\infty(0,T;L^1\cap L^\infty(\R^d))\cap L^2(0,T;H^1(\R^d)),\quad 
\alpha=1,2,\ldots,n,
\end{equation} which satisfies \eqref{aggregation-diffusion} in the weak sense, i.e., for any $\vphi\in C_b^2(\R^d)$ and any $T>0$, 
\begin{equation}\label{distributional}
\begin{aligned}
\int_{\R^d}\bar f_{\alpha}(T)&\vphi\,\d x  =  \int_{\R^d}\bar f_{\alpha}^0\vphi\,\d x+\sigma_{\alpha}\int_0^T\int_{\R^d}\Delta\vphi  \bar f_{\alpha} \,\d x\,\d t \\&- \sum_{\beta =1}^n a_{\alpha\beta} \int_0^T\int_{\R^d}\nabla\vphi\cdot\bar f_{\alpha} \big( \nabla V \ast \bar f_{\beta}\big)\,\d x\,\d t,\quad 
\alpha=1,2,\ldots,n.   
\end{aligned}
\end{equation}

\end{theorem}

The proof of Theorem \ref{well-possedness} takes advantage of  the intermediate PDE \eqref{intermediate}, which can be seen as a regularised aggregation-diffusion system. The intermediate PDE satisfies the uniform-in-$\eps$ estimate \eqref{uniform in eps} under the smallness condition as shown in Lemma \ref{uniformeps}. By sending the regularisation parameter $\eps=\eps(N)$ to $0$ when $N\to\infty$, it converges to the original aggregation-diffusion systems in $L^1(0,T;L^1(\R^d))$, which will be shown in details in Subsection \ref{existence subsection}. The uniqueness result can be obtained in a similar way as the proof of Proposition \ref{second step}, where we combine the relative entropy and the $L^2$-distance. 

To conclude this section, we want to remind the reader that the rigorous quantitative propagation of chaos result (Theorem \ref{POC}) holds
as long as the solution $\bar f$ of the  aggregation-diffusion system \eqref{aggregation-diffusion} and the solution $\t f_\eps$ of the intermediate PDE \eqref{intermediate} both lie in $L^\infty(0,T;L^1\cap L^\infty(\R^{dn}))\cap L^2(0,T;H^1(\R^{dn}))$  with the uniform-in-$\eps$ bound \eqref{uniform in eps}. These conditions can be achieved thanks to Theorem \ref{well-possedness}.

\section{Proof of convergence in probability}\label{proof convprob} In this section, we will prove the convergence in probability \eqref{convergence in probability} as mentioned in Section \ref{proof}.  Define a subset of probability space $\Omega$ such as, for some $\lambda>0$, 
\begin{equation}\label{clam}
\C_{\lambda}(t)=\{\omega\in \Omega :\, \max_{\alpha=1,\ldots,n}\max_{i=1,\ldots,N}|\t{X}_{\alpha,i}^\eps(t)-X_{\alpha,i}^\eps(t)\big|\geq N^{-\lambda}\}.    
\end{equation}
We have the following proposition, which shows that for suitable $\lambda$ the probability of $C_\lambda$ is arbitrarily small. In other words, the probability of the extreme event  $|\t{X}_{\alpha,i}^\eps(t)-X_{\alpha,i}^\eps(t)\big|\geq N^{-\lambda}$ is small enough. 

\begin{proposition}\label{convergence in prob}
Let $X_{\alpha,i}$ and $\t{X}_{\alpha,i}$ be strong solutions of \eqref{X} and \eqref{tildeXi}  respectively up to any time $T>0$. Under the assumptions of Theorem \ref{POC}, recall that $\ell$ satisfies the range \eqref{range l} namely
$$
\left\{\begin{array}{cc}
    &\displaystyle0<\ell<\min\left(\frac{1} {C_0},\frac1{2d}\right),\quad\text{when } s=d-2,\\&
 \\&\displaystyle0<\ell<\frac{1}{2s+4},\quad\text{when } s<d-2,
\end{array}\right.
$$where  $C_0$  is a constant such that $$C_0=C_0(T,d,n,\max_{\alpha,\beta}|a_{\alpha\beta}|, \|\chi\|_{W^{2,1}\cap W^{2,\infty}},\max_\alpha\sup_\eps\|\t f_{\eps,\alpha}\|_{L^\infty(0,T;L^1\cap L^\infty)}).$$
For some $\lambda$ satisfying  \begin{equation}\label{range lambda}\ell< \lambda<\frac{1}{2}-\ell (s+1),\end{equation} then it holds for any $\gamma>0$ that    $$
\sup_{t\in[0,T]}\P(\C_{\lambda}(t))=\sup_{t\in[0,T]}\P\big(\max_{\alpha=1,\ldots,n}\max_{i=1,\ldots,N}|\t{X}_{\alpha,i}^\eps(t)-X_{\alpha,i}^\eps(t)\big|\geq N^{-\lambda}\big)\leq C(T,\gamma) N^{-\gamma}.
$$   
\end{proposition}
Without bringing confusion later, the notation $\max_{\alpha=1,\ldots,n}\max_{i=1,\ldots,N}$ is always shorten as $\max_{\alpha,i}$. 
Whenever we use matrix valued functions, $|A|$ denotes the Frobenius norm of the matrix $A$. Before getting into the proof of Proposition \ref{convergence in prob}, we present an important ingredient first, namely a version of law of large numbers result.

\begin{lemma}[Law of large numbers]\label{LLN} Let $\t{X}_{\alpha,i}^\eps$ be the solution of system \eqref{tildeXi} and let $\t{f}_{\alpha,\eps}$ be the density function associated to $\t{X}_{\alpha,i}^\eps$ satisfying \eqref{intermediate}. Given $ 0\leq\theta <\frac{1}{2}$ and a family of  bounded functions $ \Psi_\eps=\{\psi_{\eps}^{\alpha,\beta}\}_{\alpha,\beta=1,\ldots,n}$ depending on $\eps$ which can take values in $\R$, $\R^d$ or $\R^{d\times d}$, we define the set 
\begin{equation}\label{def set A}
\begin{aligned}
& \mathcal{A}_{\theta, \Psi_\eps}^N(t):=\bigcup_{\alpha,\beta=1}^n\bigcup_{i=1}^N\bigg\{\omega \in \Omega:\bigg|\frac{1}{N} \sum_{j=1}^N \psi_\eps^{\alpha,\beta}\big(\t{X}_{\alpha,i}^\eps(t)-\t{X}_{\beta,j}^\eps(t)\big)-\big(\psi_\eps^{\alpha,\beta} * \t{f}_{\beta,\eps}\big)\big(\t{X}_{\alpha,i}^\eps(t)\big)\bigg|>N^{-\theta}\bigg\}.
\end{aligned}
\end{equation}
Then, for every $m \in \mathbb{N}$ and $T>0$, it holds
$$
\begin{aligned}
& \sup_{t\in[0,T]}\P\left(\mathcal{A}_{\theta, \Psi_\eps}^N(t)\right) \leq n^2\max_{\alpha,\beta}C(m,\alpha,\beta,T)\left\|\psi_\eps^{\alpha,\beta}\right\|_{L^{\infty}}^{2 m} N^{ m(2\theta-1)+1}.
\end{aligned}
$$
\end{lemma}
The fact that $\theta<\frac{1}{2}$ can be heuristically interpreted by the scaling of the central limit theorem. We can also see easily that
$$\big(\mathcal{A}_{\theta, \Psi_\eps}^N(t)\big)^c:=\bigcap_{\alpha,\beta=1}^n\bigcap_{i=1}^N\bigg\{\omega \in \Omega:\bigg|\frac{1}{N} \sum_{j=1}^N \psi_\eps^{\alpha,\beta}\big(\t{X}_{\alpha,i}^\eps(t)-\t{X}_{\beta,j}^\eps(t)\big)-\big(\psi_\eps^{\alpha,\beta} * \t{f}_{\beta,\eps}\big)\big(\t{X}_{\alpha,i}^\eps(t)\big)\bigg|\leq N^{-\theta}\bigg\}.
$$ The proof is similar to \cite[Lemma 4.2]{H23} but here it is for multi-species case, which we will prove in Appendix \ref{appendix LLN}.

We also present an $L^\infty$-bound estimate of the regularised potential as a separate remark, which will appear several times later. 
\begin{remark}\label{fractional gradient}
For $k=1,2$, we have the  following bound: 
\begin{equation*}
\begin{aligned}
\left\|\nabla^kV^\eps_{\alpha\beta}\right\|_{L^\infty}\leq \max_{\alpha,\beta}|a_{\alpha\beta}| \left\|\frac{1}{|\cdot|^{s}}\ast \nabla^{k}\chi^\eps\right\|_{L^{\infty}} 
\leq & \,\frac{C(\max_{\alpha,\beta}|a_{\alpha\beta}|, \|\chi\|_{W^{2,1}\cap W^{2,\infty} })}{\eps^{k+s}}\\= & \,C(\max_{\alpha,\beta}|a_{\alpha\beta}|, \|\chi\|_{W^{2,1}\cap W^{2,\infty} })N^{\ell(k+s)},
\end{aligned}    
\end{equation*}
where the proof can be found in \cite[Lemma 17]{chen2024fluctuations}. 
\end{remark}
To prove Proposition \ref{convergence in prob}, we construct the stopping time, for some parameter $\lambda>0$, as 
$$
\tau_{\lambda}(\omega):=\inf\left\{t\in(0,T):\max_{\alpha,i}\big|\t{X}_{\alpha,i}^\eps(t)-X_{\alpha,i}^\eps(t)\big|\geq \frac{1}{N^\lambda}\right\}>0,
$$
which is well-defined, since the corresponding SDEs have continuous trajectories.
Define the stochastic process $S_\lambda$  as
\begin{equation}\label{Slambda}
\begin{aligned}
S_{\lambda}(t):=N^\lambda \max_{\alpha,i}\big|\t{X}_{\alpha,i}^\eps(t\wedge \tau_\lambda)-X_{\alpha,i}^\eps(t\wedge \tau_\lambda)\big|.
\end{aligned}
\end{equation}
It is easy to see that $S_\lambda(t)\leq 1$, which we will use later. The required probability of set \eqref{clam} in Proposition \ref{convergence in prob} can be bounded by the expectation of the process $S_\lambda^p(t):=(S_\lambda(t))^p$ as
\begin{equation}\label{any p}
\begin{aligned}
\P(\C_{\lambda}(t))\leq \P(\{\omega\in\Omega:\tau_{\lambda}\leq t\})
=&\,\P\big(\{\omega\in\Omega:\max_{\alpha,i}\big|\t{X}_{\alpha,i}^\eps(t\wedge \tau_{\lambda})-X_{\alpha,i}^\eps(t\wedge \tau_{\lambda})\big|= N^{-\lambda}\}\big)\\
=\,& \P (\{\omega\in \Omega: S_{\lambda}(t)=1\}),
\end{aligned}
\end{equation}
because the event $\max_{\alpha,i}|\t{X}_{\alpha,i}^\eps(t)-X_{\alpha,i}^\eps(t)\big|\geq N^{-\lambda}$ implies the event $\tau_{\lambda}\leq t$ by the definition of the set $\C_{\lambda}(t)$ in \eqref{clam}, and then for almost any $\omega$ such that $\tau_{\lambda}\leq t$, it holds 
$$
\max_{\alpha,i}\big|\t{X}_{\alpha,i}^\eps(t\wedge \tau_{\lambda})-X_{\alpha,i}^\eps(t\wedge \tau_{\lambda})\big|=\max_{\alpha,i}\big|\t{X}_{\alpha,i}^\eps( \tau_{\lambda})-X_{\alpha,i}^\eps( \tau_{\lambda})\big|= N^{-\lambda}.
$$
Since the set $\{\omega\in \Omega: S_{\lambda}^p(t)=1\}$ for some power $p\in\N$, actually does not depend on $p$, then for any $p\in\N$, it yields by Markov's inequality that
\begin{equation}\label{clamp}
\P(\C_{\lambda}(t))\leq \P (\{\omega\in \Omega: S_{\lambda}(t)=1\}) =\P\big(\{\omega\in \Omega: S_{\lambda}^p(t)\geq 1\}\big)\leq \E [S_{\lambda}^p(t)].
\end{equation}
Notice that \eqref{clamp} holds for any $p\in\N$, hence it is sufficient to show the following lemma with a suitable $p$ to deduce Proposition \ref{convergence in prob}.
\begin{lemma}\label{expectation lemma}
Assume the range of $\ell$ is \eqref{range l}. For any $\gamma>0$, there exist some $p\in \N$ and $\ell< \lambda<1/2-\ell (s+1)$ such that 
\begin{equation}\label{expectation S}
\sup_{t\in[0,T]}\E [S_{\lambda}^p(t)]\leq C(n,T,p)N^{-\gamma},
\end{equation}
where the constant is independent of $N$.    
\end{lemma} 
\begin{proof}
Since the particle system \eqref{X} satisfied by $X_{\alpha,i}^\eps$ and the independent copy of the intermediate SDE  \eqref{tildeXi} satisfied by $\t{X}_{\alpha,i}^\eps$ have the same Brownian motion,  It\^{o}'s formula reduces to the following differential identity 
$$
\begin{aligned}
&\d_\tau\big|\t{X}_{\alpha,i}^\eps(\tau)-X_{\alpha,i}^\eps(\tau)\big|^p=p\big|\t{X}_{\alpha,i}^\eps(\tau)-X_{\alpha,i}^\eps(\tau)\big|^{p-2}\big(\t{X}_{\alpha,i}^\eps(\tau)-X_{\alpha,i}^\eps(\tau)\big)\cdot\d_\tau\big(\t{X}_{\alpha,i}^\eps(\tau)-X_{\alpha,i}^\eps(\tau)\big) \\
= &\,p\big|\t{X}_{\alpha,i}^\eps(\tau)-X_{\alpha,i}^\eps(\tau)\big|^{p-2}\big(\t{X}_{\alpha,i}^\eps(\tau)-X_{\alpha,i}^\eps(\tau)\big)\\&\qquad\times \Big(\frac{1}{N}\sum_{j=1}^N \sum_{\beta=1}^n \nabla V_{\alpha\beta}^\eps( X_{\alpha,i}^\eps(\tau) - X_{\beta,j}^\eps(\tau))-\sum_{\beta=1}^n\nabla V_{\alpha\beta}^\eps \ast \t{f}_{\beta,\eps}(\tau,\t{X}_{\alpha,i}^\eps(\tau))\Big) 
\end{aligned}
$$
with $\t{X}_{\alpha,i}^\eps(0)=X_{\alpha,i}^\eps(0)$, because the quadratic variation term vanishes. We integrate from time $0$ to $t\wedge \tau_{\lambda}$ on both hand-side of the equality above, then for almost all $\omega\in\Omega$, we have the estimate
$$
\begin{aligned}
&\big|\t{X}_{\alpha,i}^\eps(t\wedge \tau_{\lambda})-X_{\alpha,i}^\eps(t\wedge \tau_{\lambda})\big|^p\\\leq &\int_0^{t\wedge \tau_{\lambda}}p\big|\t{X}_{\alpha,i}^\eps(\tau)-X_{\alpha,i}^\eps(\tau)\big|^{p-1}\bigg|\frac{1}{N}\sum_{j=1}^N \sum_{\beta=1}^n \nabla V_{\alpha\beta}^\eps( X_{\alpha,i}^\eps(\tau)  X_{\beta,j}^\eps(\tau))\\&\qquad\qquad\qquad\qquad\qquad\qquad\qquad\qquad\qquad\qquad-\sum_{\beta=1}^n\nabla V_{\alpha\beta}^\eps \ast \t{f}_{\beta,\eps}(\tau,\t{X}_{\alpha,i}^\eps(\tau))\bigg|\,\d \tau.     
\end{aligned}
$$
Recalling the definition of $S_{\lambda}$ in \eqref{Slambda}, then we  control $S_{\lambda}^p(t)$ by the sum of two quantities $I_1(t)$ and $I_2(t)$ as
$$
\begin{aligned}
S_{\lambda}^p(t)=\,&N^{\lambda p} \max_{\alpha,i}\big|\t{X}_{\alpha,i}^\eps(t\wedge \tau_{\lambda})-X_{\alpha,i}^\eps(t\wedge \tau_{\lambda})\big|^p   \\
\leq &\, N^{\lambda p}\max_{\alpha,i}\int_0^{t\wedge \tau_{\lambda}}p\big|\t{X}_{\alpha,i}^\eps(\tau)-X_{\alpha,i}^\eps(\tau)\big|^{p-1}\bigg|\frac{1}{N}\sum_{j=1}^N \sum_{\beta=1}^n \nabla V_{\alpha\beta}^\eps( X_{\alpha,i}^\eps(\tau) - X_{\beta,j}^\eps(\tau))\\&\qquad\qquad\qquad\qquad\qquad\qquad\qquad\qquad\qquad\qquad\qquad-\sum_{\beta=1}^n\nabla V_{\alpha\beta}^\eps \ast \t{f}_{\beta,\eps}(\tau,\t{X}_{\alpha,i}^\eps(\tau))\bigg|\,\d \tau\\
\leq &\, p N^{\lambda }\max_{\alpha,i}\sum_{\beta=1}^n\int_0^{t\wedge \tau_{\lambda}}S_{\lambda}^{p-1}(\tau)\bigg|\frac{1}{N}\sum_{j=1}^N  \nabla V_{\alpha\beta}^\eps( X_{\alpha,i}^\eps(\tau) - X_{\beta,j}^\eps(\tau))\\&\qquad\qquad\qquad\qquad\qquad\qquad\qquad\qquad\qquad\qquad-\nabla V_{\alpha\beta}^\eps \ast \t{f}_{\beta,\eps}(\tau,\t{X}_{\alpha,i}^\eps(\tau))\bigg|\,\d \tau,
\end{aligned}
$$
where we replace $N^{\lambda(p-1)}\big|\t{X}_{\alpha,i}^\eps(\tau)-X_{\alpha,i}^\eps(\tau)\big|^{p-1}$ by $S_\lambda^{p-1}(\tau)$ for any $0\leq\tau\leq t\wedge \tau_\lambda$. We further let
$$
\begin{aligned}
I_{1,\alpha,\beta,i}(\tau)=\,&\Big|\frac{1}{N}\sum_{j=1}^N  \nabla V_{\alpha\beta}^\eps(\t X_{\alpha,i}^\eps(\tau) -\t X_{\beta,j}^\eps(\tau))-\nabla V_{\alpha\beta}^\eps \ast \t{f}_{\beta,\eps}(\t{X}_{\alpha,i}^\eps(\tau))\Big|,
\end{aligned}
$$
and
%\begin{equation}\label{I2abi}
$$\begin{aligned}
I_{2,\alpha,\beta,i}(\tau)=\,&\Big|\frac{1}{N}\sum_{j=1}^N \Big(\nabla V_{\alpha\beta}^\eps(X_{\alpha,i}^\eps - X_{\beta,j}^\eps) -\nabla V_{\alpha\beta}^\eps(\t X_{\alpha,i}^\eps -\t X_{\beta,j}^\eps)\Big)\Big|.
\end{aligned}
$$%\end{equation}
Then it holds 
\begin{equation}\label{I1I2}
\begin{aligned}
\E[S_{\lambda}^p(t)]\leq & np N^{\lambda }\E\left[\int_0^{t\wedge \tau_{\lambda}}S_{\lambda}^{p-1}(\tau)\max_{\alpha,\beta}\max_{i}I_{1,\alpha,\beta,i}(\tau)\d \tau\right]\\&+np N^{\lambda }\E\left[\int_0^{t\wedge \tau_{\lambda}}S_{\lambda}^{p-1}(\tau)\max_{\alpha,\beta}\max_{i}I_{2,\alpha,\beta,i}(\tau)\d \tau\right]\\
=:& I_1(t)+I_2(t).
\end{aligned}
\end{equation}
For some arbitrary $\theta_1\in[0,1/2)$ and $m_1\in \N$, Lemma \ref{LLN} leads to the estimate for the probability of the set defined as \eqref{def set A} with a family of the vector-valued functions $\nabla\mathbf{V}^\eps:=\{\nabla V^\eps_{\alpha\beta}\}_{\alpha,\beta=1,\ldots,n}$ as 
$$
\sup_{t\in[0,T]}\P\left(\mathcal{A}_{\theta_1, \nabla\mathbf{V}^{\eps}}^N(t)\right) \leq n^2\max_{\alpha,\beta}C(m_1,\alpha,\beta,T)\left\|\nabla V^\eps_{\alpha\beta}\right\|_{L^{\infty}}^{2 m_1} N^{ m_1(2\theta_1-1)+1},
$$
where the set is $$
\begin{aligned}
\mathcal{A}_{\theta_1, \nabla\mathbf{V}^{\eps}}^N:=\bigcup_{\alpha,\beta=1}^n\bigcup_{i=1}^N\bigg\{\omega \in \Omega:\bigg|\frac{1}{N} \sum_{j=1}^N \nabla V^\eps_{\alpha\beta}\big(\t{X}_{\alpha,i}^\eps-\t{X}_{\beta,j}^\eps\big)-\big(\nabla V^\eps_{\alpha\beta} * \t{f}_{\beta,\eps}\big)\big(\t{X}_{\alpha,i}^\eps\big)\bigg|>N^{-\theta_1}\bigg\}.
\end{aligned}
$$
As a result, we infer that by using Young's inequality
$$
I_1(t)\leq \frac{p-1}{p}\E\Big[\int_0^{t\wedge \tau_{\lambda}}S_{\lambda}^p(\tau) \,\d \tau \Big] +n^pp^{p-1}N^{\lambda p} \E\Big[\int_0^{t\wedge \tau_{\lambda}}  \max_{\alpha,\beta,i}\big|I_{1,\alpha,\beta,i}(\tau)\big|^p\,\d \tau\bigg]. 
$$
We now split the last integral as
\begin{align*}
\E\Big[\int_0^{t\wedge \tau_{\lambda}} & \max_{\alpha,\beta,i}\big|I_{1,\alpha,\beta,i}(\tau)\big|^p\,\d \tau\bigg]\\ 
&\leq
\int_0^t\E\Big[\max_{\alpha,\beta,i}\big| I_{1,\alpha,\beta,i}(\tau)\big|^pI_{(\mathcal{A}^N_{\theta_1, \nabla\mathbf{V}^{\eps}})^c}\Big]\d \tau  
+\int_0^t\E\Big[\max_{\alpha,\beta,i}\big| I_{1,\alpha,\beta,i}(\tau)\big|^pI_{\mathcal{A}^N_{\theta_1, \nabla\mathbf{V}^{\eps}}}\Big]\d \tau  \\
&\leq\,T  N^{-p\theta_1}+C(m_1)Tn^{2} \max_{\alpha,\beta}\left\|\nabla V_{\alpha\beta}^\eps\right\|_{L^{\infty}}^{2 m_1+p} N^{ m_1(2\theta_1-1)+1}.
\end{align*}
Collecting terms, we obtain
\begin{equation}\label{eq I_1}
I_1(t)
\leq\, C(n,p,T) N^{p(\lambda -\theta_1)}+C(m_1,n,p,T)N^{\lambda p+(2m_1+p)\ell(1+s)+m_1(2\theta_1-1)+1}+\int_0^t\E\big[S_{\lambda}^p(\tau) \big]\,\d \tau,
\end{equation}
where we used Remark \ref{fractional gradient} in the last step to estimate $\|\nabla V^\eps_{\alpha\beta}\|_{L^\infty}$. 

Now we focus on the mean-field estimate of $I_2$ in \eqref{I1I2}. Recall the regularised potential $V_{\alpha\beta}^\eps= a_{\alpha\beta}\chi^\eps \ast V$ with the mollifier satisfying assumption (H1), (H2) and (H5), where $\chi^\eps$ is supported on $B_\eps$. 
When $|x|\geq 2\eps$, we have $|x|\leq 2|x-y|$ with $y\in B_\eps$. And for some constant $C_1$ depending on $s$, $d$ and $\max_{\alpha,\beta}|a_{\alpha\beta}|$,it yields 
$$
|\nabla^2 V^\eps_{\alpha\beta}(x)|\leq C_1\int_{B_\eps}\frac{1}{|x-y|^{s+2}}\chi^\eps(y)\,\d y\leq C_1 \sup_{y\in B_\eps}\frac{1}{|x-y|^{s+2}}\int_{B_\eps}\chi^\eps(y)\,\d y\leq \frac{2^{s+2}C_1}{|x|^{s+2}}.
$$
And for any $x$,  it holds by Remark \ref{fractional gradient} that 
$$
|\nabla^2 V^\eps_{\alpha\beta}(x)|\leq \|\nabla^2 V^\eps_{\alpha\beta}\|_{L^\infty}\leq \frac{C_2}{\eps^{s+2}},
$$
where the constant 
$C_2$ depends on  $\max_{\alpha,\beta}|a_{\alpha\beta}|$ and $\|\chi\|_{W^{2,1}\cap W^{2,\infty} }$. We now use these bounds of the Hessians as follows: if $|x|\geq 4\eps$ and $|\xi|\leq2\eps$, then there exits some constant $\iota\in(0,1)$ such that 
$$
|\nabla V^\eps_{\alpha\beta}(x+\xi)-\nabla V^\eps_{\alpha\beta}(x)|\leq |\nabla^2  V^\eps_{\alpha\beta}(x+\iota\xi)||\xi|\leq \frac{2^{s+2}C_1}{|x|^{s+2}}|\xi|;
$$
on the other hand, for $|x|<4\eps$ and $|\xi|\leq 2\eps$, we have 
$$
|\nabla V^\eps_{\alpha\beta}(x+\xi)-\nabla V^\eps_{\alpha\beta}(x)|\leq |\nabla^2  V^\eps_{\alpha\beta}(x+\iota\xi)||\xi|\leq  \frac{C_2}{\eps^{s+2}}|\xi|. 
$$
Inspired by \cite{HLP20}, we construct an auxiliary continuous function $K^\eps:\R^{d}\rightarrow\R$ as follows, 
\begin{equation}\label{auxiliary}
K^\eps(x)=\begin{cases}\displaystyle
    \frac{1}{|x|^{ {s+2}}}\quad |x|\geq 4\eps \\ \\
    \displaystyle\frac{1}{(4\eps)^{s+2}}\quad |x|< 4\eps,
\end{cases}\end{equation}
which can be used to bound $|\nabla V^\eps_{\alpha\beta}(x+\xi)-\nabla V^\eps_{\alpha\beta}(x)|$ in a unified way for any $x\in\R^d$ with $|\xi|\leq 2\eps$. For some constant $C_3 =C_3 (d,s,\|\chi\|_{W^{2,1}\cap W^{2,\infty} },
\max_{\alpha,\beta}|a_{\alpha\beta}|)$, we have for all $x$ that
\begin{equation}\label{ineq C3}
|\nabla V^\eps_{\alpha\beta}(x+\xi)-\nabla V^\eps_{\alpha\beta}(x)|\leq  C_3  K^\eps(x)|\xi|,\quad |\xi|\leq 2\eps. 
\end{equation}
Notice that if we assume $\lambda> \ell$, then we have the following bound before the stopping time $\tau_\lambda$ 
\begin{equation}\label{double difference}
\big|(X_{\alpha,i}^\eps- X_{\beta,j}^\eps)-(\t X_{\alpha,i}^\eps -\t X_{\beta,j}^\eps)\big|(\tau)\leq 2\max_{\alpha,i}\big|X_{\alpha,i}^\eps(\tau) - \t X_{\alpha,i}^\eps(\tau) \big|\leq 2N^{-\lambda}< 2N^{-\ell}= 2\eps,
\end{equation}
for any $\tau\leq \tau_\lambda$. Remark that the reason why we cannot let $\lambda=\ell$ is because that
\eqref{J1} in the next section can be controlled under condition $\lambda>\ell$ .  Putting together \eqref{ineq C3} and \eqref{double difference}, we deduce
$$
\begin{aligned}
I_{2,\alpha,\beta,i}\leq&\Big|\frac{1}{N}\sum_{j=1}^N \big|\nabla V_{\alpha\beta}^\eps(X_{\alpha,i}^\eps - X_{\beta,j}^\eps) -\nabla V_{\alpha\beta}^\eps(\t X_{\alpha,i}^\eps -\t X_{\beta,j}^\eps)\big|\Big|\\\leq&C_3 \Big|\frac{1}{N}\sum_{j=1}^N K^\eps(\t X_{\alpha,i}^\eps -\t X_{\beta,j}^\eps)\big|(X_{\alpha,i}^\eps - X_{\beta,j}^\eps)-(\t X_{\alpha,i}^\eps -\t X_{\beta,j}^\eps)\big|\Big|\\
\leq&C_3  \Big|\Big(\frac{1}{N}\sum_{j=1}^N  K^\eps(\t X_{\alpha,i}^\eps -\t X_{\beta,j}^\eps)-K^\eps\ast \t f_{\beta,\eps}(\t X_{\alpha,i}^\eps)\Big)\big|(X_{\alpha,i}^\eps - X_{\beta,j}^\eps)-(\t X_{\alpha,i}^\eps -\t X_{\beta,j}^\eps)\big|\Big|\\&+C_3 \Big|K^\eps\ast \t f_{\beta,\eps}(\t X_{\alpha,i}^\eps)\big|(X_{\alpha,i}^\eps - X_{\beta,j}^\eps)-(\t X_{\alpha,i}^\eps -\t X_{\beta,j}^\eps)\big|\Big|\\
\leq&2C_3  \Big|\frac{1}{N}\sum_{j=1}^N  K^\eps(\t X_{\alpha,i}^\eps -\t X_{\beta,j}^\eps)-K^\eps\ast \t f_{\beta,\eps}(\t X_{\alpha,i}^\eps)\Big|\max_{\alpha,i}\big|X_{\alpha,i}^\eps - \t X_{\alpha,i}^\eps \big|\\&+2C_3  \big\|K^\eps \ast \t f_{\beta,\eps}\big\|_{L^\infty}\max_{\alpha,i}\big|X_{\alpha,i}^\eps - \t X_{\alpha,i}^\eps \big|.
\end{aligned}
$$
Plugging the estimate above into $I_2$ we get
\begin{equation}\label{EI2t}
\begin{aligned}
I_2(t)=&np N^{\lambda }\E\left[\int_0^{t\wedge \tau_{\lambda}}S_{\lambda}^{p-1}(\tau)\max_{\alpha,\beta}\max_{i}I_{2,\alpha,\beta,i}(\tau)\d \tau\right]\\\leq& 2npC_3 \E\bigg[pN^{\lambda }   \int_0^{t\wedge \tau_{\lambda}}S_{\lambda}^{p-1}\max_{\alpha,i}\big|X_{\alpha,i}^\eps - \t X_{\alpha,i}^\eps \big|\\&\qquad\qquad\qquad\times \max_{\alpha ,\beta}\max_{i}\Big|\frac{1}{N}\sum_{j=1}^N  K^\eps(\t X_{\alpha,i}^\eps -\t X_{\beta,j}^\eps)-K^\eps\ast \t f_{\beta,\eps}(\t X_{\alpha,i}^\eps)\Big|\,\d \tau    \bigg]\\&+2npC_3 \E\bigg[pN^{\lambda }   \int_0^{t\wedge \tau_{\lambda}}S_{\lambda}^{p-1}\max_\beta \|K^\eps \ast \t f_{\beta,\eps}\|_{L^\infty}\max_{\alpha,i}\big|X_{\alpha,i}^\eps - \t X_{\alpha,i}^\eps \big|\,\d \tau    \bigg]\\
=\,&2npC_3 \E\bigg[  \int_0^{t\wedge \tau_\lambda}S_\lambda^{p}\max_{\alpha ,\beta}\max_{i}\Big|\frac{1}{N}\sum_{j=1}^N  K^\eps(\t X_{\alpha,i}^\eps -\t X_{\beta,j}^\eps)-K^\eps\ast \t f_{\beta,\eps}(\t X_{\alpha,i}^\eps)\Big| \,\d \tau    \bigg]\\&+2npC_3 \max_{\beta}\|K^\eps \ast \t f_{\beta,\eps}\|_{L^\infty}\int_0^t\E[S_\lambda^p] \,\d \tau.
\end{aligned}
\end{equation}
For the first term on the right-hand side above, we use Young's inequality to obtain
$$
\begin{aligned}
&2npC_3 \E\bigg[  \int_0^{t\wedge \tau_{\lambda}}\max_{\alpha ,\beta}\max_{i}\Big|\frac{1}{N}\sum_{j=1}^N  K^\eps(\t X_{\alpha,i}^\eps -\t X_{\beta,j}^\eps)-K^\eps\ast \t f_{\beta,\eps}(\t X_{\alpha,i}^\eps)\Big|S_{\lambda}^p \,\d \tau    \bigg]\\\leq&\frac{(2npC_3 )^p}{p}\E\bigg[  \int_0^{t\wedge \tau_{\lambda}}\max_{\alpha ,\beta}\max_{i}\Big|\frac{1}{N}\sum_{j=1}^N  K^\eps(\t X_{\alpha,i}^\eps -\t X_{\beta,j}^\eps)-K^\eps\ast \t f_{\beta,\eps}(\t X_{\alpha,i}^\eps)\Big|^p\,\d \tau\bigg]\\&\qquad\qquad\qquad\qquad\qquad\qquad\qquad\qquad\qquad\qquad\qquad+\E\bigg[\frac{p-1}{p}\int_0^{t\wedge \tau_{\lambda}}(S_{\lambda})^{\frac{p^2}{p-1}} \,\d \tau    \bigg]\\\leq&\,
\frac{(2npC_3 )^p}{p}\int_0^t\E\bigg[  \max_{\alpha ,\beta}\max_{i}\Big|\frac{1}{N}\sum_{j=1}^N  K^\eps(\t X_{\alpha,i}^\eps -\t X_{\beta,j}^\eps)-K^\eps\ast \t f_{\beta,\eps}(\t X_{\alpha,i}^\eps)\Big|^p\bigg]\,\d \tau+\int_0^t\E[S_{\lambda}^p] \,\d \tau,   
\end{aligned}
$$
where we used the property $S_\lambda\leq 1$. We then apply the law of large numbers (Lemma \ref{LLN}) with $\|K^\eps\|_{L^\infty}\leq  (4\eps)^{-(s+2)}=(4N^{\ell})^{s+2}$ to the first term on the right-hand side above. Similar to the aforementioned estimate of $I_1(t)$, for some $0\leq\theta_2<1/2$ and any $m_2\in\N$, it holds 
$$
\begin{aligned}
\int_0^t\E\bigg[  \max_{\alpha ,\beta}\max_{i}&\Big|\frac{1}{N}\sum_{j=1}^N  K^\eps(\t X_{\alpha,i}^\eps -\t X_{\beta,j}^\eps)-K^\eps\ast \t f_{\beta,\eps}(\t X_{\alpha,i}^\eps)\Big|^p\bigg]\,\d \tau\\& \leq  C(n,m_2)TN^{(2m_2+p)\ell(s+2)+m_2(2\theta_2-1)+1}+TN^{-p\theta_2}.
\end{aligned}
$$
Since the uniform bound \eqref{uniform in eps} of $\t f_{\eps}$, we have
\begin{equation}\label{K bound}
\begin{aligned}
 \big\|K^\eps\ast \t f_{\beta,\eps}\big\|_{L^\infty}\leq & \Big\| \int_{|x-y|<4\eps} K^\eps(x-y)\t f_{\beta,\eps}(y)\,\d y\Big\|_{L^\infty}\\
 &+\Big\| \int_{4\eps\leq|x-y|<1} K^\eps(x-y)\t f_{\beta,\eps}(y)\,\d y\Big\|_{L^\infty}+\Big\| \int_{|x-y|\geq 1} K^\eps(x-y)\t f_{\beta,\eps}(y)\,\d y\Big\|_{L^\infty}\\
 \leq & \frac{\|\t f_{\beta,\eps}\|_{L^\infty}}{(4\eps)^{s+2}}|B_{4\eps}|+\|\t f_{\beta,\eps}\|_{L^\infty}\int_{4\eps\leq|x-y|<1} \frac{1}{|x-y|^{s+2}}\,\d y+\|\t f_{\beta,\eps}\|_{L^1}.
\end{aligned}
\end{equation}
For the sub-Coulomb case, it holds by \eqref{uniform in eps} that
\begin{equation}\label{bound K subcoulomb}
\big\|K^\eps\ast \t f_{\beta,\eps}\big\|_{L^\infty}\leq C(d)\sup_\eps\|\t f_{\beta,\eps}\|_{L^\infty(0,T;L^1\cap L^\infty)}<C(T,d),
\end{equation}
where the constant is independent of $\eps$.
While for the critical Coulomb case ($s=d-2$), the following bound holds
$$
\int_{4\eps\leq|x-y|<1} \frac{1}{|x-y|^{d}}\,\d y\leq C(d) \log\frac{1}{\eps}\leq C(d)\ell\log N,
$$
which implies that 
\begin{equation}\label{bound K coulomb}
\big\|K^\eps\ast \t f_{\beta,\eps}\big\|_{L^\infty}\leq C(d)\sup_\eps\|\t f_{\beta,\eps}\|_{L^\infty(0,T;L^1\cap L^\infty)}\ell\log N<C(T,d)\ell\log N,
\end{equation}
where the constant is independent of $\eps$. Then we get the following  controls for \eqref{EI2t}: for the Coulomb case $s=d-2$,
\begin{equation}\label{eq I_2 coulomb}
\begin{aligned}
I_2(t)\leq& C(m_2,n,p,T)N^{(2m_2+p)\ell(s+2)+m_2(2\theta_2-1)+1}+C(n,p,T)N^{-p\theta_2}+C_4p\ell\log N\int_0^t\E[S_\lambda^p(\tau)] \,\d \tau;
\end{aligned}
\end{equation}
for the sub-Coulomb case $s<d-2$,
\begin{equation}\label{eq I_2 subcoulomb}
\begin{aligned}
I_2(t)\leq& C(m_2,n,p,T)N^{(2m_2+p)\ell(s+2)+m_2(2\theta_2-1)+1}+C(n,p,T)N^{-p\theta_2}+C_4p\int_0^t\E[S_\lambda^p(\tau)] \,\d \tau,
\end{aligned}
\end{equation}
where the constant 
$$C_4=2npC_3C(T,d)=C_4(T,d,\chi,n,\max_{\alpha,\beta}|a_{\alpha\beta}|,\max_\alpha\sup_\eps\|\t f_{\eps,\alpha}\|_{L^\infty(0,T;L^1\cap L^\infty)}).$$
Combining the estimate of $I_1(t)$ satisfied by \eqref{eq I_1} and estimate of $I_2(t)$ satisfied by \eqref{eq I_2 coulomb} or \eqref{eq I_2 subcoulomb}, we conclude that for the Coulomb case $s=d-2$, \eqref{I1I2} can be controlled as
\begin{equation}\label{ESN coulomb}
\begin{aligned}
\E\big[S_\lambda^p(t)\big]\leq& C(n,p,T)N^{p(\lambda -\theta_1)}+C(m_1,n,p,T)N^{m_1(\ell (2s+2)+2\theta_1-1)+p(\ell (s+1)+\lambda )+1}+C(n,p,T)N^{-p\theta_2}\\&+C(m_2,n,p,T)N^{m_2(\ell(2s+4)+2\theta_2-1)+p(\ell(s+2))+1}+C_4p\ell\log N\int_0^{t} \E\big[S_\lambda^p(\tau)\big]\,\d \tau;
\end{aligned}
\end{equation}
and for the sub-Coulomb case $s<d-2$, \eqref{I1I2} can be controlled as
\begin{align}\label{ESN subcoulomb}
\E\big[S_\lambda^p(t)\big]\leq& C(n,p,T)N^{p(\lambda -\theta_1)}+C(m_1,n,p,T)N^{m_1(\ell (2s+2)+2\theta_1-1)+p(\ell (s+1)+\lambda )+1}+C(n,p,T)N^{-p\theta_2}\nonumber
\\&+C(m_2,n,p,T)N^{m_2(\ell(2s+4)+2\theta_2-1)+p(\ell(s+2))+1}+C_4p\int_0^{t} \E\big[S_\lambda^p(\tau)\big]\,\d \tau.
\end{align}
Notice that the only difference between \eqref{ESN coulomb} and \eqref{ESN subcoulomb} is the prefactor of the last integral.

For any given $\gamma'>0$, our aim to conclude the proof is to choose suitable $(\theta_1, \theta_2,p,m_1,m_2)$ in order such that the sum of the first four terms on the right-hand side of both \eqref{ESN coulomb} and \eqref{ESN subcoulomb} can be controlled by $C(m_1,m_2,n,p,T)N^{-\gamma'}$ for any large enough $N$. Then, we require for all $0<s\leq d-2$ that 
\begin{itemize}
\item[(1)] $p(\lambda-\theta_1)\leq -\gamma' $; 
\item[(2)]$m_1(\ell (2s+2)+2\theta_1-1)+p(\ell (s+1)+\lambda )+1\leq -\gamma'$;
\item[(3)]$-p\theta_2\leq-\gamma'$; 
\item[(4)]$m_2(\ell(2s+4)+2\theta_2-1)+p(\ell(s+2))\leq-\gamma'$; 
\item[(5)]$\lambda>\ell$.
\end{itemize}
where the last constraint comes from \eqref{double difference}.
We choose first $(\theta_1, \theta_2)$ such that
$$
0<\ell<\lambda<\theta_1 <\frac{1}{2}-\ell(s+1), \quad 0\leq\theta_2<\frac{1}{2}-\ell(s+2)
$$
holds. Then we are able to take $p$ large enough such that conditions $(1)$ and $(3)$ are satisfied. We now choose $(m_1,m_2)$ depending on $p$, such that conditions $(2)$ and $(4)$ are satisfied. To make the range of $\lambda$, $\theta_1$ and $\theta_2$ not empty,  the range of $\ell$ needs to be taken as 
\begin{equation*}
0<\ell<\frac{1}{2s+4},
\end{equation*} 
with the corresponding range of $\lambda$ taken as \eqref{range lambda}, i.e.
$\ell<\lambda<1/2-\ell(s+1).$
Actually, we can take $\theta_1$ close enough to $1/2-\ell(s+1)$ and $\lambda$ close enough to $\ell$ such that $\theta_1-\lambda=\theta_2$. A choice of $p$ for all cases is $p=p'$ which satisfies
\begin{equation}\label{p'}
\Big(\frac{1}{2}-\ell(s+2)\Big)p'-\varrho'=\gamma'.
\end{equation}
for arbitrarily small  $\varrho'$. It is also admissible for us to take any $p\geq p'$ in order to satisfy (1) and (3).  

In the sub-Coulomb case, we can then take for any $\gamma>0$, $\gamma'=\gamma$, and estimate the four terms in \eqref{ESN subcoulomb} by the power $N^{-\gamma}$ to conclude 
$$
\E\big[S_{\lambda}^p(t)\big]\leq C(m_1,m_2,p,T) N^{-\gamma}+C_4p\int_0^{t} \E\big[S_{\lambda}^p(\tau)\big]\,\d \tau,
$$
which implies by Gronwall's inequality that  
$$
\sup_{t\in[0,T]}\E [S_{\lambda}^p(t)]\leq \frac{C(p,T)}{N^\gamma}.
$$ 

Now, we turn to the Coulomb case by rewriting \eqref{ESN coulomb} as 
$$
\E\big[S_{\lambda}^p(t)\big]\leq C(m_1,m_2,n,p,T) N^{-(\frac{1}{2}-\ell(s+2))p+\varrho'}+C_4p\ell\log N\int_0^{t} \E\big[S_{\lambda}^p(\tau)\big]\,\d \tau,
$$
which implies by Gronwall's inequality that  
$$
\sup_{t\in[0,T]}\E [S_{\lambda}^p(t)]\leq C(m_1,m_2,n,p,T)\frac{e^{C_4Tp\ell \log N}}{N^{(\frac{1}{2}-\ell(s+2))p-\varrho'}}\leq C(m_1,m_2,n,p,T)N^{-\gamma},
$$
where 
\begin{equation}\label{constraint for coulomb}
\gamma=\gamma'-C_4Tp\ell\leq\left(\frac{1}{2}-\ell(C_4T+d)\right)p-\varrho'.\end{equation}
It requires further that 
$$
0<\ell<\frac{1}{2(C_4T+d)}
$$
In order to satisfy (1), (3) and $\eqref{constraint for coulomb}$, $p$ should not only satisfy $p\geq p'$ defined as \eqref{p'} but also $p\geq p''$ which satisfies
\begin{equation}\label{p''}
\Big(\frac{1}{2}-\ell (C_4T+d)\Big)p''-\varrho'=\gamma.
\end{equation}
Combining \eqref{p'} and  \eqref{p''}, we can take  $p=\max\{p',p''\}$ as one of the admissible choices. 

In conclusion, if we assume that the range of $\ell$ satisfies \eqref{range l}, 
$$
\left\{\begin{array}{cc}
    &\displaystyle0<\ell<\min\left(\frac{1} {C_0},\frac1{2d}\right),\quad\text{when } s=d-2\quad \text{(Coulomb case)},\\&
 \\&\displaystyle0<\ell<\frac{1}{2s+4},\quad\text{when } s<d-2\quad \text{(sub-Coulomb case)},
\end{array}\right.
$$where  $C_0$  is a constant such that $$C_0=2(C_4T+d)=C_0(T,d,n,\max_{\alpha,\beta}|a_{\alpha\beta}|, \|\chi\|_{W^{2,1}\cap W^{2,\infty}},\max_\alpha\sup_\eps\|\t f_{\eps,\alpha}\|_{L^\infty(0,T;L^1\cap L^\infty)}).$$
Then for some $\lambda$ lying in the non-empty set  $$\lambda\in\left(\ell,\frac{1}{2}-\ell (s+1)\right),\quad \text{for } s\leq d-2, $$
we have obtained \eqref{expectation S} holds for any $\gamma>0$, namely
$$
\sup_{t\in[0,T]}\E [S_{\lambda}^p(t)]\leq C(n,T,p,\gamma)N^{-\gamma}.
$$
\end{proof}

\begin{remark}\label{gradient2d}
The bounds in Remark \ref{fractional gradient} hold for the two dimensional Coulomb case $s=0$, and thus the main result of Proposition \ref{convergence in prob} holds in two dimensions for the $V(x)=\log|x|$ interaction potential. 
\end{remark}

\section{Distance between particle system and intermediate PDEs}\label{mean-field limit} 

We investigate the proof of Proposition \ref{RE convergence} in this section. We use the notation 
$$\d \mathbf{X}=\d x_{1,1}\cdots\d x_{1,N}\cdots\d x_{n,1}\cdots\d x_{n,N}$$ as the Lebesgue measure on $\R^{dnN}$. 
The time derivative of the relative entropy between the solution $f_{N,\eps}$ of the Liouville equation \eqref{Liouville} and the solution $\t{f}_{N,\eps}$ of the tensorised intermediate PDE \eqref{tensor intermediate} is as follows 
$$
\begin{aligned}
&\frac{\,\d}{\,\d t}H(f_{N,\eps}|\t{f}_{N,\eps})=\int_{\R^{dnN}}\p_tf_{N,\eps}\log \frac{f_{N,\eps}}{\t{f}_{N,\eps}}\,\d \mathbf{X}-\int_{\R^{dnN}}\p_t\t{f}_{N,\eps}\frac{f_{N,\eps}}{\t{f}_{N,\eps}}\,\d \mathbf{X}\\
=\,&-\sum_{\alpha=1}^n\sum_{i=1}^N\sigma_{\alpha}\int_{\R^{dnN}}f_{N,\eps}\big|\nabla_{x_{\alpha,i}}\log\frac{f_{N,\eps}}{ \t f_{N,\eps}}\big|^2\,\d \mathbf{X}\\
&+\sum_{\alpha,\beta=1}^n\sum_{i=1}^N\int_{\R^{dnN}}f_{N,\eps}\Big[\nabla V_{\alpha\beta}^\eps\ast \t{f}_{\beta,\eps}(x_{\alpha,i})-\frac{1}{N}\sum_{j=1}^N\nabla V_{\alpha\beta}^\eps(x_{\alpha,i}-x_{\beta,j})\Big]\cdot\nabla_{x_{\alpha,i}}\log\frac{f_{N,\eps}}{\t{f}_{N,\eps}}\,\d \mathbf{X},
\end{aligned}
$$
where Cauchy-Schwarz inequality implies $$
\frac{\,\d}{\,\d t}\frac{H(f_{N,\eps}|\t{f}_{N,\eps})}{N}\leq \frac{C}{N}\sum_{\alpha,\beta=1}^n\sum_{i=1}^N\int_{\R^{dnN}}f_{N,\eps}\Big|\nabla V_{\alpha\beta}^\eps\ast \t{f}_{\beta,\eps}(x_{\alpha,i})-\frac{1}{N}\sum_{j=1}^N\nabla V_{\alpha\beta}^\eps(x_{\alpha,i}-x_{\beta,j})\Big|^2\,\d \mathbf{X},
$$
and here we renormalise the relative entropy by dividing by $N$ and the constant $C$ dependent on $\sigma_\alpha$. We rewrite the term on the right-hand side into the expectation form as,
$$
\begin{aligned}
&\frac{1}{N}\sum_{\alpha,\beta=1}^n\sum_{i=1}^N\int_{\R^{dnN}}f_{N,\eps}\big|\nabla V_{\alpha\beta}^\eps\ast \t{f}_{\beta,\eps}(x_{\alpha,i})-\frac{1}{N}\sum_{j=1}^N\nabla V_{\alpha\beta}^\eps(x_{\alpha,i}-x_{\beta,j})\big|^2\,\d \mathbf{X}\\
=\,& \E\Big[\frac{1}{N}\sum_{\alpha,\beta=1}^n\sum_{i=1}^N\big|\nabla V_{\alpha\beta}^\eps\ast \t{f}_{\beta,\eps}\big(X_{\alpha,i}^\eps(t)\big)-\frac{1}{N}\sum_{j=1}^N\nabla V_{\alpha\beta}^\eps\big(X_{\alpha,i}^\eps(t)-X_{\beta,j}^\eps(t)\big)\big|^2\Big]\\
\leq&\, 2\E\Big[\frac{1}{N}\sum_{\alpha,\beta=1}^n\sum_{i=1}^N\big|\nabla V_{\alpha\beta}^\eps\ast \t{f}_{\beta,\eps}\big(X_{\alpha,i}^\eps(t)\big)-\nabla V_{\alpha\beta}^\eps\ast \t{f}_{\beta,\eps}\big(\t{X}_{\alpha,i}^\eps(t)\big)\big|^2\Big]\\
&+ 2\E\Big[\frac{1}{N}\sum_{\alpha,\beta=1}^n\sum_{i=1}^N\big|\nabla V_{\alpha\beta}^\eps\ast \t{f}_{\beta,\eps}\big(\t{X}_{\alpha,i}^\eps(t)\big)-\frac{1}{N}\sum_{j=1}^N\nabla V_{\alpha\beta}^\eps\big(\t{X}_{\alpha,i}^\eps(t)-\t{X}_{\beta,j}^\eps(t)\big)\big|^2\Big]\\
&+ 2\E\Big[\frac{1}{N}\sum_{\alpha,\beta=1}^n\sum_{i=1}^N\big|\frac{1}{N}\sum_{j=1}^N\nabla V_{\alpha\beta}^\eps\big(\t{X}_{\alpha,i}^\eps(t)-\t{X}_{\beta,j}^\eps(t)\big)-\frac{1}{N}\sum_{j=1}^N\nabla V_{\alpha\beta}^\eps\big(X_{\alpha,i}^\eps(t)-X_{\beta,j}^\eps(t)\big)\big|^2\Big]\\
=:& 2\big(J_1(t)+J_2(t)+J_3(t)\big).
\end{aligned}
$$
We will show $J_1$, $J_2$ and $J_3$ converge to 0 with some certain rate as $N\to\infty$. Firstly, we can apply Lemma \ref{LLN} to estimate $J_2$ by a law of large numbers estimate. For some $m_3\in\N$, some $\theta_3\in[0,\frac{1}{2})$ and any $t\in[0,T]$, it holds
\begin{align}\label{J2}
J_2(t)=\,&\E\bigg[\frac{1}{N}\sum_{\alpha,\beta=1}^n\sum_{i=1}^N\bigg|\nabla V_{\alpha\beta}^\eps\ast \t{f}_{\beta,\eps}\big(\t{X}_{\alpha,i}^\eps(t)\big)-\frac{1}{N}\sum_{j=1}^N\nabla V_{\alpha\beta}^\eps\big(\t{X}_{\alpha,i}^\eps(t)-\t{X}_{\beta,j}^\eps(t)\big)\bigg|^2I_{\mathcal{A}^N_{\theta_3, \nabla\mathbf{V}^{\eps}}}\bigg]\nonumber\\
&+\E\bigg[\frac{1}{N}\sum_{\alpha,\beta=1}^n\sum_{i=1}^N\bigg|\nabla V_{\alpha\beta}^\eps\ast \t{f}_{\beta,\eps}\big(\t{X}_{\alpha,i}^\eps(t)\big)-\frac{1}{N}\sum_{j=1}^N\nabla V_{\alpha\beta}^\eps\big(\t{X}_{\alpha,i}^\eps(t)-\t{X}_{\beta,j}^\eps(t)\big)\bigg|^2I_{(\mathcal{A}^N_{\theta_3, \nabla\mathbf{V}^{\eps}})^c}\bigg]\nonumber\\
\leq& C(n,T)\max_{\alpha,\beta}\|\nabla V_{\alpha\beta}^\eps\|_{L^\infty}^2\P(\mathcal{A}^N_{\theta_3, \nabla\mathbf{V}^{\eps}})+C(n)N^{-2\theta_3}\nonumber\\
\leq& C(n,m_3,\max_{\alpha,\beta}|a_{\alpha\beta}|,T)N^{(2 m_3+2)\ell(1+s)} N^{ m_3(2\theta_3-1)+1}+C(n)N^{-2\theta_3},
\end{align}
where we use the bound $\|\nabla V^\eps_{\alpha\beta}\|_{L^\infty}\lesssim N^{\ell(1+s)}$ by Remark \ref{fractional gradient}. To deal with $J_1$, we will use Proposition \ref{convergence in prob} to get
$$
\begin{aligned}
J_1(t)
=\,&\E\bigg[\frac{1}{N}\sum_{\alpha,\beta=1}^n\sum_{i=1}^N\bigg|\nabla V_{\alpha\beta}^\eps\ast \t{f}_{\beta,\eps}\big(X_{\alpha,i}^\eps(t)\big)-\nabla V_{\alpha\beta}^\eps\ast \t{f}_{\beta,\eps}\big(\t{X}_{\alpha,i}^\eps(t)\big)\bigg|^2I_{\C_\lambda(t)}\bigg]\\&+\E\bigg[\frac{1}{N}\sum_{\alpha,\beta=1}^n\sum_{i=1}^N\bigg|\nabla V_{\alpha\beta}^\eps\ast \t{f}_{\beta,\eps}\big(X_{\alpha,i}^\eps(t)\big)-\nabla V_{\alpha\beta}^\eps\ast \t{f}_{\beta,\eps}\big(\t{X}_{\alpha,i}^\eps(t)\big)\bigg|^2I_{\C_\lambda^c(t)}\bigg]\\
\leq & 4\sum_{\alpha,\beta=1}^n\|\nabla V_{\alpha\beta}^\eps\ast \t{f}_{\beta,\eps}(t)\|_{L^\infty}^2\P (\C_\lambda(t))\\&+\frac{1}{N}\sum_{\alpha,\beta=1}^n\sum_{i=1}^N\|\nabla^2 V_{\alpha\beta}^\eps\ast \t{f}_{\beta,\eps}(t)\|^2_{L^\infty}\E\big[|X_{\alpha,i}^\eps(t)-\t{X}_{\alpha,i}^\eps(t)|^2I_{\C_\lambda^c(t)}\big]\\\leq & C(n,T,\gamma)\max_{\alpha,\beta}\|\nabla V_{\alpha\beta}^\eps\ast \t{f}_{\beta,\eps}(t)\|_{L^\infty}^2 N^{-\gamma}+n^2\max_{\alpha,\beta}\|\nabla^2 V_{\alpha\beta}^\eps\ast \t{f}_{\beta,\eps}(t)\|^2_{L^\infty}N^{-2\lambda}.
\end{aligned}
$$
In the last line above, we have for $s\leq d-2$ by the uniform bound \eqref{uniform in eps} that
$$
\begin{aligned}
\sup_{t\in[0,T]}\big\|\nabla V^\eps_{\alpha\beta}\ast \t f_{\beta,\eps}\big\|_{L^\infty}\leq&C(\max_{\alpha,\beta}|a_{\alpha\beta}|)\sup_{t\in[0,T]}\Big(\big\|\frac{1}{|x|^{s+1}}\big|_{B_1}\ast \t{f}_{\beta,\eps}\big\|_{L^\infty}+\big\|\frac{1}{|x|^{s+1}}\big|_{B_1^c}\ast \t{f}_{\beta,\eps}\big\|_{L^\infty}\Big)\\\leq&C(\max_{\alpha,\beta}|a_{\alpha\beta}|,s)\sup_\eps\sup_{t\in[0,T]}\Big(\big\|\t{f}_{\beta,\eps}\big\|_{L^\infty}+\big\|\t{f}_{\beta,\eps}\big\|_{L^1}\Big)\leq C(\max_{\alpha,\beta}|a_{\alpha\beta}|,s,T).
\end{aligned}
$$
While we still can bound the Hessian as above for $s<d-2$:
$$\sup_{t\in[0,T]}\|\nabla^2 V^\eps_{\alpha\beta}\ast \t f_{\beta,\eps}\|_{L^\infty}\leq C(\max_{\alpha,\beta}|a_{\alpha\beta}|,s)\sup_\eps\big\|\t{f}_{\beta,\eps}\big\|_{L^\infty(0,T;L^1\cap L^\infty)}\leq C(\max_{\alpha,\beta}|a_{\alpha\beta}|,s,T),$$  we use directly $\nabla^2 V^\eps_{\alpha\beta}=a_{\alpha\beta}\nabla V\ast \nabla\chi^\eps$ for the case $s=d-2$ to obtain
\begin{align}\label{hessian}
&\sup_{t\in[0,T]}\big\|\nabla^2 V_{\alpha\beta}^\eps\ast \t{f}_{\beta,\eps}\big\|_{L^\infty}=\sup_{t\in[0,T]}\Big\|a_{\alpha\beta}\nabla  V\ast\nabla\chi^\eps\ast \t{f}_{\beta,\eps}\Big\|_{L^\infty}\\
\leq&C(\max_{\alpha,\beta}|a_{\alpha\beta}|)\sup_{t\in[0,T]}\Big(\Big\|\frac{1}{|x|^{d-1}}\big|_{B_1}\ast|\nabla\chi^\eps|\ast \t{f}_{\beta,\eps}\Big\|_{L^\infty}+\Big\|\frac{1}{|x|^{d-1}}\big|_{B_1^c}\ast|\nabla\chi^\eps|\ast \t{f}_{\beta,\eps}\Big\|_{L^\infty}\Big)\nonumber\\
\leq &C(\max_{\alpha,\beta}|a_{\alpha\beta}|)\sup_{t\in[0,T]}\Big(\big\||\nabla\chi^\eps|\ast \t{f}_{\beta,\eps}\big\|_{L^\infty}+\big\||\nabla\chi^\eps|\ast \t{f}_{\beta,\eps}\big\|_{L^1}\Big)\nonumber\\
\leq& \frac{C(\max_{\alpha,\beta}|a_{\alpha\beta}|,\|\chi\|_{W^{1,1}\cap W^{1,\infty} })}{\eps} \sup_\eps\big\|\t{f}_{\beta,\eps}\big\|_{L^\infty(0,T;L^1\cap L^\infty)}\nonumber\\
\leq  &C(\max_{\alpha,\beta}|a_{\alpha\beta}|,\|\chi\|_{W^{1,1}\cap W^{1,\infty} },T)N^\ell.\nonumber
\end{align}
Thus we conclude that
\begin{equation}\label{J1}
\sup_{t\in[0,T]}J_1(t)
\leq C(n,T,\gamma)  N^{-\gamma}+C(n,T)N^{2\ell-2\lambda}.
\end{equation}
\begin{remark}\label{higher regularity}
We are able to obtain $J_1\lesssim N^{-\gamma}+N^{-2\lambda}$ in the case $s<d-2$. And notice that if $\t{f}_{\beta,\eps}$ admits higher regularity as $W^{1,1}\cap W^{1,\infty}$, then we can put one gradient on $\t{f}_{\beta,\eps}$ \eqref{hessian} when $s=d-2$ to obtain $J_1\lesssim N^{-\gamma}+N^{-2\lambda}$.\end{remark}
For $J_3$ we again split in $\omega\in \C_\lambda$ and $\omega\in \C_\lambda^c$, such as
$$
\begin{aligned}
J_3(t)=\,&\E\bigg[\frac{1}{N}\sum_{\alpha,\beta=1}^n\sum_{i=1}^N\bigg|\frac{1}{N}\sum_{j=1}^N\nabla V_{\alpha\beta}^\eps\big(\t{X}_{\alpha,i}^\eps(t)-\t{X}_{\beta,j}^\eps(t)\big)-\frac{1}{N}\sum_{j=1}^N\nabla V_{\alpha\beta}^\eps\big(X_{\alpha,i}^\eps(t)-X_{\beta,j}^\eps(t)\big)\bigg|^2\bigg]\\
=\,&\E\bigg[\frac{1}{N}\sum_{\alpha,\beta=1}^n\sum_{i=1}^N\bigg|\frac{1}{N}\sum_{j=1}^N\nabla V_{\alpha\beta}^\eps\big(\t{X}_{\alpha,i}^\eps(t)-\t{X}_{\beta,j}^\eps(t)\big)-\frac{1}{N}\sum_{j=1}^N\nabla V_{\alpha\beta}^\eps\big(X_{\alpha,i}^\eps(t)-X_{\beta,j}^\eps(t)\big)\bigg|^2I_{\C_{\lambda}(t)}\bigg]\\+&\E\bigg[\frac{1}{N}\sum_{\alpha,\beta=1}^n\sum_{i=1}^N\bigg|\frac{1}{N}\sum_{j=1}^N\nabla V_{\alpha\beta}^\eps\big(\t{X}_{\alpha,i}^\eps(t)-\t{X}_{\beta,j}^\eps(t)\big)-\frac{1}{N}\sum_{j=1}^N\nabla V_{\alpha\beta}^\eps\big(X_{\alpha,i}^\eps(t)-X_{\beta,j}^\eps(t)\big)\bigg|^2I_{\C_{\lambda}^c(t)}\bigg]\\=:&J_{31}(t)+J_{32}(t).
\end{aligned}
$$
The first term $J_{31}$ can be estimated by the probability of the set $C_\lambda$ which is given by Proposition \ref{convergence in prob} in the following way
\begin{equation}\label{J31}
\begin{aligned}&\E\bigg[\frac{1}{N}\sum_{\alpha,\beta=1}^n\sum_{i=1}^N\bigg|\frac{1}{N}\sum_{j=1}^N\nabla V_{\alpha\beta}^\eps\big(\t{X}_{\alpha,i}^\eps(t)-\t{X}_{\beta,j}^\eps(t)\big)-\frac{1}{N}\sum_{j=1}^N\nabla V_{\alpha\beta}^\eps\big(X_{\alpha,i}^\eps(t)-X_{\beta,j}^\eps(t)\big)\bigg|^2I_{\C_{\lambda}(t)}\bigg]\\\leq & 4n^2
\max_{\alpha,\beta}\|\nabla V_{\alpha\beta}^\eps\|_{L^\infty}^2\P (\C_{\lambda}(t))\leq C(\max_{\alpha,\beta}|a_{\alpha\beta}|,n,\gamma,T) N^{2\ell(s+1)-\gamma}.
\end{aligned}
\end{equation}
In order to estimate $J_{32}$, we first remark that under the set $\C_{\lambda}^c(t)$, the assumption $\lambda>\ell$ implies that
$$
\big|(X_{\alpha,i}^\eps - X_{\beta,j}^\eps)-(\t X_{\alpha,i}^\eps -\t X_{\beta,j}^\eps)\big|(t)\leq 2\max_{\alpha,i}\big|X_{\alpha,i}^\eps - \t X_{\alpha,i}^\eps \big|(t)\leq 2N^{-\lambda}< 2N^{-\ell}= 2\eps.
$$
This allows us to use the auxiliary function $K^\eps$ defined in \eqref{auxiliary} to obtain the following estimate  
$$
\begin{aligned}
J_{32}&\leq C_3 \E\bigg[\frac{1}{N}\sum_{\alpha,\beta=1}^n\sum_{i=1}^N\Big|\frac{1}{N}\sum_{j=1}^N  K^\eps(\t X_{\alpha,i}^\eps -\t X_{\beta,j}^\eps)\left|(X_{\alpha,i}^\eps - X_{\beta,j}^\eps)-(\t X_{\alpha,i}^\eps -\t X_{\beta,j}^\eps)\right|\Big|^2I_{\C_{\lambda}^c}\bigg]\\\leq& C_3n^2\E\bigg[\max_{\alpha ,\beta,i}\Big|\Big(\frac{1}{N}\sum_{j=1}^N  K^\eps(\t X_{\alpha,i}^\eps -\t X_{\beta,j}^\eps)-K^\eps\ast \t f_{\beta,\eps}(\t X_{\alpha,i}^\eps)\Big)\big|(X_{\alpha,i}^\eps - X_{\beta,j}^\eps)-(\t X_{\alpha,i}^\eps -\t X_{\beta,j}^\eps)\big|\Big|^2I_{\C_{\lambda}^c}\bigg]\\&+C_3n^2\E\bigg[\max_{\alpha ,\beta,i}\Big|K^\eps\ast \t f_{\beta,\eps}(\t X_{\alpha,i}^\eps)\big|(X_{\alpha,i}^\eps - X_{\beta,j}^\eps)-(\t X_{\alpha,i}^\eps -\t X_{\beta,j}^\eps)\big|\Big|^2I_{\C_{\lambda}^c}\bigg]\\
\leq& 4C_3n^2N^{-2\lambda}\bigg(\E\bigg[\max_{\alpha ,\beta}\max_{i}\Big|\frac{1}{N}\sum_{j=1}^N  K^\eps(\t X_{\alpha,i}^\eps -\t X_{\beta,j}^\eps)-K^\eps\ast \t f_{\beta,\eps}(\t X_{\alpha,i}^\eps)\Big|^2\bigg]+\max_\beta \|K^\eps \ast \t f_{\beta,\eps}\|_{L^\infty}^2\bigg),    
\end{aligned}
$$
where the constant $C_3 =C_3 (d,s,\|\chi\|_{W^{2,1}\cap W^{2,\infty} },
\max_{\alpha,\beta}|a_{\alpha\beta}|)$ is the same as in \eqref{ineq C3}.  
Recalling the bound $\|K^\eps\|_{L^\infty}\leq  (4N^{\ell})^{s+2}$, we apply Lemma \ref{LLN} (the law of large numbers) with some $0\leq\theta_4<1/2$ and any $m_4\in\N$ to obtain 
$$
\begin{aligned}
&\E\bigg[\frac{1}{N}\sum_{\alpha,\beta=1}^n\sum_{i=1}^N\bigg|\frac{1}{N}\sum_{j=1}^N\nabla V_{\alpha\beta}^\eps\big(\t{X}_{\alpha,i}^\eps(t)-\t{X}_{\beta,j}^\eps(t)\big)-\frac{1}{N}\sum_{j=1}^N\nabla V_{\alpha\beta}^\eps\big(X_{\alpha,i}^\eps(t)-X_{\beta,j}^\eps(t)\big)\bigg|^2I_{\C_{\lambda} ^c(t)}\bigg]\\
\leq &C(n,m_4,T)N^{-2 \lambda }N^{(2m_4+2)\ell(s+2)+m_4(2\theta_4-1)+1}+C(n)N^{-2 \lambda }N^{-2\theta_4}+C (T,\ell,n)(1+\kappa\log N)^2N^{-2\lambda},
\end{aligned}
$$
where the last term is due to the bound \eqref{bound K subcoulomb} and \eqref{bound K coulomb}  which covers the cases $s< d-2$ $(\kappa=0)$ and $s=d-2$ $(\kappa=1)$. We remark that all the constants above also depend on the matrix $(a_{\alpha\beta})_{\alpha, \beta =1,\ldots, n}$ and the dimension $d$, which is omitted from now on to simplify the presentation. 

Together with \eqref{J31}, we infer that
\begin{equation}\label{J3}
\begin{aligned}
\sup_{t\in[0,T]}J_3(t)\leq& C(n,\gamma,T)N^{2\ell(s+1)-\gamma}+C(n,m_4,T)N^{-2 \lambda +m_4(\ell(2s+4)+2\theta_4-1)+2\ell(s+2)+1}\\&+C(n,T)N^{-2 \lambda -2\theta_4}+C(T,\ell,n)(1+\kappa\log N)^2N^{-2\lambda}.
\end{aligned}
\end{equation}
Combining the estimate \eqref{J1}, \eqref{J2} and \eqref{J3}, we come to 
\begin{equation}\label{HN}
\begin{aligned}
&\frac{\,\d}{\,\d t}\frac{H(f_{N,\eps}|\t{f}_{N,\eps})}{N}\leq  C\sup_{t\in[0,T]}\big(J_1(t)+J_2(t)+J_3(t)\big)
\leq C(n,T,\gamma)N^{-\gamma}+C(n,T)N^{2\ell-2\lambda} \\
&+ C(n,m_3,T)N^{m_3(\ell(2s+2)+2\theta_3-1)+\ell(2s+2)+1} +C(n)N^{-2\theta_3}+C(n,\gamma,T)N^{2\ell(s+1)-\gamma}\\&+C(n,m_4,T)N^{-2 \lambda +m_4(\ell(2s+4)+2\theta_4-1)+\ell(2s+4)+1}+C(n,T)N^{-2 \lambda -2\theta_4}\\&+C(T,\ell,n)(1+\kappa\log N)^2N^{-2\lambda}.
\end{aligned}
\end{equation}
If we assume $\theta_3$ and $\theta_4$ satisfy 
$$
0<\theta_3<\frac{1}{2}-\ell (s+1),\quad 0\leq\theta_4<\frac{1}{2}-\ell (s+2),$$
then we can take  $(m_3,m_4,\gamma)$ large enough to make the first, third, fifth and sixth terms in \eqref{HN} converge to 0 arbitrarily algebraically fast. In the view of the forth term $N^{-2\theta_3}$ above, we observe that by choosing $\theta_3$ close to $1/2-\ell (s+1)$, the slowest convergence rate in \eqref{HN} is $N^{2\ell - 2\lambda}$, which is indeed converging to zero due to the constraint \eqref{range lambda}, i.e.
$$\ell< \lambda<\frac{1}{2}-\ell (s+1).$$ 
Hence, we finally obtain  
\begin{align}\label{d_dtH}
    \frac{\,\d}{\,\d t}\frac{H(f_{N,\eps}|\t{f}_{N,\eps})}{N}\leq C(n,T,m_3,m_4,\ell,\gamma) N^{2\ell-2\lambda}.
\end{align}
In order to optimise the above  convergence rate, we take $\lambda>0$ close enough to $1/2-\ell(s+1)$.  Thus, for some $\varrho>0$ arbitrarily small, the relative entropy can be controlled by
$$
\sup_{t\in[0,T]}\frac{H\big(f_{N,\eps}(t)|\t{f}_{N,\eps}(t)\big)}{N}\leq \frac{C(T)}{N^{1-\ell(2s+4)-\varrho}},
$$
which concludes the proof of  Proposition \ref{RE convergence}. We remark that the constant above depends on $n$ and the choice of $(\ell, m_3, m_4, \gamma)$.
\begin{remark}[Improved convergence rate]\label{higher regularity 2}
Under the assumptions of Remark \ref{higher regularity}, i.e. $s< d-2$ or $\t{f}_{N,\eps}$ has higher order regularity, the leading order term in \eqref{HN}, which is arising from $J_1$, is no longer present. Indeed, under those assumptions, by choosing $\theta_3$ and $\lambda$ both close to $1/2-\ell(s+1)$ as well as taking $(m_3,m_4,\gamma)$ large enough, we can improve \eqref{d_dtH} by
$$
\frac{\,\d}{\,\d t}\frac{H(f_{N,\eps}|\t{f}_{N,\eps})}{N}\lesssim N^{-2\theta_3}+N^{-2\lambda}+(1+\kappa\log N)^2N^{-2\lambda},
$$
which finally leads to $$
\sup_{t\in[0,T]}\frac{H\big(f_{N,\eps}(t)|\t{f}_{N,\eps}(t)\big)}{N}\leq \frac{C(T)}{N^{1-\ell(2s+2)-\varrho}},
$$ 
for some $\varrho>0$ arbitrarily small.
\end{remark}
%$$0<\ell(s+3)\leq \lambda<\theta_3<\frac{1}{2}-\ell (s+2),\quad m,\gamma \text{ large enough}.$$
%Then, it holds $\ell(s+2)<\lambda$ and 

\section{Distance between intermediate PDEs and the limiting system}\label{PDE distance}
We present the proof of Proposition \ref{second step} in this section which follows by standard arguments, i.e. separating in short and long range of the involved potentials. For the reader's convenience we present the proof. Recall our limiting and intermediate PDE systems \eqref{aggregation-diffusion} and \eqref{intermediate} i.e.,
$$
\p_t \bar f_\alpha=\sum_{\beta=1}^{n}a_{\alpha\beta}\div\big(\bar f_\alpha\nabla V\ast \bar f_\beta\big)+\sigma_{\alpha}\Delta  \bar f_\alpha,
$$
and
$$
\p_t \t f_{\alpha,\eps}  = \sum_{\beta =1}^n\div\big(\t f_{\alpha,\eps}  \nabla V_{\alpha\beta}^\eps \ast \t f_{\beta,\eps}\big) +\sigma_{\alpha} \Delta \t f_{\alpha,\eps}.
$$
As assumed in Theorem \ref{POC} as well as Proposition \ref{second step}, the  solutions of \eqref{aggregation-diffusion} and \eqref{intermediate} satisfy that
$$
\bar f_\alpha \in L^\infty(0,T;L^1\cap L^\infty)\cap L^2(0,T;H^1),\quad \alpha=1,\ldots,n
$$
and
$$
\t f_{\alpha,\eps} \in L^\infty(0,T;L^1\cap L^\infty)\cap L^2(0,T;H^1),\quad \alpha=1,\ldots,n
$$
where the norm 
$ 
\|\t f_{\alpha,\eps}\|_{L^\infty(0,T;L^1\cap L^\infty)\cap L^2(0,T;H^1)}
$ is uniformly in $\eps$.

\begin{proof}[Proof of Proposition \ref{second step} ]
Let $h_\alpha:=\t f_{\alpha,\eps}-\bar f_\alpha$, which satisfies
\begin{equation}\label{difference}
\begin{aligned}
\p_t h_\alpha=\,&\sigma_\alpha \Delta h_\alpha+ \sum_{\beta=1}^n \div \Big(h_\alpha\big(\nabla V_{\alpha\beta}^\eps \ast \t f_{\beta,\eps}\big)+\bar f_\alpha \big(\nabla V^\eps_{\alpha\beta}\ast  h_\beta \big)+\bar f_\alpha \big(\nabla V^\eps_{\alpha\beta} \ast \bar f_\beta-a_{\alpha\beta}\nabla V\ast \bar f_\beta\big)\Big).  
\end{aligned}
\end{equation}
Multiplying \eqref{difference} by $h_\alpha$ itself and then integrating with respect to the spacial variable, we get 
\begin{equation}\label{I,II,III}
\frac{1}{2}\frac{\,\d }{\,\d t}\|h_\alpha\|_{L^2(\R^d)}^2+\sigma_\alpha\|\nabla h_\alpha\|_{L^2(\R^d)}^2=I+II+III,
\end{equation}
where 
$$
I=-\sum_{\beta=1}^n \int_{\R^d} h_\alpha\nabla h_\alpha \cdot \nabla V_{\alpha\beta}^\eps \ast \t f_{\beta,\eps}\,\d x,
$$
$$
II=-\sum_{\beta=1}^n \int_{\R^d} \bar f_\alpha\nabla h_\alpha \cdot \nabla V_{\alpha\beta}^\eps \ast h_{\beta}\,\d x,
$$
and
$$
III=-\sum_{\beta=1}^n \int_{\R^d} \bar f_\alpha\nabla h_\alpha \cdot \big(\nabla V^\eps_{\alpha\beta} \ast \bar f_\beta-a_{\alpha\beta}\nabla V\ast \bar f_\beta\big)\,\d x.
$$
We estimate $I,II$ and $III$ separately. For $I$, it holds by assumption \eqref{uniform in eps} that 
\begin{equation}\label{ineq I}
\begin{aligned}
|I|\leq& \sum_{\beta=1}^n \|\nabla h_\alpha\|_{L^2}\| h_\alpha\|_{L^2}\|\nabla V_{\alpha\beta}^\eps \ast \t f_{\beta,\eps}\|_{L^\infty}\\\leq& \max_{\alpha,\beta} |a_{\alpha\beta}|\|\nabla h_\alpha\|_{L^2}\| h_\alpha\|_{L^2}\sum_{\beta=1}^n\|\nabla V\ast \t f_{\beta,\eps}\|_{L^\infty}\\\leq&C(\max_{\alpha,\beta}|a_{\alpha\beta}|,n) \|\nabla h_\alpha\|_{L^2}\| h_\alpha\|_{L^2}\max_\beta\sup_\eps\big(\|\t f_{\beta,\eps}\|_{L^\infty}+\|
\t f_{\beta,\eps}\|_{L^1}\big)\\
\leq& C(\max_{\alpha,\beta}|a_{\alpha\beta}|,n,T) \|\nabla h_\alpha\|_{L^2}\| h_\alpha\|_{L^2}.
\end{aligned}
\end{equation}
The term $II$ satisfies
\begin{equation}\label{ineq II}
\begin{aligned}
|II|\leq&\,\sum_{\beta=1}^n \|\nabla h_\alpha\|_{L^2}\|\bar f_\alpha\nabla V^\eps_{\alpha\beta} \ast h_\beta\|_{L^2}\\\leq&\,  \sum_{\beta=1}^n \|\nabla h_\alpha\|_{L^2}\Big(\|\bar f_\alpha\nabla V^\eps_{\alpha\beta}|_{B_1} \ast h_\beta\|_{L^2}+\|\bar f_\alpha\nabla V^\eps_{\alpha\beta}|_{B_1^c} \ast h_\beta\|_{L^2} \Big)
\\\leq&\max_{\alpha,\beta}|a_{\alpha\beta}|  \sum_{\beta=1}^n \|\nabla h_\alpha\|_{L^2}\Big(\|\bar f_\alpha\|_{L^\infty}\|\nabla V|_{B_1}\ast\chi^\eps \ast h_\beta\|_{L^2}+\|\bar f_\alpha\|_{L^2}\|\nabla V|_{B_1^c}\ast\chi^\eps \ast h_\beta\|_{L^\infty}\Big) 
\\
\leq&\max_{\alpha,\beta}|a_{\alpha\beta}|\sum_{\beta=1}^n \|\nabla h_\alpha\|_{L^2}\Big(\|\bar f_\alpha\|_{L^\infty}\|\nabla V|_{B_1}\ast\chi^\eps\|_{L^1}\|  h_\beta\|_{L^2}+\|\bar f_\alpha\|_{L^2}\|\nabla V|_{B_1^c}\ast\chi^\eps \|_{L^\infty}\| h_\beta\|_{L^1}\Big)\\
\leq &C(\max_{\alpha,\beta}|a_{\alpha\beta}|,T)\|\nabla h_\alpha\|_{L^2}\sum_{\beta=1}^n(\| h_\beta\|_{L^2}+\| h_\beta\|_{L^1}),
\end{aligned}
\end{equation}
where we used in the last step the interpolation $$\sup_{t\in[0,T]}\|\bar f_\alpha\|_{L^2}\leq \sup_{t\in[0,T]}(\|\bar f_\alpha\|_{L^1}\|\bar f_\alpha\|_{L^\infty})^{1/2}\leq C(T).$$ 
Let us estimate the third term $III$ now. When $s<d-2$, the $L^\infty$-norm $\|\nabla^2 V \ast \bar f_\beta\|_{L^\infty}$ is bounded for almost any $t\in[0,T]$. In that case, the term $III$ then satisfies
\begin{equation}\label{ineq III sub}
\begin{aligned}
|III|\leq &\sum_{\beta=1}^n \|\nabla h_\alpha\|_{L^2}\| \bar f_\alpha\|_{L^2}\|\nabla V^\eps_{\alpha\beta} \ast \bar f_\beta-a_{\alpha\beta}\nabla V\ast \bar f_\beta\|_{L^\infty}\\
\leq &\max_{\alpha,\beta}|a_{\alpha\beta}| \sum_{\beta=1}^n \|\nabla h_\alpha\|_{L^2}\| \bar f_\alpha\|_{L^2}\Big\|\int_{\R^d}\chi^\eps(y)\big(\nabla V \ast \bar f_\beta(x-y)-\nabla V\ast \bar f_\beta(x)\big)\,\d y\Big\|_{L^\infty}
\\
\leq &\max_{\alpha,\beta}|a_{\alpha\beta}|\sum_{\beta=1}^n \|\nabla h_\alpha\|_{L^2}\| \bar f_\alpha\|_{L^2}\|\|\nabla^2 V \ast \bar f_\beta\|_{L^\infty}\int_{\R^d}\chi^\eps(y)|y|\,\d y
\\
\leq &C(\max_{\alpha,\beta}|a_{\alpha\beta}|)\sum_{\beta=1}^n \|\nabla h_\alpha\|_{L^2}\| \bar f_\alpha\|_{L^2}\big(\|\bar f_\beta\|_{L^\infty}+\|\bar f_\beta\|_{L^1}\big)\int_{\R^d}\chi^\eps(y)|y|\,\d y\\
\leq &\,\eps nC(\max_{\alpha,\beta}|a_{\alpha\beta}|,T) \|\nabla h_\alpha\|_{L^2}.%\| \bar f_\alpha\|_{L^2}
\end{aligned}
\end{equation}
When $s=d-2$, we use the decomposition $V=V_{in}+V_{out}:=Vw+V(1-w)$ with smooth function $w$ satisfying
$$
w(x)=\left\{\begin{array}{cc}
    1  &\quad x\in B_1 \\
    0\leq w(x)\leq 1&\quad x\in B_2\cap B_1^c\\
     0 &\quad x\in B_2^c.
\end{array}\right. 
$$
It holds that $\nabla V_{in}|_{B_1}=\nabla (Vw)|_{B_1}=\nabla V|_{B_1}$ and $ \nabla V_{out}|_{B_2^c}=\nabla (V(1-w))|_{B_2^c}=\nabla V|_{B_2^c}$, which implies $\nabla V_{in}\in L^1$ and $\nabla V_{out}\in L^2$. Then, it satisfies
\begin{equation}\label{ineq V V'}
\begin{aligned}
|III|\leq &\sum_{\beta=1}^n \|\nabla h_\alpha\|_{L^2}\big\| \bar f_\alpha\big(a_{\alpha\beta}\nabla V\ast \chi^\eps \ast \bar f_\beta-a_{\alpha\beta}\nabla V\ast \bar f_\beta\big)\big\|_{L^2}\\\leq &\max_{\alpha,\beta}|a_{\alpha\beta}|\sum_{\beta=1}^n \|\nabla h_\alpha\|_{L^2}\big\| \bar f_\alpha\big( V\ast \chi^\eps \ast\nabla \bar f_\beta- V\ast \nabla \bar f_\beta\big)\big\|_{L^2}\\
\leq &\max_{\alpha,\beta}|a_{\alpha\beta}|\sum_{\beta=1}^n \|\nabla h_\alpha\|_{L^2}\Big(\big\| \bar f_\alpha\big(( V_{in}^\eps - V_{in})\ast \nabla \bar f_\beta\big)\big\|_{L^2}+\big\| \bar f_\alpha\big( ( V_{out}^\eps - V_{out})\ast \nabla \bar f_\beta\big)\big\|_{L^2}\Big)\\
\leq &\max_{\alpha,\beta}|a_{\alpha\beta}|\sum_{\beta=1}^n \|\nabla h_\alpha\|_{L^2}\Big(\| \bar f_\alpha\|_{L^\infty}\| V_{in}^\eps - V_{in}\|_{L^1}+\| \bar f_\alpha\|_{L^2}\| V_{out}^\eps - V_{out}\|_{L^2}\Big)\| \nabla \bar f_\beta\|_{L^2},
\end{aligned}
\end{equation}
where we used the notation $V_{in}^\eps:=V_{in}\ast \chi^\eps$ and $V_{out}^\eps:=V_{out}\ast \chi^\eps$.  
The terms which involve $V_{in}$ and $V_{out}$ can be estimated by using Fubini's theorem as follows
\begin{equation}\label{V'}
\begin{aligned}
\| V_{in}^\eps - V_{in}\|_{L^1}\leq &\int_{\R^d} \int_{\R^d} |V_{in}(x-y)-V_{in}(x)|\chi^\eps(y)\,\d y\,\d x\\
= & \int_{\R^d} \int_{\R^d} \left|\int_0^1 \nabla V_{in}(x-(1-\iota) y)\,\d \iota\cdot y\right|\chi^\eps(y)\,\d y\,\d x\\
\leq & \int_{\R^d}  \int_0^1 \int_{\R^d} \left|\nabla V_{in}(x-(1-\iota) y)\right|\,\d x\, \d \iota \,|y|\chi^\eps(y)\,\d y\\
\leq & \, C\eps\|\nabla V_{in}\|_{L^1};
\end{aligned}
\end{equation}
similarly, due to H\"{o}lder's inequality, we have
\begin{equation}\label{V''}
\begin{aligned}
\| V_{out}^\eps - V_{out}\|_{L^2}^2\leq &\int_{\R^d} \int_{\R^d} |V_{out}(x-y)-V_{out}(x)|^2\chi^\eps(y)\,\d y\,\d x\\
\leq & \int_{\R^d}  \int_0^1 \int_{\R^d} \left|\nabla V_{out}(x-(1-\iota) y)\right|^2\,\d x\, \,\d \iota |y|^2\chi^\eps(y)\,\d y\\
\leq &\,C \eps^2\|\nabla V_{out}\|_{L^2}^2.
\end{aligned}
\end{equation}
So, we obtain that for $s=d-2$ 
\begin{equation}\label{ineq III cob}
|III|\leq C(\max_{\alpha,\beta}|a_{\alpha\beta}|,T) \sum_{\beta=1}^n\eps\|\nabla h_\alpha\|_{L^2}\|\nabla \bar f_\beta\|_{L^2}.
\end{equation}
Thus, for $s\leq d-2$, we combine \eqref{ineq III sub} and \eqref{ineq III cob} to obtain
\begin{equation}\label{ineq III}
 |III|\leq \eps C(\max_{\alpha,\beta}|a_{\alpha\beta}|,T) \sum_{\beta=1}^n\|\nabla h_\alpha\|_{L^2}\big(\|\nabla \bar f_\beta\|_{L^2}+1\big).   
\end{equation}
By \eqref{ineq I}, \eqref{ineq II} and \eqref{ineq III}  using Cauchy-Schwarz inequality, it yields for some small $\delta$ such that $\delta\leq\sigma_\alpha$,
$$
\begin{aligned}
&|I|+|II|+|III|\\\leq& \delta \|\nabla h_\alpha\|_{L^2}^2+ C(\delta,\max_{\alpha,\beta}|a_{\alpha\beta}|,n,T) (\eps^2\sum_{\beta=1}^n\big(\|\nabla \bar f_\beta\|_{L^2}^2+1\big)+ \| h_\alpha\|_{L^2}^2+\sum_{\beta=1}^n (\| h_\beta\|_{L^2}^2+\| h_\beta\|_{L^1}^2)).
\end{aligned}
$$
Summing up all species of \eqref{I,II,III}, while omitting the dependence on parameters $\min_\alpha\sigma_\alpha$, $\max_{\alpha,\beta}|a_{\alpha\beta}|$ and $n$, gives us the following estimate
\begin{equation}\label{L2L1}
\begin{aligned}
\frac{\,\d }{\,\d t}\sum_{\alpha=1}^n\|h_\alpha\|_{L^2}^2\leq  &C(T)\Big[\eps^2\sum_{\alpha=1}^n(\|\nabla \bar f_\alpha\|_{L^2}^2+1)+\sum_{\alpha=1}^n(\|h_\alpha\|_{L^2}^2+\|h_\alpha\|_{L^1}^2)\Big]\\\leq  &C(T)\Big[\eps^2\sum_{\alpha=1}^n(\|\nabla \bar f_\alpha\|_{L^2}^2+1)+\sum_{\alpha=1}^n(\|h_\alpha\|_{L^2}^2+H(\t f_{\alpha,\eps}|\bar f_\alpha))\Big], 
\end{aligned}
\end{equation}
where we applied Csiszár-Kullback-Pinsker inequality \eqref{CKP} to control $L^1$-norm by relative entropy.

\medskip

Now we compute the evolution of the relative entropy reads as
$$
\begin{aligned}
&\frac{\,\d}{\,\d t}H(\t f_{\alpha,\eps}|\bar f_\alpha)=  \int_{\R^d}  \p_t\t f_{\alpha,\eps}\log \frac{\t f_{\alpha,\eps}}{\bar f_\alpha}\,\d x-\int_{\R^d} \p_t\bar f_{\alpha}\frac{\t f_{\alpha,\eps}}{\bar f_\alpha}\,\d x\\
=\,&-\sigma_\alpha\int_{\R^d} \t f_{\alpha,\eps}\Big|\nabla\log\frac{\t f_{\alpha,\eps}}{\bar f_\alpha}\Big|^2\,\d x+\sum_{\beta=1}^n \int_{\R^d} \t f_{\alpha,\eps}\nabla\log\frac{\t f_{\alpha,\eps}}{\bar f_{\alpha}}\cdot \big(a_{\alpha\beta}\nabla V\ast \bar f_{\beta}-\nabla V^\eps_{\alpha\beta}\ast \t f_{\beta,\eps}\big)\,\d x,
\end{aligned}
$$
Then Cauchy-Schwarz inequality implies that 
\begin{equation}\label{RE estimate}
\begin{aligned}
&\frac{\,\d}{\,\d t}H(\t f_{\alpha,\eps}|\bar f_\alpha)\leq  C(\sigma_\alpha)\sum_{\beta=1}^n \int_{\R^d} \t f_{\alpha,\eps} \big|a_{\alpha\beta}\nabla V\ast \bar f_{\beta}-\nabla V^\eps_{\alpha\beta}\ast \t f_{\beta,\eps}\big|^2\,\d x\\
\leq&C(\sigma_\alpha,\max_{\alpha,\beta}|a_{\alpha\beta}|)\sum_{\beta=1}^n \int_{\R^d} \t f_{\alpha,\eps} \big|\nabla V\ast \bar f_{\beta}-\nabla V\ast\t f_{\beta,\eps}+\nabla V\ast\t f_{\beta,\eps}-\nabla V\ast \chi^\eps\ast \t f_{\beta,\eps}\big|^2\,\d x\\
\leq&C(\sigma_\alpha,\max_{\alpha,\beta}|a_{\alpha\beta}|)\sum_{\beta=1}^n \bigg[\int_{\R^d} \t f_{\alpha,\eps} \big|\nabla V\ast\t f_{\beta,\eps}-\nabla V\ast \chi^\eps\ast \t f_{\beta,\eps}\big|^2\d x\\&\qquad\qquad\qquad\qquad\qquad\qquad\qquad+\int_{\R^d} \t f_{\alpha,\eps} \big|\nabla V\ast \bar f_{\beta}-\nabla V\ast\t f_{\beta,\eps}|^2\,\d x\bigg].
\end{aligned}
\end{equation}
We notice that, when $s<d-2$, it holds 
$$
\begin{aligned}
&\int_{\R^d} \t f_{\alpha,\eps}(x)\Big|\int_{\R^d}\chi^\eps(y)\big(\nabla V \ast \t f_{\beta,\eps}(x-y)-\nabla V\ast \t f_{\beta,\eps}(x)\big)\,\d y\Big|^2    \,\d x\\
\leq &\int_{\R^d} \t f_{\alpha,\eps}(x)\int_{\R^d}\chi^\eps(y)\big|\nabla V \ast \t f_{\beta,\eps}(x-y)-\nabla V\ast \t f_{\beta,\eps}(x)\big|^2\,\d y    \,\,\d x\\
\leq & \|\nabla^2 V^\eps\ast \t f_{\beta,\eps}\|_{L^\infty}^2\int_{\R^d} \chi^\eps(y)|y|^2\,\d y\leq C(T) \eps^2.
\end{aligned}
$$
While for the case $s=d-2$, we use the same decomposition $V=V_{in}+V_{out}$ as in \eqref{ineq V V'} again to obtain
$$
\begin{aligned}
&\sum_{\beta=1}^n \int_{\R^d} \t f_{\alpha,\eps} \big|\nabla V\ast\t f_{\beta,\eps}-\nabla V\ast \chi^\eps\ast \t f_{\beta,\eps}\big|^2\d x \\
\leq &C\sum_{\beta=1}^n \bigg[\int_{\R^d} \t f_{\alpha,\eps} \big|( V_{in}^\eps - V_{in})\ast \nabla \t f_{\beta,\eps}\big|^2\d x+ \int_{\R^d} \t f_{\alpha,\eps} \big|( V_{out}^\eps - V_{out})\ast \nabla \t f_{\beta,\eps}\big|^2\d x\bigg]\\
\leq &C\sum_{\beta=1}^n \bigg[  \| \t f_{\alpha,\eps}\|_{L^\infty}\big\|( V_{in}^\eps - V_{in})\ast \nabla \t f_{\beta,\eps}\big\|^2_{L^2}+ \| \t f_{\alpha,\eps}\|_{L^1}\big\|( V_{out}^\eps - V_{out})\ast \nabla \t f_{\beta,\eps}\big\|^2_{L^\infty}\bigg]\\
\leq &C\sum_{\beta=1}^n \bigg[\| \t f_{\alpha,\eps}\|_{L^\infty}\| V_{in}^\eps - V_{in}\|_{L^1}^2\| \nabla\t f_{\beta,\eps}\|_{L^2}^2+\| \t f_{\alpha,\eps}\|_{L^1}\| V_{out}^\eps - V_{out}\|_{L^2}^2\| \nabla \t f_{\beta,\eps}\|_{L^2}^2\bigg]\\
\leq & C(T)\eps^2\sum_{\beta=1}^n\|\nabla \t f_{\beta,\eps}\|_{L^2}^2,
\end{aligned}
$$
where the last inequality comes from \eqref{uniform in eps}, \eqref{V'} and \eqref{V''}.
That is to say, when $s\leq d-2$, 
\begin{equation}\label{ineq H1}
\sum_{\beta=1}^n \int_{\R^d} \t f_{\alpha,\eps} \big|\nabla V\ast\t f_{\beta,\eps}-\nabla V\ast \chi^\eps\ast \t f_{\beta,\eps}\big|^2\d x\leq C(T)\eps^2\sum_{\beta=1}^n(\|\nabla \t f_{\beta,\eps}\|_{L^2}^2+1).  
\end{equation}
The second term on the right-hand side of \eqref{RE estimate} can be estimated for each $\beta=1,\ldots,n$ as follows
\begin{equation}\label{ineq H2}
\begin{aligned}
\int_{\R^d} \t f_{\alpha,\eps} &\big|\nabla V\ast \bar f_{\beta}-\nabla V\ast\t f_{\beta,\eps}|^2\,\d x=\int_{\R^d} \t f_{\alpha,\eps} \big|\nabla V\ast h_\beta|^2\\
\leq & 2\int_{\R^d} \t f_{\alpha,\eps}\big|\nabla V|_{B_1}\ast h_\beta|^2\,\d x+2\int_{\R^d} \t f_{\alpha,\eps}\big|\nabla V|_{B_1^c}\ast h_\beta|^2\,\d x\\
\leq &2\|\t f_{\alpha,\eps}\|_{L^\infty}\|\nabla V|_{B_1}\ast h_\beta\|_{L^2}^2+2\int_{\R^d} \t f_{\alpha,\eps}(x)\Big|\int_{\R^d}\nabla V|_{B_1^c}(y) h_\beta(x-y)\,\d y\Big|^2\,\d x\\
\leq&2\|\t f_{\alpha,\eps}\|_{L^\infty}\|\nabla V|_{B_1}\|_{L^1}^2\| h_\beta\|_{L^{2}}^2+2\| \t f_{\alpha,\eps}\|_{L^1}\|\nabla V|_{B_1^c}\|_{L^\infty}^2 \|h_\beta\|^2_{L^1}\\
\leq &\,C(T)\big(\|h_\beta\|^2_{L^2}+\|h_\beta\|^2_{L^1}\big).
\end{aligned}
\end{equation}
We plug in \eqref{ineq H1} and \eqref{ineq H2} into \eqref{RE estimate} to get the following estimate of the relative entropy
\begin{equation*}
 \frac{\,\d}{\,\d t}H(\t f_{\alpha,\eps}|\bar f_\alpha)\leq C(T)\Big[ \eps^2\sum_{\alpha=1}^n\|\nabla \t f_{\alpha,\eps}\|_{L^2}^2+\sum_{\beta=1}^n\big(\|h_\beta\|_{L^2}^2+H(\t f_{\beta,\eps}|\bar f_\beta)\big)\Big].
\end{equation*}
Recall \eqref{L2L1} and omit the dependence on some parameters to deduce that
$$
\frac{\,\d }{\,\d t}\sum_{\alpha=1}^n\|h_\alpha\|_{L^2}^2\leq C(T)\Big[  \eps^2\sum_{\alpha=1}^n\big(\|\nabla \bar f_\beta\|_{L^2}^2+1\big)+\sum_{\alpha=1}^n(\|h_\alpha\|_{L^2}^2+H(\t f_{\alpha,\eps}|\bar f_\alpha))\Big].
$$
Adding these two estimates together gives us
$$
\begin{aligned}
\frac{\,\d }{\,\d t}\sum_{\alpha=1}^n& \Big(\|h_\alpha\|_{L^2}^2+H(\t f_{\alpha,\eps}|\bar f_\alpha)\Big)\\\leq& C(T)\Big[\eps^2\sum_{\alpha=1}^n\left(\|\nabla \t f_{\alpha,\eps}\|_{L^2}^2+\|\nabla \bar f_{\alpha}\|_{L^2}^2+1\right)+\sum_{\alpha=1}^n\Big(\|h_\alpha\|_{L^2}^2+H(\t f_{\alpha,\eps}|\bar f_\alpha)\Big)\Big]. 
\end{aligned}
$$
Recall the fact that $\t f_{\alpha,\eps}, \bar f_\alpha\in L^2(0,T,H^1)$ with bounds independent of $\eps$, namely
$$
\sup_\eps\sum_{\alpha=1}^n\int_0^T\left(\|\nabla \t f_{\alpha,\eps}\|_{L^2}^2+\|\nabla \bar f_{\alpha}\|_{L^2}^2+1\right)\d t<C(T),
$$
and the same initial data as $\t f_{\alpha,\eps}(0)=\bar f_{\alpha}(0)=\bar f_\alpha^0$.  We apply Gronwall's inequality to obtain that, for any $\alpha=1, \ldots,n$
\begin{equation*}
\sup_{t\in[0,T]}\big(\|h_\alpha\|_{L^2}^2+H(\t f_{\alpha,\eps}|\bar f_\alpha)\big)\leq C(T)\eps^2,
\end{equation*}
which concludes the proof of Proposition \ref{second step}.
\end{proof}

\begin{remark}\label{remark:logPDE}
    We remark that the above proof does not easily extend to $d=2$, $V(x)= \log|x|$. In particular, the decomposition in \eqref{ineq V V'} in $V_{in}$ and $V_{out}$ with corresponding regularity does not hold due to the slow decay at infinity of $\nabla V = |x|^{-1}$.
\end{remark}

\section{Global well-posedness of aggregation-diffusion system}\label{pde analysis}
\subsection{Global-in-time existence of solutions of (\ref{aggregation-diffusion})}\label{existence subsection}
To prove the existence part of Theorem \ref{well-possedness}, we focus first on getting uniform-in-time a priori $L^p$-estimate of the intermediate system \eqref{intermediate} under the smallness assumptions for any $p>1$, and then obtaining the global-in-time $L^\infty$ bound by using the mild formulation. These uniform-in-$\eps$ bounds imply that there is a weak solution of the limiting system \eqref{aggregation-diffusion} by a well-known by now compactness argument and passing to the limit, see \cite{FP08,CG24} and the appendix for details.   

\begin{lemma}\label{uniformeps}
Under assumptions (H1)-(H5) and the smallness condition  \eqref{d+1}, for any $T>0$, the norm satisfies
$$ 
\sup_\eps\|\t f_{\alpha,\eps}\|_{L^\infty(0,T;L^1\cap L^\infty)\cap L^2(0,T;H^1)}<C(T),
$$ where the constant $C(T)$ is independent of  $\eps$.
\end{lemma}

\begin{proof} 
\

{\bf Step 1.- Propagation of $L^p$ bounds.} Fix $p>1$. We multiply the intermediate PDEs \eqref{intermediate} by $(\t f_{\alpha,\eps})^{p-1}$, and apply Cauchy-Schwarz inequality with some constant $\delta>0$ to get
$$
\begin{aligned}
\frac{1}{p}\frac{\,\d}{\,\d t}\int_{\R^d} (\t f_{\alpha,\eps})^{p}\d x&=\,-\sigma_\alpha(p-1)\int_{\R^d} (\t f_{\alpha,\eps})^{p-2}|\nabla \t f_{\alpha,\eps}|^2\d x\\&-\sum_{\beta=1}^n a_{\alpha\beta}(p-1)\int_{\R^d} (\t f_{\alpha,\eps})^{p-1}\nabla \t f_{\alpha,\eps}\cdot\nabla V\ast \chi^\eps\ast \t f_{\beta,\eps}\d x\\
\leq &-\sigma_\alpha(p-1)\int_{\R^d} (\t f_{\alpha,\eps})^{p-2}|\nabla \t f_{\alpha,\eps}|^2\d x\\&+\sum_{\beta=1}^n |a_{\alpha\beta}|(p-1)\Big[\delta\int_{\R^d} (\t f_{\alpha,\eps})^{p-2}|\nabla \t f_{\alpha,\eps}|^2\d x+ \frac{1}{4\delta} \int_{\R^d} (\t f_{\alpha,\eps})^{p}|\nabla V\ast\chi^\eps\ast \t f_{\beta,\eps}|^2\d x\Big],
\end{aligned}
$$
which yields the estimate
\begin{equation}\label{ineq dfdt}
\begin{aligned}
\frac{1}{p}\frac{\,\d}{\,\d t}\int_{\R^d} (\t f_{\alpha,\eps})^{p}\d x\leq &-\frac{4(p-1)}{p^2}\big(\sigma_\alpha-\sum_{\beta=1}^n |a_{\alpha\beta}|\delta\big)\int_{\R^d} \big|\nabla (\t f_{\alpha,\eps})^{\frac{p}{2}}\big|^2\d x\\&+  \frac{(p-1)}{4\delta}\sum_{\beta=1}^n |a_{\alpha\beta}|\int_{\R^d} (\t f_{\alpha,\eps})^{p}|\nabla V\ast\chi^\eps\ast \t f_{\beta,\eps}|^2\d x.
\end{aligned}
\end{equation}
In order to further estimate the second term on the right-hand side, we use H\"{o}lder's inequality which yields
$$
\begin{aligned}
\int_{\R^d} (\t f_{\alpha,\eps})^{p}|\nabla V\ast\chi^\eps\ast \t f_{\beta,\eps}|^2\d x\leq \left(\int_{\R^d} (\t f_{\alpha,\eps})^{\frac{pd}{d-2}}\d x\right)^{\frac{d-2}{d}}   \left(\int_{\R^d} |\nabla V\ast \chi^\eps\ast \t f_{\beta,\eps}|^{d}\d x\right)^{\frac{2}{d}}.
\end{aligned}
$$
By Gagliardo–Nirenberg–Sobolev's inequality, for $d\geq3$ it holds that
$$
\|(\t f_{\alpha,\eps})^{\frac{p}{2}}\|_{L^{\frac{2d}{d-2}}}\leq C_{GNS}\|\nabla (\t f_{\alpha,\eps})^{\frac{p}{2}}\|_{L^{2}}, 
$$
i.e.,
$$
\left(\int_{\R^d} (\t f_{\alpha,\eps})^{\frac{pd}{d-2}}\d x\right)^{\frac{d-2}{d}}\leq C_{GNS}^2\int_{\R^d} \big|\nabla (\t f_{\alpha,\eps})^{\frac{p}{2}}\big|^2\d x.
$$
On the other hand, Young's convolution inequality and Hardy-Littlewood-Sobolev's inequality (for example, see \cite{lieb1983sharp}), we obtain with $\frac{1}{d}-\frac{d-s}{d}+1=\frac{s+1}{d}$ the following bound
$$
\|\nabla V\ast \chi^\eps\ast \t f_{\beta,\eps}\|_{L^d}\leq\|\chi^\eps\|_{L^1}\|\nabla V\ast \t f_{\beta,\eps}\|_{L^d}\leq C_{HLS} \|\t f_{\beta,\eps}\|_{L^\frac{d}{d-s}}.$$ We  then arrive at
$$
\begin{aligned}
\int_{\R^d} (\t f_{\alpha,\eps})^{p}|\nabla V\ast\chi^\eps\ast \t f_{\beta,\eps}|^2\d x\leq &\,C_{HLS}^2C_{GNS}^2\|\t f_{\beta,\eps}\|_{L^{\frac{d}{d-s}}}^2\int |\nabla (\t f_{\alpha,\eps})^{\frac{p}{2}}|^2\d x\\
\leq &\, C_{HLS}^2C_{GNS}^2\|\t f_{\beta,\eps}\|_{L^p}^{\frac{2sp}{d(p-1)}}\int |\nabla (\t f_{\alpha,\eps})^{\frac{p}{2}}|^2\d x, 
\end{aligned}
$$
where we used interpolation for $p\geq \frac{d}{d-s}$ that
$$
\|\t f_{\beta,\eps}\|_{L^{\frac{d}{d-s}}}\leq \|\t f_{\beta,\eps}\|_{L^1}^{\frac{(d-s)/d-1/p}{1-1/p}}\|\t f_{\beta,\eps}\|_{L^p}^{\frac{1-(d-s)/d}{1-1/p}}=\|\t f_{\beta,\eps}\|_{L^p}^{\frac{sp}{d(p-1)}}.
$$
In order to estimate \eqref{ineq dfdt}, we put together previous bounds to conclude that
$$
\begin{aligned}
\frac{1}{p}\frac{\,\d}{\,\d t}\int_{\R^d} & (\t f_{\alpha,\eps})^p\d x\leq-\frac{4(p-1)}{p^2}\big(\sigma_\alpha-\sum_{\beta=1}^n |a_{\alpha\beta}|\delta\big)\int_{\R^d} \big|\nabla (\t f_{\alpha,\eps})^{\frac{p}{2}}\big|^2\d x\\&+  \frac{(p-1)}{4\delta}C_{HLS}^2C_{GNS}^2\sum_{\beta=1}^n |a_{\alpha\beta}|\|\t f_{\beta,\eps}\|_{L^p}^{\frac{2sp}{d(p-1)}}\int_{\R^d} |\nabla  (\t f_{\alpha,\eps})^{\frac{p}{2}}|^2\d x\\
& \leq \frac{p-1}{4\delta}\bigg(C_{HLS}^2C_{GNS}^2\sum_{\beta=1}^n |a_{\alpha\beta}|\|\t f_{\beta,\eps}\|_{L^p}^{\frac{2sp}{d(p-1)}}-\frac{16\delta\big(\sigma_\alpha-\sum_{\beta=1}^n |a_{\alpha\beta}|\delta\big)}{p^2}\bigg)\int_{\R^d} |\nabla  (\t f_{\alpha,\eps})^{\frac{p}{2}}|^2\d x.
\end{aligned}
$$
For any $\alpha=1, \ldots, n$, if we assume 
$$
\sum_{\beta=1}^{n} |a_{\alpha\beta}|\|\t f_{\beta,\eps}(0)\|_{L^p}^{\frac{2sp}{d(p-1)}}=\sum_{\beta=1}^{n} |a_{\alpha\beta}|\|\bar f_\beta^0\|_{L^p}^{\frac{2sp}{d(p-1)}}\leq \frac{16\delta\big(\sigma_\alpha-\sum_{\beta=1}^{n} |a_{\alpha\beta}|\delta\big)}{p^2C_{HLS}^2C_{GNS}^2},
$$
then for any time $t$, it holds
\begin{equation}\label{small initial delta}
\frac{1}{p}\frac{\,\d}{\,\d t}\int_{\R^d} (\t f_{\alpha,\eps})^p\d x\leq 0.
\end{equation}
That is to say, it is sufficient to assume
\begin{equation}\label{small initial}
\forall \alpha,\quad \sum_{\beta=1}^{n} |a_{\alpha\beta}|\|\bar f_\beta^0\|_{L^p}^{\frac{2sp}{d(p-1)}}\leq \frac{4\sigma_\alpha^2}{p^2C_{HLS}^2C_{GNS}^2\sum_{\beta=1}^{n} |a_{\alpha\beta}|},
\end{equation}
by taking suitable $\delta$ in \eqref{small initial delta}.
Then $\|\t f_{\alpha,\eps
}(t)\|_{L^p}$ decreases in time.
It can be seen that the larger $p$ is, the smaller initial data is required. In conclusion, for any fixed $p\geq\frac{d}{d-s}$, assume the $L^p$ norm of initial data is fairly small (or in other words diffusion $\sigma_\alpha$ is big enough) such that \eqref{small initial} holds, then the $L^p$-norm of the density of $\alpha$-th species is decreasing in time. While for $1<p<\frac{d}{d-s}$, $\|\t f_{\alpha,\eps}\|_{L^p}$ can be interpolated by $\|\t f_{\alpha,\eps}\|_{L^1}$ and $\|\t f_{\alpha,\eps}\|_{L^{d/(d-s)}}$. Then for any $p\geq 1$, under suitable assumption on initial data, we are able to show that the following bound 
$$
\sup_\eps\|\t f_{\alpha,\eps}\|_{L^{\infty}(0,\infty; L^1\cap L^p(\R^d))}<C, \quad \alpha=1,2,\ldots,n,
$$
holds.

{\bf Step 2.- $L^\infty$-bounds.} We rewrite $\t f_{\alpha,\eps}$ into the mild formulation such as
$$
\t f_{\alpha,\eps}(t)=\Gamma_\alpha(t)\ast \bar f_\alpha^0-\sum_{\beta=1}^n\int_0^t \Gamma_\alpha(t-\tau)\ast\div \big(\t f_{\alpha,\eps}(\tau)\nabla V^\eps_{\alpha\beta}\ast \t f_{\beta,\eps}(\tau)\big)\,\d \tau, 
$$
with heat kernel given as
$$
\Gamma_\alpha(t):=\frac{1}{(4\sigma_\alpha \pi t)^{d/2}}\exp (-\frac{|\cdot|^2}{4\sigma_\alpha t}).
$$
It yields the estimate
\begin{equation}\label{mild}
\begin{aligned}
\|\t f_{\alpha,\eps}(t)\|_{L^\infty}\leq&\,\|\Gamma_\alpha(t)\ast \bar f_\alpha^0\|_{L^\infty}+\sum_{\beta=1}^n\int_0^t \big\|\Gamma_\alpha(t-\tau)\ast\div (\t f_{\alpha,\eps}\nabla V^\eps_{\alpha\beta}\ast \t f_{\beta,\eps})\big\|_{L^\infty}\,\d \tau\\
\leq&\,\|\bar f_\alpha^0\|_{L^\infty}+\sum_{\beta=1}^n\int_0^t \big\|\nabla\Gamma_\alpha(t-\tau)\big\|_{L^{q'}}\big\| \t f_{\alpha,\eps}(\tau)\big\|_{L^{q}}\big\|\nabla V^\eps_{\alpha\beta}\ast \t f_{\beta,\eps}(\tau)\big\|_{L^\infty}\,\d \tau\\\leq&\,\|\bar f_\alpha^0\|_{L^\infty}+\sup_{\tau\in[0,t]}\Big(\big\| \t f_{\alpha,\eps}(\tau)\big\|_{L^{q}}\sum_{\beta=1}^n\big\|\nabla V^\eps_{\alpha\beta}\ast \t f_{\beta,\eps}(\tau)\big\|_{L^\infty}\Big)\int_0^t \big\|\nabla\Gamma_\alpha(t-\tau)\big\|_{L^{q'}}\,\d \tau,\\
\end{aligned}
\end{equation}
where $1/q+1/q'=1$. Since
$$
\big\|\nabla\Gamma_\alpha(t-\tau)\big\|_{L^{q'}}\leq \frac{C}{(t-\tau)^{\frac{d+1}{2}-\frac{d}{2q'}}},
$$
to make \eqref{mild} integrable with respect to time variable, we should let $1\leq q'<\frac{d}{d-1}$, i.e., $d<q\leq \infty$. Also, it holds by taking the cut-off that
\begin{equation}\label{ineq supercoulomb}
\begin{aligned}
\big\|\nabla V^\eps_{\alpha\beta}\ast \t f_{\beta,\eps}\big\|_{L^\infty} &\leq C(\max_{\alpha,\beta}|a_{\alpha\beta}|)   \left\|\frac{1}{|\cdot|^{s+1}}\ast \t f_{\beta,\eps}\right\|_{L^\infty}
\\&\leq  C(\max_{\alpha,\beta}|a_{\alpha\beta}|)  \left\|\frac{1}{|\cdot|^{s+1}}\big|_{B_1}\ast \t f_{\beta,\eps}\right\|_{L^\infty}+C(\max_{\alpha,\beta}|a_{\alpha\beta}|)\left\|\frac{1}{|\cdot|^{s+1}}\big|_{B_1^c}\ast \t f_{\beta,\eps}\right\|_{L^\infty}
\\&\leq   C(\max_{\alpha,\beta}|a_{\alpha\beta}|) \left\|\frac{1}{|\cdot|^{s+1}}\big|_{B_1}\right\|_{L^{r'}}\big\| \t f_{\beta,\eps}\big\|_{L^r}+C(\max_{\alpha,\beta}|a_{\alpha\beta}|) \left\|\frac{1}{|\cdot|^{s+1}}\big|_{B_1^c}\right\|_{L^\infty}\big\| \t f_{\beta,\eps}\big\|_{L^1},
\end{aligned}
\end{equation}
where it requires $\frac{1}{r}+\frac{1}{r'}=1$ and $r'(s+1)<d$, i.e., $r>\frac{d}{d-(s+1)}$. Due to assumption (H2), i.e., $s\leq d-2$, any $r>d$ is admissible. For simplicity, we take both $r$ and $q$ as $r=q=d+1$, then \eqref{mild} reads as 
\begin{equation}\label{ineq LinftyLp}
\begin{aligned}
&\|\t f_{\alpha,\eps}(t)\|_{L^\infty}\\\leq&\|\bar f_\alpha^0\|_{L^\infty}+C(\max_{\alpha,\beta}|a_{\alpha\beta}|)\sup_{\tau\in[0,t]}\Big(\big\| \t f_{\alpha,\eps}(\tau)\big\|_{L^{d+1}}\sum_{\beta=1}^n\big(\big\| \t f_{\beta,\eps}(\tau)\big\|_{L^{d+1}}+1\big)\Big)\int_0^t \frac{1}{(t-\tau)^{\frac{2d+1}{2d+2}}}\,\d \tau\\
\leq&\,\|\bar f_\alpha^0\|_{L^\infty}+C(t,\max_{\alpha,\beta}|a_{\alpha\beta}|)\big(\sum_{\beta=1}^n\big\| \t f_{\beta,\eps}\big\|_{L^\infty(0,t;L^{d+1})}^2+1\big).\\
\end{aligned}
\end{equation}
By assuming \eqref{small initial} with $p=d+1$ such that
\begin{equation}\label{d+1}
\forall \alpha,\quad \sum_{\beta=1}^{n} |a_{\alpha\beta}|\|\bar f_\beta^0\|_{L^{d+1}}^{\frac{2s(d+1)}{d^2}}\leq \frac{4\sigma_\alpha^2}{(d+1)^2C_{HLS}^2C_{GNS}^2\sum_{\beta=1}^n |a_{\alpha\beta}|},
\end{equation}
we are able to show for any species $\alpha$, 
$\big\| \t f_{\alpha,\eps}\big\|_{L^\infty(0,\infty;L^{d+1})}$ is bounded uniformly in $\eps$. We infer from \eqref{ineq LinftyLp} that,  for any $T>0$ and any $\alpha=1,\ldots,n$, it holds
$$
\sup_{t\in[0,T]}\big\| \t f_{\alpha,\eps}(t)\big\|_{L^\infty(\R^d)}<C(T),
$$
where we omit the dependence of $\max_{\alpha,\beta}|a_{\alpha\beta}|$.
Also, we can multiply \eqref{intermediate} by $\t f_{\alpha,\eps}$ to get
$$
\begin{aligned}
\sigma_\alpha\int_0^T\int_{\R^d} |\nabla \t f_{\alpha,\eps} |^2\,\d x\,\d t \leq&\, \frac{1}{2}\int_{\R^d} | \bar f_\alpha^0|^2\,\d x+\sum_{\beta=1}^n\int_0^T\int_{\R^d} \big|\t f_{\alpha,\eps}\nabla \t f_{\alpha,\eps}\cdot \nabla V_{\alpha\beta}^\eps\ast \t f_{\beta,\eps}\big|\,\d x\,\d t\\
\leq&\, \frac{1}{2}\int_{\R^d} | \bar f_\alpha^0|^2\,\d x+\delta\int_0^T\int_{\R^d} \big|\nabla \t f_{\alpha,\eps}\big|^2\d x\,\d t\\&+C(\delta)\sum_{\beta=1}^n\int_0^T\int_{\R^d} \big|\t f_{\alpha,\eps} \nabla V_{\alpha\beta}^\eps\ast \t f_{\beta,\eps}\big|^2\d x\,\d t,
\end{aligned}
$$
and we take $\delta$ small enough to get
$$
\int_0^T\int_{\R^d} |\nabla \t f_{\alpha,\eps} |^2\,\d x\,\d t\leq C(\sigma_\alpha)\int_{\R^d} | \bar f_\alpha^0|^2\,\d x+C(T,\sigma_\alpha)\sum_{\beta=1}^n\sup_{t\in[0,T]}\big(\big\|\nabla V_{\alpha\beta}^\eps\ast \t f_{\beta,\eps}(t)\big\|_{L^\infty}^2\big\|\t f_{\alpha,\eps}(t) \big\|_{L^2}^2\big).
$$
Since it holds  $$\sup_{t\in[0,T]}\|\t f_{\alpha,\eps}(t) \big\|_{L^2}\leq\sup_{t\in[0,T]}\Big(\|\t f_{\alpha,\eps}(t) \big\|_{L^1}\|\t f_{\alpha,\eps}(t) \big\|_{L^\infty}\Big)^\frac{1}{2} <C(T),$$ and
$$
\begin{aligned}
&\sup_{t\in[0,T]}\big\|\nabla V^\eps_{\alpha\beta}\ast \t f_{\beta,\eps}\big\|_{L^\infty}\\ \leq  &C\sup_{t\in[0,T]}\left\|\frac{1}{|\cdot|^{s+1}}\big|_{B_1}\right\|_{L^{1}}\big\| \t f_{\beta,\eps}\big\|_{L^\infty}+C\sup_{t\in[0,T]}\left\|\frac{1}{|\cdot|^{s+1}}\big|_{B_1^c}\right\|_{L^\infty}\big\| \t f_{\beta,\eps}\big\|_{L^1}\\\leq& C(T),
\end{aligned}
$$
where both constants are  uniformly  bounded in $\eps$ by boundedness of 
$\sup_\eps\big\| \t f_{\alpha,\eps}\big\|_{L^\infty(0,T;L^1\cap L^\infty)}$.
Thus, we have proved the uniform estimate of intermediate PDEs 
\eqref{intermediate}.
\end{proof}

\begin{remark}\label{lp2d}
The previous result does not hold for the Coulomb case $V(x)=\log |x|$ in two dimensions due to the Sobolev embeddings estimates.    
\end{remark}
\begin{remark}
The well-posedness result can be modified to apply to a larger range of the singularity. For $s\in (0,d-1)$, the term $\big\|\nabla V^\eps_{\alpha\beta}\ast \t f_{\beta,\eps}\big\|_{L^\infty}$ can be bounded by $C(\|\t f_{\beta,\eps}\|_{L^r}+1)$ for a large enough $r$ as shown in \eqref{ineq supercoulomb}. Then under some suitable smallness condition with $L^r$-norm, the solution of \eqref{aggregation-diffusion} is globally well-posed.
\end{remark}

We postpone the compactness argument in Appendix \ref{compactness}. Therein, we will prove that $\t f_{\alpha\,\eps}$ strongly converges to $\bar f_\alpha$ in $L^1(0,T;L^1(\R^d))$ for any species $\alpha$ (see Lemma \ref{L1convergence}), which allows us to pass to the limit, see Lemma \ref{pass to the limit}. Thus, we conclude that there exists a weak solution of the limiting PDE \eqref{aggregation-diffusion} defined in Theorem \ref{well-possedness} such that 
$$
\bar f_\alpha \in L^\infty(0,T;L^1\cap L^\infty)\cap L^2(0,T;H^1),\quad \alpha=1,2,\ldots,n.
$$

\subsection{Uniqueness of the solution of (\ref{aggregation-diffusion})}\label{uniqueness subsection}

We will show the uniqueness in this subsection. If $\bar f=(\bar f_1,\ldots,\bar f_n)$ and $\bar g=(\bar g_1,\ldots,\bar g_n)$ are two weak solutions satisfying
$$
\bar f_\alpha,\bar g_\alpha\in L^\infty(0,T;L^1\cap L^\infty)\cap L^2(0,T;H^1),\quad \alpha=1,2,\ldots,n,
$$
with the same initial data $\bar f^0_\alpha=\bar g^0_\alpha$.  We still use $h_\alpha$  to denote the difference as $h_\alpha=\bar g_\alpha-\bar f_\alpha$, which satisfies 
$$
\begin{aligned}
\p_t h_\alpha=\,& \sum_{\beta=1}^n a_{\alpha\beta}\div \Big(h_\alpha\big(\nabla V \ast \bar g_\beta\big)+\bar f_\alpha \big(\nabla V\ast  h_\beta \big)\Big)+\sigma_\alpha \Delta h_\alpha.  
\end{aligned}
$$
Multiplying by $h_\alpha$ and then integrating respect to spacial variable, we get 
\begin{equation}\label{haplha}
\frac{1}{2}\frac{\,\d }{\,\d t}\|h_\alpha\|_{L^2}^2+\sigma_\alpha\|\nabla h_\alpha\|_{L^2}^2=I+II,    
\end{equation}
where 
$$
I=-\sum_{\beta=1}^n a_{\alpha\beta}\int_{\R^d} h_\alpha\nabla h_\alpha \cdot \nabla V \ast \bar g_\beta\,\d x
\qquad \mbox{and} \qquad
II=-\sum_{\beta=1}^n a_{\alpha\beta}\int_{\R^d} \bar f_\alpha\nabla h_\alpha \cdot \nabla V\ast h_{\beta}\,\d x.
$$
We estimate $I$ and $II$ as follows:
$$
\begin{aligned}
|I|\leq  &\, C\sum_{\beta=1}^n \|\nabla h_\alpha\|_{L^2}\| h_\alpha\|_{L^2}\|\nabla V \ast \bar g_\beta\|_{L^\infty}\\\leq  & \,C\|\nabla h_\alpha\|_{L^2}\| h_\alpha\|_{L^2}\big(\|\nabla V|_{B_1} \ast\bar g_\beta\|_{L^\infty}+\|\nabla V|_{B_1^c} \ast\bar g_\beta\|_{L^1}\big)\\
\leq  &\, C(T) \|\nabla h_\alpha\|_{L^2}\| h_\alpha\|_{L^2};
\end{aligned}
$$
and 
$$
\begin{aligned}
|II|\leq  &\, C\sum_{\beta=1}^n \|\nabla h_\alpha\|_{L^2}\|\bar f_\alpha\nabla V \ast h_\beta\|_{L^2}\\\leq&\,  C\sum_{\beta=1}^n \|\nabla h_\alpha\|_{L^2}\Big(\|\bar f_\alpha\nabla V|_{B_1} \ast h_\beta\|_{L^2}+\|\bar f_\alpha\nabla V|_{B_1^c} \ast h_\beta\|_{L^2} \Big)
\\\leq  &\, C \sum_{\beta=1}^n \|\nabla h_\alpha\|_{L^2}\Big(\|\bar f_\alpha\|_{L^\infty}\|\nabla V|_{B_1} \ast h_\beta\|_{L^2}+\|\bar f_\alpha\|_{L^2}\|\nabla V|_{B_1^c} \ast h_\beta\|_{L^\infty}\Big) 
\\\leq  &\, C\sum_{\beta=1}^n \|\nabla h_\alpha\|_{L^2}\Big(\|\bar f_\alpha\|_{L^\infty}\|\nabla V|_{B_1}\|_{L^1}\|  h_\beta\|_{L^2}+\|\bar f_\alpha\|_{L^2}\|\nabla V|_{B_1^c} \|_{L^\infty}\| h_\beta\|_{L^1}\Big)\\
\leq  &\, C(T)\|\nabla h_\alpha\|_{L^2}\sum_{\beta=1}^n(\| h_\beta\|_{L^2}+\| h_\beta\|_{L^1}).
\end{aligned}
$$
Combining $I,II$ and using Cauchy's inequality, we get, for some small enough $\delta$, that
$$
|I|+|II|\leq \delta \|\nabla h_\alpha\|_{L^2}^2+ C(\delta,T) (\| h_\alpha\|_{L^2}^2+\sum_{\beta=1}^n (\| h_\beta\|_{L^2}^2+\| h_\beta\|_{L^1}^2)).
$$
Summing up all species in \eqref{haplha} together with Csiszár-Kullback-Pinsker inequality \eqref{CKP} gives us the estimate 
\begin{equation}\label{L2L1unique}
\frac{\,\d }{\,\d t}\left(\sum_{\alpha=1}^n\|h_\alpha\|_{L^2}^2)\leq  C(T) \sum_{\alpha=1}^n(\|h_\alpha\|_{L^2}^2+\|h_\alpha\|_{L^1}^2\right)\leq  C(T) \sum_{\alpha=1}^n(\|h_\alpha\|_{L^2}^2+H(\bar g_\alpha|\bar f_\alpha)),
\end{equation}
Apart from evolution of $L^2$-distance, the evolution of the relative entropy reads as
$$
\begin{aligned}
\frac{\,\d}{\,\d t}H(\bar g_\alpha|\bar f_\alpha)&=  \frac{\,\d}{\,\d t}\Big(\int_{\R^d}  \bar g_\alpha\log \bar g_\alpha\,\d x-\int_{\R^d} \bar g_\alpha\log  \bar f_\alpha\,\d x\Big)
= \int_{\R^d} \p_t \bar g_\alpha\log \frac{\bar g_\alpha}{\bar f_\alpha}\,\d x -\int_{\R^d} \frac{\bar g_\alpha}{\bar f_\alpha}\p_t\bar f_\alpha\,\d x\\
&=-\sigma_\alpha\int_{\R^d} \bar g_\alpha\Big|\nabla\log\frac{\bar g_\alpha}{\bar f_\alpha}\Big|^2\,\d x+\sum_{\beta=1}^n a_{\alpha\beta}\int_{\R^d} \bar g_\alpha\nabla\log\frac{\bar g_\alpha}{\bar f_{\alpha}}\cdot \big(\nabla V\ast \bar f_{\beta}-\nabla V\ast \bar g_\beta\big)\,\d x.
\end{aligned}
$$
Then Cauchy-Schwarz inequality implies that 
\begin{equation*}
\begin{aligned}
\frac{\,\d}{\,\d t}H(\bar g_\alpha|\bar f_\alpha)\leq  C\sum_{\beta=1}^n \int_{\R^d} \bar g_\alpha \big|\nabla V\ast \bar f_{\beta}-\nabla V\ast \bar g_\beta\big|^2\,\d x,
\end{aligned}
\end{equation*}
where the right-hand side can be bounded as
$$
\begin{aligned}
\int_{\R^d} \bar g_\alpha \big|\nabla V\ast \bar f_{\beta}-\nabla V\ast\bar g_\beta|^2&=\int_{\R^d} \bar g_\alpha \big|\nabla V\ast h_\beta|^2\,\d x\\
&\leq  2\int_{\R^d} \bar g_\alpha\big|\nabla V|_{B_1}\ast h_\beta|^2\,\d x+2\int_{\R^d} \bar g_\alpha\big|\nabla V|_{B_1^c}\ast h_\beta|^2\,\d x\\
&\leq 2\|\bar g_\alpha\|_{L^\infty}\|\nabla V|_{B_1}\|_{L^1}^2\| h_\beta\|_{L^{2}}^2+2\| \bar g_\alpha\|_{L^1}\|\nabla V|_{B_1^c}\|_{L^\infty}^2 \|h_\beta\|^2_{L^1}\\
&\leq C(\|h_\beta\|^2_{L^1}+\|h_\beta\|^2_{L^2}).
\end{aligned}
$$
Therefore, the relative entropy between two weak solutions can be estimated as
\begin{equation*}
 \frac{\,\d}{\,\d t}H(\bar g_\alpha|\bar f_\alpha)\leq C(T)\sum_{\beta=1}^n\big(\|h_\beta\|_{L^2}^2+H(\bar g_\beta|\bar f_\beta)\big).
\end{equation*}
Recall \eqref{L2L1unique}, which implies
$$
\frac{\,\d }{\,\d t}\sum_{\alpha=1}^n\|h_\alpha\|_{L^2}^2\leq C(T)  \sum_{\alpha=1}^n\left(\|h_\alpha\|_{L^2}^2+H(\bar g_\alpha|\bar f_\alpha)\right).
$$
we finally arrive at 
$$
\frac{\,\d }{\,\d t}\sum_{\alpha=1}^n\left(\|h_\alpha\|_{L^2}^2+H(\bar g_\alpha|\bar f_\alpha)\right)\leq C(T)\sum_{\alpha=1}^n\left(\|h_\alpha\|_{L^2}^2+H(\bar g_\alpha|\bar f_\alpha)\right).
$$
The same initial data $\bar f_\alpha^0=\bar g_\alpha^0$ and Gronwall's inequality imply that for any $\alpha=1, \ldots, n$
$$
\|h_\alpha\|_{L^2}^2+H(\bar g_\alpha|\bar f_\alpha)=0.
$$
We deduce that for almost any $t\in[0,T]$, $\bar f_\alpha(t)=\bar g_\alpha(t)$ in $L^1\cap L^2(\R^d)$. Then the uniqueness also holds in the function space $L^\infty(0,T;L^1\cap L^\infty)\cap L^2(0,T;H^1)$.

Therefore, there exists a unique weak solution of \eqref{aggregation-diffusion} and we conclude  Theorem \ref{well-possedness}.

\bigskip

\section*{Acknowledgments}
The research of JAC was supported by the Advanced Grant Nonlocal\--CPD (Non\-local PDEs for Complex Particle Dynamics: Phase Transitions, Patterns and Synchronization) of the European Research Council Executive Agency (ERC) under the European Union’s Horizon 2020 research and innovation programme (grant agreement No. 883363).
JAC was also partially supported by EPSRC grant number EP/V051121/1.
JAC was also partially supported by the “Maria de Maeztu” Excellence Unit IMAG, reference CEX2020-001105-M, funded by MCIN/AEI/10.13039/501100011033/.

\appendix

\section{Appendix}
\subsection{Proof of Lemma \ref{subadditivity}}\label{appendix sub}
\begin{proof}
By the variational definition of the relative entropy, which can be seen, for example,  in \cite{dembo2009large}, it holds
$$
\begin{aligned}
H(f_{N,\eps}|\t f_{N,\eps})&=\int_{\R^{dnN}}f_{N,\eps}\log\frac{f_{N,\eps}}{\t f_{N,\eps}}\,\,\d \mathbf{X}
\\&= \sup_{\Phi\in C_b(\R^{dnN})} \Big\{\int_{\R^{dnN}}f_{N,\eps}\Phi\,\d \mathbf{X}-\log \int_{\R^{dnN}}\t f_{N,\eps}\exp\Phi\,\d \mathbf{X}\Big\}.    
\end{aligned}
$$
We take
$$
\begin{aligned}
\Phi(\mathbf{X})&:=\Phi(x_{1,1},\ldots,x_{1,N},\ldots,x_{n,1},\ldots,x_{n,N})\\&= \sum_{p=1}^{[\frac{N}{\max_\alpha K_\alpha}]}\vphi\big(\underbrace{x_{1,(p-1)K_1+1},\ldots,x_{1,(p-1)K_1+K_1}}_{K_1},\ldots,\underbrace{x_{n,(p-1)K_n+1},\ldots,x_{n,(p-1)K_n+K_n}}_{K_n}\big)
\end{aligned}
$$
where $[a]$ is the integer part of $a\in\R$ and $\vphi\in C_b(\R^{|\mathbf{K}|d})$ with $|\mathbf{K}|=\sum_\alpha K_\alpha$. Notice that there is no redundant variables on the right-hand side, which leads to the equality by Fubini’s theorem that
$$
\begin{aligned}
\int_{\R^{dnN}}\t f_{N,\eps}\exp\Phi=\prod_{p=1}^{[\frac{N}{\max_\alpha K_\alpha}]}\int_{\R^{d|K|}}\left(\prod_{\alpha=1}^n\t f_{\alpha,\eps}^{\otimes K_\alpha}\right)\exp\vphi\,\left(\prod_{\alpha=1}^n\,\d x_{\alpha,(p-1)K_\alpha+1}\cdots\,\d x_{\alpha,(p-1)K_\alpha+K_\alpha}\right).
\end{aligned}
$$  Then we get
\begin{equation}\label{ineq N maxk}
\begin{aligned}
 H(f_{N,\eps}|\t f_{N,\eps})\geq &\int_{\R^{dnN}}f_{N,\eps}\Phi\,\d \mathbf{X}-\log \int_{\R^{dnN}}\t f_{N,\eps}\exp\Phi\,\d \mathbf{X}\\
=&\, \left[\frac{N}{\max_\alpha K_\alpha}\right]\int_{\R^{d|K|}}f_{N,\eps}^{(\mathbf{K})} \vphi-\left[\frac{N}{\max_\alpha K_\alpha}\right]\log \int_{\R^{d|\mathbf{K}|}}\prod_{\alpha=1}^n\t f_{\alpha,\eps}^{\otimes K_\alpha}\exp\vphi,
\end{aligned}
\end{equation}
where the last inequality is due to particles are fully exchangeable within each species. Since the relative entropy between $f_{N,\eps}^{(\mathbf{K})}$ and $\prod_{\alpha=1}^n\t f_{\alpha,\eps}^{\otimes K_\alpha}$ can be represented such as
$$
H(f_{N,\eps}^{(\mathbf{K})}|\prod_{\alpha=1}^n\t f_{\alpha,\eps}^{\otimes K_\alpha})=\sup_{\vphi\in C_b(\R^{d|\mathbf{K}|})}\Big\{\int_{\R^{d|K|}}f_{N,\eps}^{(\mathbf{K})} \vphi-\log \int_{\R^{d|\mathbf{K}|}}\prod_{\alpha=1}^n\t f_{\alpha,\eps}^{\otimes K_\alpha}\exp\vphi\Big\},
$$
the inequality \eqref{subadd ineq} holds by taking supremum of $\vphi$ over $C_b(\R^{d|\mathbf{K}|})$ on the right-hand side of \ref{ineq N maxk}.
\end{proof} 
\subsection{Proof of Lemma \ref{LLN}}\label{appendix LLN}

\begin{proof}
Let 
$$h_{ij}(t)=\psi_\eps^{\alpha,\beta}\big(\t{X}_{\alpha,i}^\eps(t)-\t{X}_{\beta,j}^\eps(t)\big)-\big(\psi_\eps^{\alpha,\beta} * \t{f}_{\beta,\eps}\big)\big(t,\t{X}_{\alpha,i}^\eps(t)\big),$$
where for any $t\in[0,T]$, it holds
$$\E\big[h_{ij}(t)\big|\t X_{\alpha,i}^\eps(t)\big]=\E\big[\psi_\eps^{\alpha,\beta}\big(\t{X}_{\alpha,i}^\eps(t)-\t{X}_{\beta,j}^\eps(t)\big)-\big(\psi_\eps^{\alpha,\beta} * \t{f}_{\beta,\eps}\big)\big(t,\t{X}_{\alpha,i}^\eps(t)\big)\big|\t X_{\alpha,i}^\eps(t)\big]=0.$$
We omit the dependence of species $\alpha,\beta$ on $h_{ij}$ to simplify the notation. 
By Markov's inequality, for any $m\in \N$, we get
$$
\begin{aligned}
\P\left(\mathcal{A}_{\theta, \Psi_\eps}^N(t)\right)=&\, \P \left(\bigcup_{\alpha, \beta=1}^n\bigcup_{i=1}^N\Big\{\omega \in \Omega:\Big|\frac{1}{N} \sum_{j=1}^N h_{ij} \Big|>N^{-\theta}\Big\}\right)
\leq \,n^2\max_{\alpha,\beta}\sum_{i=1}^N N^{2m\theta} \E \left[\Big|\frac{1}{N} \sum_{j=1}^N h_{ij} \Big|^{2m}\right]\\
=&\,n^2\max_{\alpha,\beta} \sum_{i=1}^N N^{2m(\theta-1)}\E \left[\left(\sum_{j,k=1}^Nh_{ij} \cdot h_{ik} \right)^m\right].
\end{aligned}
$$
By definition of the Frobenius norm, we can reduce to the case when $h_{ij} $ is a scalar function, hence
$$
\E \left[\Big(\sum_{j,k=1}^Nh_{ij}  h_{ik} \Big)^m\right]=\sum_{\Gamma}\E \big[h_{ij_1} h_{ij_2} \cdots h_{ij_{2m}} \big], 
$$
where $\Gamma =\{ (j_{1},j_{2},\ldots j_{2m}): j_{1},j_{2},\ldots j_{2m}=1,\ldots,N\}$.
Then, if there exists some specific index $j'$ that only appears once,
$$
\begin{aligned}
\E \left[h_{ij_1} h_{ij_2} \cdots h_{ij_{2m}} \right]=&\,\E \left[\E \big[h_{ij_1} h_{ij_2} \cdots h_{ij_{2m}} \big|\t X_{\alpha,i}^\eps\big]\right]\\=&\, \E \left[\E \big[h_{ij_1} \cdots \hat{h}_{ij'} \cdots h_{ij_{2m}} \big|\t X_{\alpha,i}^\eps\big]\E\big[h_{ij'} \big|\t X_{\alpha,i}^\eps\big]\right]=0,
\end{aligned}
$$
where $\hat{h}_{ij'} $ denotes that the corresponding  $h_{ij'} $ is vacant in the multiplication. Next, we consider $\Gamma'\subset\Gamma$ as the subset such that any index $j_k$ ($k=1,\ldots,2m$) appears at least twice. Then, the cardinality of the set 
$\Gamma'$ is of order $N^m$ up to a constant only depending on $m$. Therefore we have the bound
$$
\sup_{t\in[0,T]}\sum_{\Gamma}\E \big[h_{ij_1} h_{ij_2} \cdots h_{ij_{2m}} \big]= \sup_{t\in[0,T]}\sum_{\Gamma'}\E \big[h_{ij_1} h_{ij_2} \cdots h_{ij_{2m}} \big]\leq C(m,\alpha,\beta)\|\psi_\eps^{\alpha,\beta} \|^{2m}_{L^\infty}N^m,
$$
and it implies the desired estimate
$$
\sup_{t\in[0,T]}\P\left(\mathcal{A}_{\theta, \Psi_\eps }^N(t)\right) \leq n^2\max_{\alpha,\beta}C(m,\alpha,\beta)\left\|\psi_\eps^{\alpha,\beta} \right\|_{L^{\infty}}^{2 m} N^{ m(2\theta-1)+1}.
$$
\end{proof}

\subsection{Compactness argument}\label{compactness}

The approach to prove the existence of weak solutions comes from \cite{FP08}, and similar procedure has been used in \cite{CG24} for cross diffusion systems. Due to our present setting seeking for solutions in $
\bar f_\alpha\in L^\infty(0,T,L^1\cap L^\infty(\R^d))\cap L^2(0,T,H^1(\R^d))$, some estimates are slightly different for instance in the proof of Lemma \ref{relative compactness} and the passing to the limit on the compactness argument in \ref{pass to the limit}.
For completeness, we detail the proof here. Let $\mathcal{M}(\R^d)$ be the space of probability measure equipped with the following metric which measures the weak convergence in $\mathcal{M}(\R^{d})$, for $\mu_1,\mu_2\in \mathcal{M}(\R^d)$
\begin{equation}\label{BL distance}
d_{\M}(\mu_1,\mu_2):=
\sup_{\vphi\in BL}\Big|\int_{\R^d}\vphi(x)\mu_1(\d x)-\int_{\R^d}\vphi(x)\mu_2(\d x)\Big|,
\end{equation}
where the function space $BL$ denotes the set of functions whose both $L^\infty$-norm and Lipschitz constant are less than 1. As a consequence of Lemma \ref{uniformeps}, we have the relative compactness lemma as follows.
\begin{lemma}\label{relative compactness}
 For each $\alpha$, the sequence $(\t f_{\alpha,\eps})_{\eps>0}$ is relatively compact in $C([0,T],\M(\R^d))$.   
\end{lemma}
\begin{proof}
We will prove the lemma by the Ascoli-Arzel\`{a} theorem.

{\bf Step 1.- Relative compactness.} For each $\alpha$, we will show that there is a relatively compact subset $\mathcal{K}_\alpha\subset\mathcal{M}(\R^d)$ such that for any $t\in[0,T]$ and $\eps>0$, $\t f_{\alpha,\eps}(t)\in \mathcal{K}_\alpha$. 
Because a subset of $\mathcal{M}(\R^d)$ is relatively compact if and only if it is tight, it is equivalent to show that for any $t\in[0,T]$ and $\eta>0$, there exists a compact set $K_\alpha\subset \R^d$ with $\t f_{\alpha,\eps}(K_\alpha)\geq 1-\eta$ for all $\eps>0$.  
Recall the intermediate process $\t X_{\alpha}^{\eps}(t)$ defined by \eqref{tildeX} with $\law(\t X_{\alpha}^{\eps}(t))=\t f_{\alpha,\eps}(t)$ satisfies the SDE
$$
\,\d \t X_{\alpha}^{\eps}(t) = -\sum_{\beta=1}^{n}  \nabla V_{\alpha\beta}^\eps\ast\t f_{\beta,\eps}(t,\t X_{\alpha}^{\eps}(t)) \,\d t + \sqrt{2\sigma_{\alpha}} \,\d \t B_{\alpha}(t).
$$
Then $\t f_{\alpha,\eps}(K_\alpha)\geq 1-\eta$ is equivalent to $\P\big(\t X_{\alpha}^{\eps}(t)\in K_\alpha^c\big)\leq \eta$. We can take $K_\alpha$ as a closed ball with radius $R>0$, then the probability of $\t X_{\alpha}^{\eps}$ being outside the closed ball can be estimated as
\begin{equation}\label{ineq prob X}
\begin{aligned}
\P\Big(\big|\t X_{\alpha}^{\eps}(t)\big|>R\Big)    = \,&\P\Big(\Big|Z_{\alpha}-\int_0^t \sum_{\beta=1}^{n}  \nabla V_{\alpha\beta}^\eps\ast\t f_{\beta,\eps}(\tau,\t X_{\alpha}^{\eps}(\tau)) \,\d \tau +\sqrt{2\sigma_\alpha }\t B_\alpha(t)\Big|>R\Big)\\
\leq  & \, \P\Big(\Big|\int_0^t \sum_{\beta=1}^{n}  \nabla V_{\alpha\beta}^\eps\ast\t f_{\beta,\eps}(\tau,\t X_{\alpha}^{\eps}(\tau)) \,\d \tau \Big|>\frac{R}{3}\Big)+\P\Big(\sqrt{2\sigma_\alpha }\big|\t B_\alpha(t)\big|>\frac{R}{3}\Big)\\ &\,+\P\Big(\big|Z_{\alpha}\big|>\frac{R}{3}\Big),
\end{aligned}
\end{equation}
where the second and the third terms are independent of $\eps$ and tend to $0$ as $R\to \infty$. For the first term, we deduce by Markov's inequality
\begin{align*}
\P\Big(\Big|\int_0^t \sum_{\beta=1}^{n}  \nabla V_{\alpha\beta}^\eps\ast\t f_{\beta,\eps}(\tau,\t X_{\alpha}^{\eps}(\tau)) \,\d \tau \Big|>\frac{R}{3}\Big)   \leq  &\,\sum_{\beta=1}^{n}\P\Big(\int_0^t  \big| \nabla V_{\alpha\beta}^\eps\ast\t f_{\beta,\eps}(\tau,\t X_{\alpha}^{\eps}(\tau))\big| \,\d \tau >\frac{R}{3n}\Big) \\&\,  \leq \frac{3n}{R}\sum_{\beta=1}^{n} \E\Big[\int_0^t  \big| \nabla V_{\alpha\beta}^\eps\ast\t f_{\beta,\eps}(\tau,\t X_{\alpha}^{\eps}(\tau))\big| \,\d \tau \Big] \\
 &\,  \leq\frac{3n}{R}\sum_{\beta=1}^{n}|a_{\alpha\beta}| \int_0^t \int_{\R^d} \big| \nabla V\ast\chi^\eps\ast\t f_{\beta,\eps}\big|\t f_{\alpha,\eps} \,\d x\,\d \tau . 
\end{align*}
Since the following estimate holds by the uniform boundedness given in Lemma  \ref{uniformeps}
$$
\begin{aligned}
\sup_\eps\int_0^t \int_{\R^d} \big| \nabla V\ast\chi^\eps\ast\t f_{\beta,\eps}\big|\t f_{\alpha,\eps} \,\d x&\,\d \tau 
\leq \,t \|\t f_{\alpha,\eps}\|_{L^\infty(0,t;L^1)}\|\nabla V\ast\t f_{\beta,\eps}\|_{L^\infty(0,t;L^\infty)}\leq  C(T),
\end{aligned}
$$
the first term on the right-hand side of \eqref{ineq prob X} converges to 0 uniformly in $\eps$ by sending $R$ to $\infty$. 

\medskip

{\bf Step 2.- Equicontinuity.} We will show that the sequence $(\t f_{\alpha,\eps})_{\eps>0}$ is equicontinuous, i.e., for every $\eta>0$ there exists $\delta>0$ and  $t,\tau\in[0,T]$ such that for all $ |t-\tau|<\delta$, it holds
    $$
\sup_\eps d_{\M}(\t f_{\alpha,\eps}(\tau),\t f_{\alpha,\eps}(t))<\eta.
$$ For $\tau,t\in [0,T]$, the distance between $\t f_{\alpha,\eps}(\tau)$ and $\t f_{\alpha,\eps}(t)$ can be estimated as
$$
d_{\M}\big(\t f_{\alpha,\eps}(t), \t f_{\alpha,\eps}(\tau)\big)   \leq \E\Big[\big|\t X_{\alpha}^{\eps}(t)-\t X_{\alpha}^{\eps}(\tau)\big|\Big], 
$$
due to the duality \eqref{BL distance},
and we have
\begin{align*}
\E\Big[\big|\t X_{\alpha}^{\eps}(t)-\t X_{\alpha}^{\eps}(\tau)\big|\Big]
= &\, \E\Big[\Big|\sum_{\beta=1}^{n} \int^\tau_t  \nabla V_{\alpha\beta}^\eps\ast\t f_{\beta,\eps}(\tau,\t X_{\alpha}^{\eps}(\tau)) \,\d r+\sqrt{2\sigma_{\alpha} }\t B_\alpha(t)-\sqrt{2\sigma_{\alpha} }\t B_\alpha(\tau)\Big|\Big] \\
  \leq &\, \sum_{\beta=1}^n|a_{\alpha\beta}|\E\Big[\Big|\int_\tau^t \nabla V\ast\chi^\eps\ast\t f_{\beta,\eps}(r,\t X_{\alpha}^{\eps}(r)) \,\d r\Big|\Big]+\sqrt{2\sigma_\alpha }\E\big[\big|\t B_\alpha(t)-\t B_\alpha(\tau)\big|\big] \\
  \leq&\, C |t-\tau|\|\t f_{\alpha,\eps}\|_{L^\infty(0,T;L^\infty)}\sum_{\beta=1}^n\big(\|\t f_{\beta,\eps}\|_{L^\infty(0,T;L^\infty)}+\|\t f_{\beta,\eps}\|_{L^\infty(0,T;L^1)}\big)\\&\,+\sqrt{2\sigma_\alpha }|t-\tau|^{1 / 2}.
\end{align*}
 Collecting the previous estimates we obtain $d_{\M}\big(\t f_{\alpha,\eps}(t), \t f_{\alpha,\eps}(\tau)\big) \leq C(T)(|t-\tau|+|t-\tau|^{1 / 2})$, where the constant $C$ is independent of $\eps$ by Lemma \ref{uniformeps}, which finishes the proof of Lemma \ref{relative compactness}. 
\end{proof}

Lemma \ref{relative compactness} implies that for each species $\alpha$ the sequence $(\t f_{\alpha,\eps})_{\eps>0}$ has a convergent subsequence, which is still denoted by $(\t f_{\alpha,\eps})_{\eps>0}$. Let the limit be $\bar f_\alpha \in C\left([0, T], \M\left(\R^d\right)\right)$, i.e.,
\begin{equation*}
\t f_{\alpha,\eps}\rightarrow \bar f_\alpha
\quad\text{ 
in}\quad C([0, T], \mathcal{M}(\R^d))\text{ as } \eps\to0.
\end{equation*}
By Banach-Alaoglu theorem due to the bounds in Lemma \ref{uniformeps}, up to a subsequence, $(\t f_{\alpha,\eps})_{\eps>0}$ weakly* converges to some limit in $L^\infty(0,T, L^\infty(\R^d))$, which is the dual space of $L^1(0,T,L^1(\R^d))$ (for example, see \cite[Proposition 1.38]{roubivcek2013nonlinear}). Then the limit satisfies $\bar f_\alpha\in L^\infty(0,T, L^\infty(\R^d))$; and due to the same reason, $\bar f_\alpha\in L^2(0,T, H^1(\R^d))$. Together with $
\bar f_\alpha\in C(0,T,\M(\R^d))$, we conclude  
$$
\bar f_\alpha\in L^\infty(0,T,L^1\cap L^\infty(\R^d))\cap L^2(0,T,H^1(\R^d)).
$$ 
To pass to the limit, we need to show the following strong convergence result. 
\begin{lemma}\label{L1convergence}
For each species $\alpha$, $\t f_{\alpha,\eps}$ converges strongly to  $\bar f_{\alpha}$ in $L^1(0,T;L^1(\R^d))$ when $\eps\to 0$.
\end{lemma}
\begin{proof}
By Lemma \ref{relative compactness} and Prokhorov's theorem,  $(\t f_{\alpha,\eps})_{\eps>0}$ are uniformly tight for each $\alpha=1,\ldots,n$. 
Then for any $\delta>0$, there exists a ball $B_{R_\delta}$ such that for any $\eps>0$
$$
\int_0^T\int_{B_{R_\delta}^c} \t f_{\alpha,\eps}(t, x) \,\d x \,\d t \leq \frac{\delta}{4}\quad \text{and}\quad \int_0^T\int_{B_{R_\delta}^c} \bar f_{\alpha}(t, x) \,\d x \,\d t \leq \frac{\delta}{4},
$$
where the last inequality holds since $\bar f_{\alpha}(t) \in L^1(\R^d)$. Moreover, we have
$$
\begin{aligned}
\int_0^T\int_{\R^d}|\t f_{\alpha,\eps}-\bar f_{\alpha}|\,\d x\,\d t=  &\, \int_0^T\int_{B_{R_\delta}^c}|\t f_{\alpha,\eps}-\bar f_{\alpha}|\,\d x \,\d t+\int_0^T\int_{B_{R_\delta}}|\t f_{\alpha,\eps}-\bar f_{\alpha}|\,\d x\,\d t\\\leq &\,  \frac{\delta}{2}+\int_0^T\int_{B_{R_\delta}}|\t f_{\alpha,\eps}-\bar f_{\alpha}|\,\d x\,\d t.    
\end{aligned}
$$
To prove  Lemma \ref{L1convergence}, by the arbitrariness of $\delta>0$, it is sufficient to show $\t f_{\alpha,\eps}$ converges strongly to  $\bar f_{\alpha}$ in $L^1(0,T;L^1(B_R))$ for any bounded ball $B_R$. 

Taking some $p>0$ large enough such that $\M(B_R)$ is continuously embedded into $H^{-p}(B_R)$, for any $\delta>0$ and some constant $C(\delta)$, we have the following inequality
\begin{equation}\label{triplets}
 \|\mu\|_{L^1(B_R)}\leq \delta \|\mu\|_{H^1(B_R)}+C(\delta) \|\mu\|_{H^{-p}(B_R)},   
\end{equation}
which follows from similar arguments in \cite{FP08,CG24}.
Therefore, for $\eps,\eps'>0$ letting $\mu=\t f_{\alpha,\eps}-\t f_{\alpha,\eps'}$ and plugging it into the inequality \eqref{triplets}, we get
$$
\begin{aligned}
\big\|\t f_{\alpha,\eps}-\t f_{\alpha,\eps'}\big\|_{L^1\left(0, T; \times B_R\right)}  &\,  \leq \delta\big\|\t f_{\alpha,\eps}-\t f_{\alpha,\eps'}\big\|_{L^1\left(0, T; H^1(B_R)\right)}  +C(\delta)\big\|\t f_{\alpha,\eps}-\t f_{\alpha,\eps'}\big\|_{L^1\left(0, T; H^{-p}\left(B_R\right)\right)} \\
 &\,  \leq \delta\big(\big\|\t f_{\alpha,\eps}\big\|_{L^1\left(0, T; H^1(B_R)\right)}+\big\|\t f_{\alpha,\eps'}\big\|_{L^1\left(0, T; H^1(B_R)\right)}\big)\\&\, \quad+C(\delta)\big\|\t f_{\alpha,\eps}-\t f_{\alpha,\eps'}\big\|_{L^1\left(0, T; H^{-p}\left(B_R\right)\right)} \\
 &\,  \leq C(T)\delta+C(\delta)\int_0^Td_{\M}(\t f_{\alpha,\eps},\t f_{\alpha,\eps'})\,\d t,
\end{aligned}
$$
where we use inequality $$\big\|\t f_{\alpha,\eps}\big\|_{L^1\left(0, T; H^1(B_R)\right)}\leq \sqrt{T}\big\|\t f_{\alpha,\eps}\big\|_{L^2\left(0, T; H^1(B_R)\right)},
$$ with $\big\|\t f_{\alpha,\eps}\big\|_{L^2\left(0, T; H^1(B_R)\right)}$ bounded uniformly in $\eps$. 
By the convergence of the sequence $(\t f_{\alpha,\eps})_{\eps>0}$ in $C([0,T],\mathcal{M}(\R^d))$, we have
$$
\limsup_{\eps,\eps^{\prime}\to 0}\big\|\t f_{\alpha,\eps}-\t f_{\alpha,\eps'}\big\|_{L^1\left([0, T] \times B_R\right)}\leq C(T)\delta,
$$
which implies that $(\t f_{\alpha,\eps})_{\eps>0}$ is a Cauchy sequence in $L^1\left([0, T] \times B_R\right)$ by the arbitrariness of $\delta$. Hence, $\t f_{\alpha,\eps}$ converges strongly to  $\bar f_{\alpha}$ in $L^1(0,T;L^1(B_R))$. It concludes Lemma \ref{L1convergence}. 
\end{proof}

It remains to pass to the limit. 
\begin{lemma}\label{pass to the limit}
For any $\vphi\in C_b^2(\R^d)$ and any $T>0$, $\bar f_\alpha \in L^\infty(0,T,L^1\cap L^\infty(\R^d))\cap L^2(0,T,H^1(\R^d))$ satisfies \eqref{aggregation-diffusion} in the following sense,
\begin{align*}
\int_{\R^d}\bar f_{\alpha}(T)\vphi\,\d x  = &\, \int_{\R^d}\bar f_{\alpha}^0\vphi\,\d x+\sigma_{\alpha}\int_0^T\int_{\R^d}\Delta\vphi  \bar f_{\alpha} \,\d x\,\d t \\&- \sum_{\beta =1}^n a_{\alpha\beta} \int_0^T\int_{\R^d}\nabla\vphi\cdot\bar f_{\alpha} \big( \nabla V \ast \bar f_{\beta}\big)\,\d x\,\d t,\quad 
\alpha=1,2,\ldots,n.   
\end{align*}
Consequently, $\bar f=(\bar f_1,\ldots,\bar f_n)$ is a weak solution of the aggregation-diffusion system \eqref{aggregation-diffusion}.
\end{lemma}
\begin{proof}

Due to \eqref{intermediate}, for every $\vphi\in C_b^2(\R^d)$, the identity holds
\begin{align*}
\int_{\R^d}\t f_{\alpha,\eps}(T,x)\vphi(x)\,\d x  = &\, \int_{\R^d}\bar f_{\alpha}^0(x)\vphi(x)\,\d x+\sigma_{\alpha}\int_0^T\int_{\R^d}\Delta\vphi  \t f_{\alpha,\eps} \,\d x\,\d t \\&- \sum_{\beta =1}^n\int_0^T\int_{\R^d}\nabla\vphi\cdot\t f_{\alpha,\eps} \big( \nabla V_{\alpha\beta}^\eps \ast \t f_{\beta,\eps}\big)\,\d x\,\d t.   
\end{align*}
The first line of above equation converges to its corresponding limit. For the nonlinear term, it holds
$$
\begin{aligned}
&\bigg|\sum_{\beta =1}^n a_{\alpha\beta}\int_0^T\int_{\R^d} \nabla\vphi\cdot \bar f_\alpha\nabla V\ast \bar f_\beta\,\d x\,\d \tau -\sum_{\beta=1}^n\int_0^T\int_{\R^d} \nabla\vphi\cdot \t f_{\alpha,\eps}\nabla V_{\alpha\beta}^\eps\ast \t f_{\beta,\eps}\,\d x\,\d \tau \bigg|\\
\leq & C \sum_{\beta =1}^n\Big|\int_0^T\int_{\R^d} \nabla\vphi\cdot \bar f_\alpha\nabla V\ast \bar f_\beta\,\d x\,\d \tau -\int_0^T\int_{\R^d} \nabla\vphi\cdot \t f_{\alpha,\eps}\nabla V\ast\chi^\eps\ast \t f_{\beta,\eps}\,\d x\,\d \tau \Big|\\
\leq & C\|\nabla \vphi\|_{L^\infty}\sum_{\beta =1}^n\int_0^T\int_{\R^d} \Big|\bar f_\alpha\nabla V\ast \bar f_\beta-\t f_{\alpha,\eps}\nabla V\ast\chi^\eps\ast \t f_{\beta,\eps}\Big|\,\d x\,\d \tau ,
\end{aligned}
$$
where the last integral can be separated as follows:
$$
\begin{aligned}
\int_0^T\!\!\int_{\R^d} \Big|\bar f_\alpha\nabla V\ast \bar f_\beta-\t f_{\alpha,\eps}\nabla V\ast\chi^\eps\ast \t f_{\beta,\eps}\Big|\,\d x\,\d \tau    
&\leq \int_0^T\!\!\int_{\R^d} \Big|\bar f_\alpha\nabla V\ast \bar f_\beta-\t f_{\alpha,\eps}\nabla V\ast \bar f_\beta\Big|\,\d x\,\d \tau \\ &+\int_0^T\!\!\int_{\R^d} \Big|\t f_{\alpha,\eps}\nabla V\ast \bar f_\beta-\t f_{\alpha,\eps}\nabla V\ast \t f_{\beta,\eps}\Big|\,\d x\,\d \tau \\&+\int_0^T\!\!\int_{\R^d} \Big|\t f_{\alpha,\eps}\nabla V\ast \t f_{\beta,\eps}-\t f_{\alpha,\eps}\nabla V\ast\chi^\eps\ast \t f_{\beta,\eps}\Big|\,\d x\,\d \tau \\
=:&I+II+III.
\end{aligned}
$$
The first term $I$ can be controlled as
$$
I\leq C \big(\|\bar f_{\beta}\|_{L^\infty(0,T;L^\infty)}+\|\bar f_{\beta}\|_{L^\infty(0,T;L^1)}\big)\int_0^T\int_{\R^d} \Big|\bar f_\alpha-\t f_{\alpha,\eps}\Big|\,\d x\,\d \tau ,
$$
which converges to $0$ since $\bar f_\beta\in L^\infty(0,T;L^1\cap L^\infty(\R^d))$ and the strong convergence (Lemma \ref{L1convergence}). The second term can be estimated by
$$
\begin{aligned}
II\leq &\,\int_0^T\int_{\R^d}\t f_{\alpha,\eps} \Big|\nabla V|_{B_1}\ast (\bar f_\beta-\t f_{\beta,\eps})\Big|\,\d x\,\d \tau +\int_0^T\int_{\R^d}\t f_{\alpha,\eps} \Big|\nabla V|_{B_1^c}\ast (\bar f_\beta-\t f_{\beta,\eps})\Big|\,\d x\,\d \tau  \\
\leq  &\, C\|\t f_{\alpha,\eps}\|_{L^\infty(0,T;L^\infty)}\int_0^T\int_{\R^d} \big|\bar f_\beta-\t f_{\beta,\eps}\big|\,\d x\,\d \tau +C\|\t f_{\alpha,\eps}\|_{L^1(0,T;L^1)}\int_0^T\int_{\R^d} \big|\bar f_\beta-\t f_{\beta,\eps}\big|\,\d x\,\d \tau \\
\leq  &\,C\big(\|\t f_{\alpha,\eps}\|_{L^\infty(0,T;L^\infty)}+T\|\t f_{\alpha,\eps}\|_{L^\infty(0,T;L^1)}\big)\int_0^T\int_{\R^d} \big|\bar f_\beta-\t f_{\beta,\eps}\big|\,\d x\,\d \tau, 
\end{aligned}
$$
which converges to $0$ by Lemma \ref{L1convergence}. The third term satisfies
$$
\begin{aligned}
III\leq & \,\int_0^T\big\|\t f_{\alpha,\eps}\big\|_{L^1}\big\|(\nabla V|_{B_1}\ast\chi^\eps-\nabla V|_{B_1})\ast \t f_{\beta,\eps}\big\|_{L^\infty}\,\d \tau\\
&+ \int_0^T\big\|\t f_{\alpha,\eps}\big\|_{L^1}\big\|(\nabla V|_{B_1^c}\ast\chi^\eps-\nabla V|_{B_1^c})\ast \t f_{\beta,\eps}\big\|_{L^\infty}\,\d \tau\\\leq & \int_0^T\big\|\t f_{\alpha,\eps}\big\|_{L^1}\big\|\nabla V|_{B_1}\ast\chi^\eps-\nabla V|_{B_1}\big\|_{L^p}\big\| \t f_{\beta,\eps}\big\|_{L^{p'}}\,\d \tau\\
&+ \int_0^T\big\|\t f_{\alpha,\eps}\big\|_{L^1}\big\|\nabla V|_{B_1^c}\ast\chi^\eps-\nabla V|_{B_1^c}\big\|_{L^q}\big\| \t f_{\beta,\eps}\big\|_{L^{q'}}\,\d \tau, 
\end{aligned}
$$
where $1<p<\frac{d}{s+1}<q<\infty$ and $p',q'$ are their conjugate numbers.
Then $III$ converges to $0$ since $\nabla V|_{B_1}\in L^p(\R^d)$, $\nabla V|_{B_1^c}\in L^q (\R^d)$ and hence $$\lim_{\eps\to0}\big\|\nabla V|_{B_1}\ast\chi^\eps-\nabla V|_{B_1}\big\|_{L^p}=0\quad \text{and}\quad \lim_{\eps\to0}\big\|\nabla V|_{B_1^c}\ast\chi^\eps-\nabla V|_{B_1^c}\big\|_{L^q}=0.$$
    
\end{proof}

\bigskip

\bibliographystyle{abbrv}
\bibliography{ref}

%\bibliography{reference-moderate}
%\bibliographystyle{abbrv}
%\bibliographystyle{alpha}

\end{document}